\documentclass[11pt]{article}
\usepackage[a4paper,left=20mm,top=20mm,right=20mm,bottom=30mm]{geometry}

\usepackage{amsmath,amsopn,amsthm,amssymb}

\usepackage{authblk}
\usepackage[numbers,sort&compress]{natbib}
\usepackage{graphicx}
\usepackage{mathptmx} 
\usepackage[mathscr]{euscript}
\usepackage{mathtools}
\usepackage{amssymb}
\usepackage{wasysym}
\usepackage{float}
\usepackage{enumitem}
\usepackage{xspace}
\usepackage{url}
\usepackage[margin=30pt,font=small,labelfont=bf]{caption}

\usepackage[
 colorlinks=true,
urlcolor=black,
linkcolor=blue
]{hyperref}

\usepackage{algorithmicx}
\usepackage{algorithm} 
\usepackage[noend]{algpseudocode}
\algrenewcommand\algorithmicrequire{\textbf{Input:}}
\algrenewcommand\algorithmicensure{\textbf{Output:}}

\newtheorem{theorem}{Theorem}[section]
\newtheorem{lemma}[theorem]{Lemma} 
\newtheorem{corollary}[theorem]{Corollary}
\newtheorem{proposition}[theorem]{Proposition}

\newtheorem{definition}[theorem]{Definition}
\newtheorem{observation}[theorem]{Observation}

\newtheorem*{remark}{Remark}
\newtheorem*{summary}{Summary}

\makeatletter
\def\moverlay{\mathpalette\mov@rlay}
\def\mov@rlay#1#2{\leavevmode\vtop{%
   \baselineskip\z@skip \lineskiplimit-\maxdimen
   \ialign{\hfil$\m@th#1##$\hfil\cr#2\crcr}}}
\newcommand{\charfusion}[3][\mathord]{
    #1{\ifx#1\mathop\vphantom{#2}\fi
        \mathpalette\mov@rlay{#2\cr#3}
      }
    \ifx#1\mathop\expandafter\displaylimits\fi}
\makeatother

\newcommand{\cupdot}{\charfusion[\mathbin]{\cup}{\cdot}}

\usepackage{tabularx}
\makeatletter
\newcommand{\multiline}[1]{%
  \begin{tabularx}{\dimexpr\linewidth-\ALG@thistlm}[t]{@{}X@{}}
    #1
  \end{tabularx}
}
\makeatother

\DeclareMathOperator{\join}{\otimes}
\DeclareMathOperator{\union}{\cupdot}%{+}
\DeclareMathOperator{\med}{med}
\DeclareMathOperator{\child}{child}
\newcommand{\lca}{\ensuremath{\operatorname{lca}}}
\newcommand{\diam}{\ensuremath{\operatorname{diam}}}
\newcommand{\dist}{\ensuremath{\operatorname{dist}}}
\newcommand{\parent}{\ensuremath{\operatorname{par}}}

\newcommand{\MD}{\ensuremath{\mathbb{M}}}
\newcommand{\MDstrong}{\ensuremath{\mathbb{M}_{\mathrm{str}}}}
\newcommand{\Mmax}{\ensuremath{\mathbb{M}_{\max}}}
\newcommand{\MDT}{\ensuremath{\mathscr{T}}}

\newcommand{\PrimeCat}{\ensuremath{\vcenter{\hbox{\includegraphics[scale=0.01]{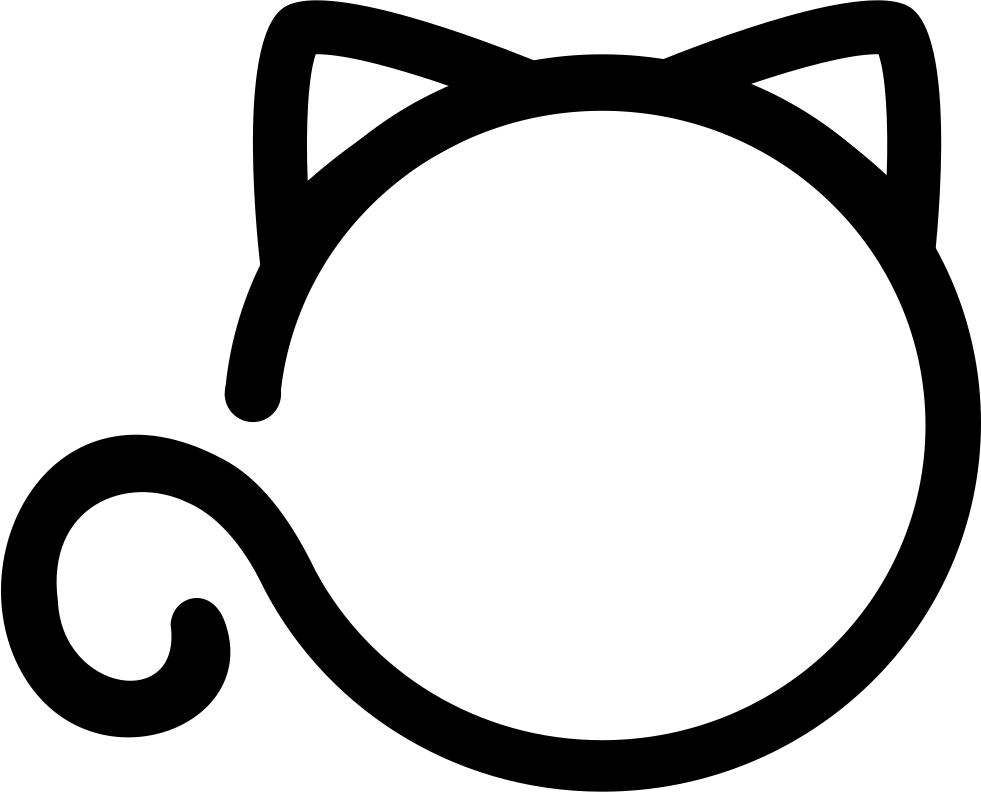}}}^{\textit{prime}}}}
\newcommand{\PolarCat}{\ensuremath{\vcenter{\hbox{\includegraphics[scale=0.01]{./cat.png}}}}}

\providecommand{\keywords}[1]{\textbf{\textit{Keywords: }} #1}

  \title{From Modular Decomposition Trees to Level-1 Networks: Pseudo-Cographs, Polar-Cats and Prime Polar-Cats}

\author[1,*]{Marc Hellmuth} 
\author[2]{Guillaume E. Scholz} 

\affil[1]{Department of Mathematics, Faculty of Science,
  Stockholm University, SE-10691 Stockholm, Sweden 
  \newline \texttt{marc.hellmuth@math.su.se}}

\affil[2]{Bioinformatics Group, Department of Computer Science \&
    Interdisciplinary Center for Bioinformatics, Universit{\"a}t Leipzig,
    H{\"a}rtelstra{\ss}e~16--18, D-04107 Leipzig, Germany.}

\affil[*]{corresponding author}

\date{\ }

\begin{document}
\sloppy

\maketitle

\abstract{ 
The modular decomposition of a graph $G$ is a natural construction to capture
key features of $G$ in terms of a labeled tree $(T,t)$ whose vertices are
labeled as ``series'' ($1$), ``parallel'' ($0$) or ``prime''. However, full
information of $G$ is provided by its modular decomposition tree $(T,t)$ only,
if $G$ does not contain prime modules. In this case, $(T,t)$ explains $G$, i.e.,
$\{x,y\}\in E(G)$ if and only if the lowest common ancestor $\lca_T(x,y)$ of $x$
and $y$ has label ``$1$''. This information, however, gets lost whenever $(T,t)$
contains vertices with label ``prime''. In this contribution, we aim at
replacing ``prime'' vertices in $(T,t)$ by simple 0/1-labeled cycles, which
leads to the concept of rooted labeled level-1 networks $(N,t)$.

We characterize graphs that can be explained by such level-1 networks
$(N,t)$, which generalizes the concept of graphs that can be explained by
labeled trees, that is, cographs. We provide three novel graph classes:
\emph{polar-cats} are a proper subclass of \emph{pseudo-cographs} which forms a
proper subclass of \emph{prime polar-cats}. In particular, every cograph is a
pseudo-cograph and prime polar-cats are precisely those graphs that can be
explained by a labeled level-1 network. The class of prime polar-cats is defined
in terms of the modular decomposition of graphs and the property that all prime
modules ``induce'' polar-cats. We provide a plethora of structural results and
characterizations for graphs of these new classes.

In particular, Polar-cats are precisely those graphs that can be explained
by an elementary level-1 network $(N,t)$, i.e., $(N,t)$ contains exactly one
cycle $C$ that is rooted at the root $\rho_N$ of $N$ and where $\rho_N$ has exactly
two children while every vertex distinct from $\rho_N$ has a unique child that
is not located in $C$. Pseudo-cographs are less restrictive and those
graphs that can be explained by particular level-1 networks $(N,t)$ that contain
at most one cycle. These findings, eventually, help us to characterize the class of all
graphs that can be explained by labeled level-1 networks, namely prime
polar-cats. Moreover, we show under which conditions there is a unique
least-resolved labeled level-1 network that explains a given graph. In
addition, we provide linear-time algorithms to recognize all these types
of graphs and to construct level-1 networks to explain them.}
\smallskip

\noindent 

\sloppy

\section{Introduction}

Cographs are among the best-studied graph classes. They are characterized by the
absence of induced simple paths $P_4$ on four vertices
\cite{Corneil:81,Sumner74,Seinsche:74} and can be represented by a unique rooted
tree $(T,t)$, called cotree, whose leaf set is $V(G)$ and whose non-leaf
vertices $v$ obtain a binary label $t(v)\in \{0,1\}$. Every cograph $G$ is
\emph{explained} by its cotree, i.e., $\{x,y\}\in E(G)$ if and only if the
lowest common ancestor $\lca_T(x,y)$ has label ``$1$''. Several linear-time
algorithms to recognize cographs and to compute the underlying cotrees have been
established see e.g.\
\cite{Corneil:85,HP:05,bretscher2008simple,DAMIAND200199,GIOAN2012708}. Cographs
are of particular interest in computer science because many combinatorial
optimization problems that are NP-complete for arbitrary graphs become
polynomial-time solvable on cographs \cite{Corneil:85,BLS:99,Gao:13}. Recent
advances in mathematical biology have shown that cographs are intimately linked
to pairwise relationships between genes
\cite{HHH+13,HW:16b,Hellmuth:15a,HSW:16,Geiss:18a,Hellmuth:18a,Schaller:21f,lafond2015orthology,LDEM:16,Lafond2014,geiss2020best}.
By way of example, the orthology relation, a key term in evolutionary biology,
collects all pairs of genes whose last common ancestor in their evolutionary
history was a speciation event (or, equivalently has label ``$1$'') and thus
forms a cograph.

In many applications, however, graphs $G$ usually violate the property of being a
cograph and thus, cannot be represented by a tree $(T,t)$ whose binary labeling $t$
uniquely determines the structure of $G$. This motivates the investigation of
generalizations. While trees are an excellent model of many evolutionary systems or
hierarchical data in general, they are often approximations and sometimes networks
are a better model of reality, in particular, in phylogenomics
\cite{huson2011survey}. In the case of distance-based phylogenetics, this naturally
connects with theory of split-decomposable metrics \cite{Bandelt:92} and their
natural representations, the Buneman graphs \cite{Dress:97,Huber:02}. The latter form
a subclass of median graphs and it has been shown by Bruckmann et al.\ \cite{BHS:21}
that every graph can be explained by a rooted binary-labeled median graph.

In this contribution, we are interested in the characterization of graphs that can be
explained by rooted level-1 networks $N$ with binary labeling $t$. We present the
formal definition for this type of network in the next section but, essentially,
level-1 networks are directed acyclic graphs with a single root and in which any two
``undirected cycles'' \cite{HS18} (also known as blocks \cite{GBP12} or 
galls
\cite{gambette2017challenge}) are edge disjoint. We provide two generalizations of
cographs: the class of \emph{pseudo-cographs} and the class $\PrimeCat$ of \emph{prime 
polar-cats}.
Pseudo-cographs have the appealing property that they ``behave'' nearly like a
cograph up to a single vertex. Note, however, that the existence of a single vertex
$v$ such that $G-v$ is a cograph is not enough to define pseudo-cographs $G$. Every
pseudo-cograph can be explained by a labeled level-1 network. We distinguish here
between weak and strong level-1 networks; a property that is defined in terms of the
cycles a level-1 network contains. In weak networks, all cycles can be locally
replaced by trees and one obtains a phylogenetic tree that still explains the same
graph. In other words, labeled weak networks do not contain more information than a
labeled tree. In strong networks, however, none of the cycles can be removed.
Pseudo-cographs are restrictive in the sense that they can only be explained by
strong level-1 networks that contain at most one cycle. To circumvent this
restriction, we provide the class $\PrimeCat$ and  the class $\PolarCat$ of
\emph{polar-cats}. Polar-cats are
particular pseudo-cographs that can be explained by so-called elementary networks and
$\PrimeCat$ is precisely the class of graphs that can be explained by a labeled
level-1 network. Graphs in $\PrimeCat$ are defined by means of the modular
decomposition; a general technique to decompose a graph into smaller building blocks
\cite{Gallai:67,EHMS:94,MS:89,McConnell:95,HP:10}. Roughly speaking a module in a
graph is a subset $M$ of vertices which share the same neighborhoods outside $M$. A
module $M$ of $G$ is prime if the subgraph $G[M]$ of $G$ induced by the vertices in
$M$ and its complement $\overline{G[M]}$ are both connected. A graph $G$ is contained
in $\PrimeCat$ if the underlying ``quotient of'' $G[M]$ is a polar-cat for all
non-trivial prime modules $M$ . We build upon the ideas provided by Bruckmann et al.\
\cite{BHS:21} and show that prime modules in the modular decomposition tree can be
replaced by elementary networks to obtain a level-1 network that explains a given
graph $G\in \PrimeCat$. In addition, we provide linear-time algorithms for the
recognition of pseudo-cographs, polar-cats as well as graphs in $\PrimeCat$ and for
the construction of level-1 networks to explain them. 

After introducing the notation and some preliminary results, we give
an overview of the main concepts and results in Section \ref{sec:mainR}.

\section{Preliminaries}
\label{sec:prelim}

\paragraph{\bf Miscellaneous}

Let $\Pi(A)=\{A_1,\dots, A_k\}$ be a collection of subsets of a non-empty set $A$.
Then, $\Pi(A)$ is a \emph{quasi-partition} of $A$, if $\cup_{i=1}^k A_i=A$ and
$A_i\cap A_j=\emptyset$ for all distinct $A_i,A_j\in \Pi(A)$. If all elements $A_i\in
\Pi(A)$ are non-empty, then $\Pi(A)$ is a \emph{partition}. A \emph{bipartition} is a
partition consisting of two elements.

\paragraph{\bf Graphs}
We consider graphs $G=(V,E)$ with  vertex set $V(G)\coloneqq V\neq \emptyset$ and 
edge set
$E(G)\coloneqq E$. A graph $G$ is \emph{undirected} if $E$ is a subset of the set of
two-element subsets of $V$ and $G$ is \emph{directed} if $E\subseteq V\times V$.
Thus, edges $e\in E$ in an undirected graph $G$ are of the form $e=\{x,y\}$ and in
directed graphs of the form $e=(x,y)$ with $x,y\in V$ being distinct. An undirected
graph is \emph{connected} if, for every two vertices $u,v\in V$, there is a path
connecting $u$ and $v$. A directed graph is \emph{connected} if its underlying
undirected graph is connected. Moreover, (directed) paths connecting two vertices $x$
and $y$ are also called \emph{$xy$-paths}. 
We write $H \subseteq G$ if $H$ is a subgraph of $G$
and $G[W]$ for the subgraph in $G$ that is induced by some subset $W \subseteq V$.
Moreover $G-v$ denotes the induced subgraph $G[V\setminus \{v\}]$. 
\emph{From here on, we will call an undirected graph simply \emph{graph}.}

Let $G=(V,E)$ be a graph. The set $N_G(v)\coloneqq \{w \in V \mid \{v,w\} \in E\}$
denotes the set of neighbors of $v$ in $G$. The complement $\overline{G}$ of
$G=(V,E)$ has vertex $V$ and edge set $\{\{x,y\}\mid \{x,y\}\notin E, x,y\in V
\text{distinct}\}$. The join of two vertex-disjoint graphs $H=(V,E)$ and $H'=(V',E')$
is defined by $H \join H'\coloneqq (V\cup V',E\cup E' \cup \{\{x,y\}\mid x\in V, y\in
V'\})$, whereas their disjoint union is given by $H \union H'\coloneqq (V\cupdot V',
E\cupdot E')$. The intersection of two not necessarily vertex-disjoint graphs $H$ and
$H'$ is $H \cap H'\coloneqq(V\cap V', E\cap E')$. A graph $G$ admits a
\emph{graph-bipartition}, if there is a bipartition $\{V_1,V_2\}$ of $V$ such that
$\{u,v\}\in E$ implies $u\in V_i$ and $v\in V_j$, $\{i,j\}=\{1,2\}$.
We write $G\simeq H$, if the graphs $G$ and $H$ are isomorphic.

A \emph{star} is a graph that has a single vertex (called \emph{center}) that is
adjacent to all other vertices and no further edges. Complete graphs $G=(V,E)$ are
denoted by $K_{|V|}$ and a path with $n$ vertices by $P_n$. We often write, for
simplicity, $x_1-x_2-\cdots - x_{n}$ for a path $P_n$ with vertices $x_1,\dots,x_n$
and edges $\{x_i,x_{i+1}\}$, $1\leq i<n$. The distance $\dist_G(x,y)$ of two vertices
$x,y\in V$ is the length, i.e. the number of edges, of a shortest $xy$-path in $G$.
Moreover, the \emph{diameter $\diam(G)$} of $G$ is $\max_{x,y\in V}(\dist_G(x,y))$.

\paragraph{\bf Networks and Trees}
We consider rooted (not necessarily binary) phylogenetic trees and networks, called
trees or networks for short (see e.g.\
\cite{steel2016phylogeny,huson_rupp_scornavacca_2010,HSS:22cluster} for an overview of 
phylogenetic
trees and networks). To be more precise, a {\em network} $N=(V,E)$ on $X$ is a
directed acyclic graph (DAG) such that either 
\begin{enumerate}
	\item[(N0)] $V=X = \{x\}$ and, thus, $E=\emptyset$.
\end{enumerate}
or $N$ satisfies  the following three properties 
\begin{enumerate}[noitemsep]
	\item[(N1)] There is a unique \emph{root} $\rho_N$ with indegree 0 and outdegree at least 2; and\smallskip
	\item[(N2)] $x\in X$ if and only if $x$ has outdegree 0 and indegree 1; and \smallskip
	\item[(N3)] Every vertex $v\in V^0 \coloneqq V \setminus X$ with $v\neq \rho_N$ has 
		\begin{enumerate}[noitemsep]
			\item indegree 1 and 	outdegree at least 2 (\emph{tree-vertex}) or
			\item indegree 2 and 	outdegree at least 1 (\emph{hybrid-vertex}).
		\end{enumerate}	
\end{enumerate}
If a network $N=(V,E)$ does not contain hybrid-vertices, then $N$ is a \emph{tree}.
The set $L(N)=X$ is the \emph{leaf set} of $N$ and $V^0$ the set of \emph{inner} 
vertices.
The set $E^0\subset E$ denotes the set of \emph{inner} edges, i.e., edges consisting
of inner vertices only. A graph is \emph{biconnected} if it contains no vertex whose
removal disconnects the graph. A \emph{biconnected component} of a graph is a maximal
biconnected subgraph. If such a biconnected component is not a single vertex or an edge, 
then it is called \emph{non-trivial}.
Non-trivial biconnected components are also known as \emph{cycle} \cite{HS18}
(or \emph{block} \cite{GBP12} or \emph{gall} \cite{gambette2017challenge}). 
A network $N$ is a \emph{level-k} network, if each biconnected
component $C$ of $N$ contains at most $k$ vertices of indegree 2 in $C$, i.e.,  
hybrid-vertices of $N$ whose parents belong to $C$. \cite{CJSS04}. 
Thus, every tree is a level-$k$ network, $k\geq 0$. 
An example of a level-1 network is shown in Fig.\ \ref{fig:L1Exmpl}.

A \emph{caterpillar} is a tree $T$ such that each inner vertex has exactly two
children and the subgraph induced by the inner vertices is a path with the root
$\rho_T$ at one end of this path. A subset $\{x,y\}\subseteq X$ is a \emph{cherry} if
the two leaves $x$ and $y$ are adjacent to the same vertex in $T$. In this case, we
also say that $x$, resp., $y$ is \emph{part of a cherry}. Note that a caterpillar
on $X$ with $|X|\geq 2$ contains precisely two vertices that are part of a cherry.

Let $N=(V,E)$ be a network on $X$. A vertex $u\in V$ is called an \emph{ancestor} of
$v\in V$ and $v$ a \emph{descendant} of $u$, in symbols $v \preceq_N u$, if
there is a directed path (possibly reduced to a single vertex) in $N$ from $u$ to
$v$. We write $v \prec_N u$ if $v \preceq_N u$ and $u\neq v$. If $u \preceq_N v$ or
$v \preceq_N u$ then $u$ and $v$ are \emph{comparable} and otherwise,
\emph{incomparable}. Moreover, if $(u,v)\in E$, vertex $\parent(v)\coloneqq u$ is
called a \emph{parent of $v$} and vertex $v$ is a \emph{child of $u$}. The set of
children of a vertex $v$ in $N$ is denoted by $\child_N(v)$. Note that in a network
$N$ the root is an ancestor of all vertices in $N$. For a non-empty subset of leaves
$A\subseteq V$ of $N$, we define $\lca_N(A)$, or a \emph{lowest common ancestor of
$A$}, to be a $\preceq_N$-minimal vertex of $N$ that is an ancestor of every vertex
in $A$. Note that in trees and level-1 networks the $\lca_N$ is uniquely determined
\cite{HS18,HSS:22cluster}. For simplicity we write $\lca_N(x,y)$ instead of 
$\lca_N(\{x,y\})$. 
For a vertex $v$ of $N$, the \emph{subnetwork $N(v)$ of $N$ rooted at $v$}, is the
network obtained from the subgraph $N[W]$ induced by the vertices in $W\coloneqq
\{w\in V(N)\mid w\preceq_N v\}$ and by suppression of $w$ if it has indegree 0
and outdegree 1 in $N[W]$ or hybrid-vertices of $N$ that have in- and outdegree 1
in $N[W]$. Clearly, $N(v)$ remains a level-1 network in case that $N$ is a level-1
network.

By \cite[Obs.\ 8.1]{HSS:22cluster}, every non-trivial biconnected component in a   
	level-1 network $N$ as defined here forms an undirected cycle in $N$. 
	Therefore, we call by slight abuse of
notation, a non-trivial biconnected component of a level-1 network $N$ a
\emph{cycle of $N$}.
In \cite{Gusfield:03}, galled trees were introduced as phylogenetic networks in which 
all cycles are vertex disjoint. Here we consider a more general version, where cycles are 
allowed to share a cutvertex, see also \cite{HSS:22cluster} for more details.
We remark that a 
cycle $C$ of a level-1 network $N$ is composed of two
directed paths $P^1$ and $P^2$ in $N$ with the same start-vertex $\rho_C$ and
end-vertex $\eta_C$, and whose vertices distinct from $\rho_C$ and $\eta_C$ are
pairwise distinct. These two paths are also called \emph{sides of $C$}. The
\emph{length} of $C$ is the number of vertices it contains. A cycle $C\subseteq N$ is
\emph{weak} if either (i) $P^1$ or $P^2$ consist of $\rho_C$ and $\eta_C$ only or
(ii) both $P^1$ and $P^2$ contain only one vertex distinct from $\rho_C$ and
$\eta_C$. A cycle $C\subseteq N$ that is not weak is called \emph{strong}. A network
$N$ is \emph{weak} if all cycles of $N$ are weak. Networks that do not contain weak
cycles are \emph{strong}. Since trees do not contain cycles, they are trivially both,
weak and strong networks.
By definition, different cycles in level-1 networks may share a common vertex but
never edges and if a hybrid-vertex $\eta_C$ of some cycle $C$ belongs to a further
cycle $C'$, then $\eta_C=\rho_{C'}$ must hold. 

A level-1 network $N$ on $X$ is \emph{elementary} if it contains a single cycle $C$ of 
length $|X|+1$ with
root $\rho_N\in V(C)$, a hybrid-vertex $\eta_N\in V(C)$ and additional edges
$\{v_i,x_i\}$ such that every vertex $v_i\in V(C)\setminus \{\rho_N\}$ is adjacent to
unique vertex $x_i\in X$. In this case, we say that $v_i$ is the vertex in $C$ that
\emph{corresponds} to $x_i$. Examples of elementary network $(N,t)$ is provided in
Fig.\ \ref{fig:discriminating}. 

For later reference, we provide
\begin{lemma}\label{lem:lca}
	Let $N$ be a strong level-1 network on $X$ 
	and $v\in V(N)$ be a vertex with outdegree at least 2. 	
	Then, $v=\lca_N(L(N(v)))$ and 
	there are two leaves $x,y\in X$ such that $\lca_N(x,y) = v$.
	Moreover, if $v$ is not the root of a cycle in $N$, then
	$L(N(u))\cap L(N(w))=\emptyset$ for all distinct children
	$u,w$ of $v$ and, in particular, $\lca_N(x,y)=v$ for all
	$x\in L(N(u))$ and $y\in L(N(w))$.
\end{lemma}
\begin{proof}
Suppose that $N$ is a strong level-1 network on $X$ and that $v\in V(N)$ is a
vertex with outdegree at least 2. Put $Y\coloneqq L(N(v))$. Using the
terminology in \cite{HSS:22cluster}, strong level-1 networks are
``semi-regular'' which allows us to use \cite[Cor.\ 4.6]{HSS:22cluster} and to
conclude that $L(N(u))\subsetneq Y$ for all children $u$ of $v$. Hence, $u\neq
\lca_N(Y)$ for all children $u$ of $v$. Moreover, by \cite[Cor.\
7.10]{HSS:22cluster}, there is a unique vertex $w\in V(N)$ such that $Y
\subseteq L(N(w))$ but $Y \not\subseteq L(N(w'))$ for all children $w'$ of $w$
in which case, $w=\lca_N(Y)$. Taking the latter two arguments together yields
$v=\lca_N(Y)$. Moreover, in $N$ it holds, for every non-empty subset $A
\subseteq X$, that there are $x, y \in A$ such that $\lca_N(x, y) = \lca_N(A)$
(cf.\ \cite[Def.\ 6.24 and L.\ 7.11]{HSS:22cluster}). Hence, there are two
leaves $x,y\in X$ such that $\lca_N(x,y) = v$.
	
Let $u,w$ be two children of $v$ and assume that $v$ is not the root of a
cycle in $N$. Clearly, $u$ and $w$ must be $\preceq_N$-incomparable, otherwise
$v,u,w$ would be contained in a common cycle. Moreover, if $L(N(u))\cap
L(N(w))\neq\emptyset$, then \cite[L.\ 7.8]{HSS:22cluster}) implies that $u$ and
$v$ must be contained in a common cycle, which is only possible if $v$ is the
root of this cycle. Hence, $L(N(u))\cap L(N(w))=\emptyset$. Since $L(N(u))\cap
L(N(w))=\emptyset$, we have $\{x,y\}\not\subseteq L(N(u))$,
$\{x,y\}\not\subseteq L(N(w))$ and $\{x,y\}\subseteq L(N(v))$. As argued above,
$v=\lca_N(x,y)$.
\end{proof}

\begin{figure}[t]
		\begin{center}
			\includegraphics[width = .5\textwidth]{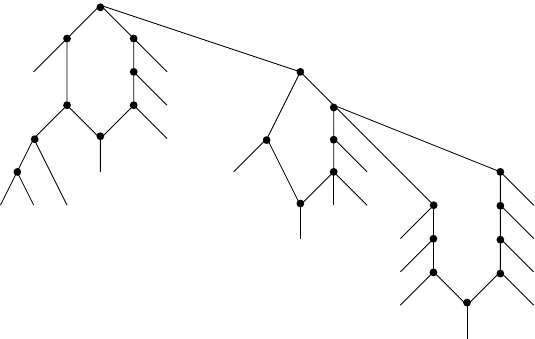}
		\end{center}
		\caption{Shown is a level-1 network. }
		\label{fig:L1Exmpl}
\end{figure}

\paragraph{\bf Labeled networks and Trees}
We consider networks $N=(V,E)$ on $X$ that are equipped with a
\emph{(vertex-)labeling} $t$ i.e., a map $t\colon V\to\{0,1,\odot\}$ such that
$t(x)=\odot$ if and only if $x\in X$. Hence, every inner vertex $v\in V^0$ must have
some \emph{binary} label $t(v)\in \{0,1\}$. The labeling of the leaves $x\in X$ with
$t(x)=\odot$ is just a technical subtlety to cover the special case $V=X=\{x\}$. A
network $N$ with such a ``binary'' labeling $t$ is called \emph{labeled network} and
denoted by $(N,t)$. To recall, for a level-1 network $N$ on $X$, 
 $\lca_N(A)$ is uniquely determined	for all $A\subseteq  X$. This allows us to establish
the following 
\begin{definition}
Given a labeled level-1 network $(N,t)$ on $X$ we denote with $\mathscr{G}(N,t) = 
(X,E)$ the
undirected graph with vertex $X$ and edges $\{x,y\}\in E$ precisely if
$t(\lca_N(x,y))=1$. An undirected graph $G = (X,E)$ is \emph{explained} by $(N,t)$ on
$X$ if $G\simeq \mathscr{G}(N,t)$.
\end{definition}
A network $(N,t)$ is \emph{least-resolved for $G$} if $(N,t)$ explains $G$ and there is
no labeled network $(N',t')$ that explains $G$ where $N'$ is obtained from $N$
by a series of edge contractions (and removal of possible
multi-edges and suppression of vertices with indegree 1 and outdegree 1) and $t'$
is some labeling of $N'$.

A labeling $t$ (or equivalently $(N,t)$) is \emph{discriminating} if $t(u)\neq t(v)$
for all $(u,v)\in E^0$. A labeling $t$ (or equivalently $(N,t)$) is
\emph{quasi-discriminating} if $t(u)\neq t(v)$ for all $(u,v)\in E^0$ with $v$ not
being a hybrid-vertex. The latter implies that every quasi-discriminating tree
$(T,t)$ is discriminating. In networks $(N,t)$, however, $t(u) = t(v)$ is
possible if $v$ is a hybrid-vertex; see Fig.\ \ref{fig:discriminating} for an
illustrative example. 
A cycle $C$ in a labeled network $(N,t)$ is \emph{quasi-discriminating} if
	$t(u)\neq t(v)$ for all $(u,v)\in E(C)$ with $v$ not being the hybrid-vertex
	of $C$.

\begin{figure}[t]
		\begin{center}
			\includegraphics[width = .6\textwidth]{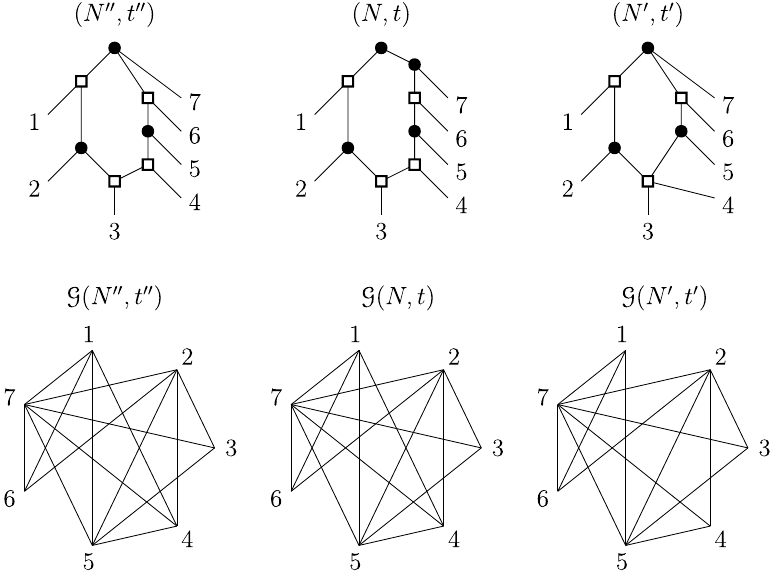}
		\end{center}
		\caption{Shown are three labeled level-1 networks $(N,t)$, $(N',t')$ and
		         $(N'',t'')$ on $X = \{1,\dots,7\}$. While
		         $(N,t)$ is elementary, the networks $(N',t')$ and $(N'',t'')$ are not.
		         Moreover, $(N,t)$ is neither discriminating nor quasi-discriminating;
		         $(N',t')$ is discriminating and thus, quasi-discriminating; and
		         $(N'',t'')$ is quasi-discriminating but not discriminating. Note,
		         $(N',t')$ and $(N'',t'')$ can be obtained from $(N,t)$ by contraction of
		         edges whose vertices have the same label. In this example, we have
		         $\mathscr{G}(N,t) \simeq \mathscr{G}(N'',t'')\not\simeq
		         \mathscr{G}(N',t')$, since $\mathscr{G}(N',t')$ does not contain the
		         edge $\{1,4\}$. This example, in particular, shows that edges that lie
		         on a cycle $C$ in $(N,t)$ and are incident to the hybrid-vertex $\eta_C$
		         can, in general, not be contracted to obtain a network that still
		         explains the same graph $\mathscr{G}(N,t)$ which is one of the reasons
		         to consider quasi-discriminating networks instead of discriminating
		         ones.
			}
		\label{fig:discriminating}
\end{figure}

\begin{remark}
In all drawings that show labeled networks,  
the leaf-label ``$\odot$'' is omitted and inner vertices with label ``$1$'', resp., ``$0$''
are drawn as ``$\bullet$'', resp., ``$\square$''. The label of the hybrid-vertices can often 
be chosen arbitrarily and is sometimes marked as ``$\times$''.
\end{remark}

\paragraph{\bf Cographs}
\emph{Cographs} form a class of undirected graphs that play an important role in
 this contribution. They are defined recursively \cite{Corneil:81}: the
single vertex graph $K_1$ is a cograph and the join $H \join H'$ and disjoint union
$H \union H'$ of two cographs $H$ and $H'$ is a cograph. Many characterizations
of cographs have been established in the last decades from which we summarize
a few of them in the following

\begin{theorem}[\cite{Corneil:81,Sumner74,Seinsche:74}]\label{thm:CharCograph}
Given a graph $G$, the following statements are equivalent:\smallskip
\begin{enumerate}[noitemsep,nolistsep]
	\item $G$ is a cograph.
	\item $G$ does not contain induced $P_4$s.
	\item For every subset $W\subseteq V(G)$ with $|W|>1$ it holds that $G[W]$ is
	      disconnected or $\overline{G[W]}$ is disconnected.
	\item Every connected induced subgraph of $G$ has diameter $\leq 2$.
	\item Every induced subgraph of $G$ is a cograph.
\end{enumerate}
\end{theorem}

It is well-known that a graph $G$ can be explained by a labeled tree $(T,t)$, called
\emph{cotree}, if and only if $G$ is a cograph \cite{Corneil:81}. In particular, we
have 
\begin{theorem}
	For every cograph $G=(V,E)$ there is a unique (up to isomorphism) discriminating
	cotree $(\widehat T,\widehat t)$ such that $G\simeq \mathscr{G}(\widehat T,\widehat
	t)$. In particular, the recognition of cographs $G$ and, in the affirmative case,
	the construction of the discriminating cotree $(\widehat T,\widehat t)$ that
	explains $G$ can be achieved in $O(|V|+|E|)$ time. 
\end{theorem}

A cograph is \emph{(strict) caterpillar-explainable} if its unique discriminating
cotree is a caterpillar.

\paragraph{\bf Modular Decomposition}
The concept of \emph{modular decomposition (MD)} is
defined for arbitrary graphs $G$ and allows us to present the structure of $G$ in the
form of a tree that generalizes the idea of cotrees \cite{Gallai:67}. However, in
general much more information needs to be stored at the inner vertices of this tree
if the original graph has to be recovered. We refer to \cite{HP:10} for an excellent
overview about MDs. 

The MD is based on the notion of modules. These are also known as autonomous sets
\cite{MR-84, Moh:85}, closed sets \cite{Gallai:67}, clans \cite{EHMS:94}, stable
sets, clumps \cite{Blass:78} or externally related sets \cite{HM-79}. A \emph{module}
$M$ is a subset $M\subseteq V$ such that for all $x,y\in M$ it holds that
$N_G(x)\setminus M = N_G(y)\setminus M$. Therefore, the vertices within a given
module $M$ are not distinguishable by the part of their neighborhoods that lie
``outside'' $M$. The singletons and $X$ are the \emph{trivial} modules of $G$ and
thus, the set $\MD(G)$ of all modules of $G$ is not empty. 
A graph $G$ is called \emph{primitive} \cite{ER1:90,ER2:90,ER3:90} (or \emph{prime}
\cite{CH:94,HP:10} or \emph{indecomposable} \cite{CI:98,ILLE:97,ST:93}) if it has at
least four vertices and contains trivial modules only. 
The smallest primitive graph is the path $P_4$ on four vertices. In particular,
every primitive graph $G=(V,E)$ must contain an induced $P_4$ and thus, cannot be a
cograph \cite{Sumner:73}. 

A module $M$ of $G$ is \emph{strong} if $M$ does not \emph{overlap} with any other
module of $G$, that is, $M\cap M' \in \{M, M', \emptyset\}$ for all modules $M'$ of
$G$. In particular, all trivial modules of $G$ are strong. We write
$\MDstrong(G)\subseteq \MD(G)$ for the set of all strong modules of $G$. Every strong
module $M\in \MDstrong(G)$ is of one of the following three types:
\begin{center}
\begin{itemize}[noitemsep]%, nolistsep]
	\item \emph{series}, if $\overline{G[M]}$ is disconnected;\smallskip
	\item \emph{parallel}, if $G[M]$ is disconnected;\smallskip
	\item \emph{prime}, otherwise, i.e., $\overline{G[M]}$ and  $G[M]$ are connected.
\end{itemize}
\end{center}
The set $\MDstrong(G)$ of strong modules is uniquely determined
\cite{HSW:16,EHMS:94}. While there may be exponentially many modules, the cardinality
of the set of strong modules of $G=(X,E)$ is in $O(|X|)$ \cite{EHMS:94}. Since strong 
modules do
not overlap, the set $\MDstrong(G)$ forms a hierarchy and gives rise to a unique tree
representation $\MDT_G$ of $G$, known as the \emph{modular decomposition tree}
(\emph{MDT}) of $G$. The vertices of $\MDT_G$ are (identified with) the elements of
$\MDstrong(G)$. Adjacency in $\MDT_G$ is defined by the maximal proper inclusion
relation, that is, there is an edge $(M',M)$ between $M,M'\in \MDstrong(G)$ if and
only if $M\subsetneq M'$ and there is no $M''\in \MDstrong(G)$ such that $M\subsetneq
M'' \subsetneq M'$. The root of $\MDT_G$ is (identified with) $X$ and every leaf $v$
corresponds to the singleton $\{v\}$, $v\in X$. Hence, $\MDstrong(G) = \{L(\MDT_G(v))
\mid v\in V(\MDT_G)\}$. 

To present the structure of $G$ in the
form of a tree that generalizes the idea of cotrees a labeling
is needed. Since every strong module $M\in \MDstrong(G)$ corresponds to a unique vertex in 
$\MDT_G$ and is either a serial, parallel or prime module, we can establish a labeling 
$\tau_G \colon V(\MDT(G)) \to \{0,1,\mathrm{prime},\odot\}$ defined by setting for all $M\in  \MDstrong(G) = V(\MDT_G)$
\[ \tau_G(M)\coloneqq
				\begin{cases}
             0 & \text{, if } M \text{ is a parallel module}  \\
						 1 & \text{, else if } M \text{ is a  series module}  \\
						 \mathrm{prime} & \text{, else if } M \text{ is a prime module and } |M|>1  \\
						 \odot & \text{, else, i.e., } |M|=1
        \end{cases} \]%        
In this way, we obtain a labeled tree $(\MDT_G,\tau_G)$ that, at least to some
extent, encodes the structure of $G$. In particular, $\tau_G(\lca(x,y))=1$ implies
$\{x,y\}\in E(G)$ and $\tau_G(\lca(x,y))=0$ implies $\{x,y\}\notin E(G)$. The
converse, however, is not satisfied in general, since $\tau_G(\lca(x,y))=p$ is
possible. Note that, if $G$ is a cograph, then $(\MDT_G,\tau_G)$ coincides with the
unique discriminating cotree $(\widehat T, \widehat t)$ of $G$
\cite{BLS:99,Corneil:81,HHH+13}. In other words, $G$ is a cograph if and only it 
its cotree does not contain \emph{prime vertices}, i.e., vertices with label ``$\mathrm{prime}$''. 
Hence, $G$ is a cograph if and only it does
not contain prime modules $M$ of size at least two.
Moreover, observe that for every 
	$M\subseteq V(G)$ with $1<|M|\leq 3$ it always holds that
$G[M]$ or $\overline{G[M]}$ is disconnected. Thus, we obtain 
\begin{observation}\label{obs:M4}
	If $M$ is a prime module of $G$, then  $|M|\geq 4$.  
\end{observation}

The first polynomial time algorithm to compute the modular decomposition is due to
Cowan et al.\ \cite{CJS:72}, and it runs in $O(|V|^4)$. Improvements are due
to Habib and Maurer \cite{HM-79}, who proposed a cubic time algorithm, and to
M{\"u}ller and Spinrad \cite{MS:89}, who designed a quadratic time algorithm. The
first two linear time algorithms appeared independently in 1994 \cite{CH:94, CS94}.
Since then a series of simplified algorithms has been published, some running in
linear time \cite{DGC:01,CS:99,TCHP:08}, and others in almost linear time
\cite{DGC:01,CS:00,HPV:00, habib2004simple}.

\section{Main Ideas and Results}

\label{sec:mainR}
The following type of graphs 
will play a central role in this contribution.
\begin{definition}[Pseudo-Cographs]\label{def:pseudo-cograph}
	A graph $G$ is a \emph{pseudo-cograph} if $|V(G)|\leq 2$ or
	if there are induced subgraphs $G_1,G_2\subseteq G$ and a vertex $v\in V(G)$
 such that 
 \begin{enumerate}[noitemsep]
 	\item[(F1)] $V(G)=V(G_1)\cup V(G_2)$, $V(G_1)\cap V(G_2) = \{v\}$,  $|V(G_1)|>1$ and 
 	$|V(G_2)|>1$; and \smallskip
 	\item[(F2)] $G_1$ and $G_2$ are cographs; and \smallskip
 	\item[(F3)] $G-v$ is either 
 	            the join or the disjoint union of $G_1-v$ and $G_2-v$. 
 \end{enumerate} 
 In this case, we also say that $G$ is a  \emph{$(v,G_1,G_2)$-pseudo-cograph}. 
\end{definition}
We emphasize that $G$ can be a $(v,G_1,G_2)$-pseudo-cograph and a
$(w,G'_1,G'_2)$-pseudo-cograph at the same time, see Section \ref{sec:PsC} for
examples. Further illustrative examples of pseudo-cographs are provided in
Fig.\ \ref{fig:PseudoCograph}. As we shall see in Section \ref{sec:PsC},
none of the conditions (F1), (F2) and (F3) is dispensable. Pseudo-cographs
generalize the class of cographs (cf.\ Lemma \ref{lem:cographPscograph}).
Moreover, the property of being a pseudo-cograph is hereditary (Lemma
\ref{lem:heri}) and pseudo-cographs are closed under complementation (Lemma
\ref{lem:complement}). Every connected induced subgraph $H$ of a pseudo-cograph
$G$ must have diameter less or equal than four (Lemma \ref{lem:diam}) and thus,
$G$ cannot contain induced paths of length larger than $5$ (Cor.\
\ref{cor:LongestPath}). The latter properties, however, do not characterize
pseudo-cographs.

If $G$ is a $(v,G_1,G_2)$-pseudo-cograph, then $G-v$ or $\overline{G-v}$ must be
disconnected and, in particular, $G-v$ must be a cograph (cf.\ Obs.\
\ref{obs:G-v-Cograph}). As shown in Lemma \ref{lem:ConnComp-v}, every connected
component of the disconnected graph in $\{G-v,\overline{G-v}\}$ must be entirely
contained in either $G_1$ or $G_2$. Moreover, if there are two connected
components $H$ and $H'$ such that subgraph induced by $V(H)\cup V(H')\cup\{v\}$
contains an induced $P_4$, then $H$ must be contained in $G_i$ and $H'$ in $G_j$
with $i,j\in \{1,2\}$ being distinct. The latter two results provide the
foundation for further characterizations of pseudo-cographs by means of the
connected components of the disconnected graph in $\{G-v,\overline{G-v}\}$, see
Prop.\ \ref{prop:charPseudo} and Thm.\ \ref{thm:CharPsG}. 

After investigating further properties of pseudo-cographs in Section
\ref{sec:PsC} we provide a quite simple construction of level-1 networks to
explain a $(v,G_1,G_2)$-pseudo-cograph $G$. This construction is based on the
two cotrees $(T_1,t_1)$ and $(T_2,t_2)$ that explain $G_1$ and $G_2$,
respectively (Def.\ \ref{def:prop:PsC-l1N}). Roughly speaking, $(T_1,t_1)$ and
$(T_2,t_2)$ are joined under a common root and the edges $(\parent(v),v)$ of the
vertex $v$ that is common in both graphs $G_1$ and $G_2$ are ``glued
together'', see Fig.\ \ref{fig:gamma} for an illustrative example. Hence, every
pseudo-cograph can be explained by a level-1 network (Prop.\
\ref{prop:PsC-l1N}). In particular, pseudo-cographs are characterized in terms
of level-1 networks $(N,t)$ that contain precisely one cycle that is rooted at
$\rho_N$ and whose hybrid-vertex has a unique leaf-child, see Thm.\
\ref{thm:CharPsC-Network-Cycle}. However, not all level-1 networks that explain
a pseudo-cograph do necessarily satisfy latter property as shown in Fig.\
\ref{fig:gamma}. This, in particular, shows that different non-isomorphic
level-1 networks can explain the same graph. As argued at the end of Section
\ref{sec:PsC}, not all graphs that can be explained by a level-1 network are
pseudo-cographs.

	We collect here the main results that we obtained pseudo-cographs. 
	\begin{summary}[Pseudo-Cographs]
		Pseudo-cographs \dots \smallskip
		\begin{itemize}[noitemsep,nolistsep]
			\item[\dots]  form a superclass of cographs (Lemma~\ref{lem:cographPscograph}).
			\item[\dots]  form a hereditary  graph class (Lemma~\ref{lem:heri}).
			\item[\dots]  are closed under complementation (Lemma~\ref{lem:complement})
			\item[\dots]  are characterized in terms of a bipartition of the set of connected 
			components in $G-v$, resp.,  $\overline{G-v}$ (Thm.~\ref{thm:CharPsG}).
		\end{itemize}\smallskip
		
		\noindent
		Moreover,
		a graph $G$ is a pseudo-cograph if and only if $|V(G)|\leq 2$ or 
		$G$ can be explained by a level-1 network $(N,t)$ that contains precisely
		one cycle $C$ such that $\rho_C = \rho_N$ and $\child_N(\eta_C)=\{x\}$ for some $x\in 
		L(N)$ (Thm.~\ref{thm:CharPsC-Network-Cycle}).
		
		In linear-time, pseudo-cographs $G$ can be recognized and, in the
		affirmative case, a labeled level-1 network that
		explains $G$ can be constructed (Thm.\ \ref{thm:PsC-recognition}).
		
	\end{summary}

To gain a deeper understanding into graphs that can be explained by a level-1
network, we first investigate the structure of level-1 networks in some more detail
in Section \ref{sec:Cog}. In particular, we show in Prop.\ \ref{prop:NhatN-sameGraph}
that for every level-1 network $(N,t)$ there is a quasi-discriminating ``version''
$(\widehat N,\widehat t)$ of $(N,t)$ (obtained from $(N,t)$ by contraction of certain
edges whose endpoint have the same label $t$) such that $\mathscr{G}(N,t) =
\mathscr{G}(\widehat N,\widehat t)$. Moreover, as shown in Lemma
\ref{lem:Weak-Reduce-Cycle}, every level-1 network that contains weak cycles can be
``transformed'' into a strong level-1 network that explains the same graph by
replacing all weak cycles locally by trees. As a consequence, a graph $G$ is a
cograph if and only if $G$ can be explained by weak level-1 networks (Thm.\
\ref{thm:WeakIffCograph}). Different weak level-1 networks that explain a given
cograph are shown in Fig.\ \ref{fig:nonUniqueN}. The latter results allow us to
consider quasi-discriminating strong networks only which simplifies many of the
proofs. 

In Section \ref{sec:PolarCat}, we continue to investigate the structure of graphs
that can be explained by elementary networks which leads to the concept of
polar-cats.
\begin{definition}[Polar-Cat]
Let $G$ be a $(v,G_1,G_2)$-pseudo-cograph with $|V(G)|\geq 3$. Then, $G$ is
\emph{polarizing (w.r.t.\ $G_1$ and $G_2$)} if $G_1$ and $G_2$ are both
connected 
(resp., both
disconnected) and $G-v$ is
the disjoint union (resp., join) of $G_1-v$ and $G_2-v$. Moreover, if $|V(G)|\geq 4$
and $G_1$ and $G_2$ are caterpillar-explainable such that $v$ is part of a cherry in
both caterpillars explaining $G_1$, resp, $G_2$ then $G$ is \emph{cat  (w.r.t.\ $G_1$ 
and $G_2$)}. If $G$ is
both, polarizing and cat, it is a \emph{polar-cat}\footnote{Not to be confused with
their living counterparts: \url{https://marc-hellmuth.github.io/polarCat}}. 

The class of polar-cats is denoted by $\PolarCat$ and comprises all graphs $G$ 
	 for which	there is a vertex $v\in V(G)$ and induced subgraphs 
	$G_1,G_2\subseteq G$ such that 
	$G$ is a $(v,G_1,G_2)$-pseudo-cograph that is polarizing and cat w.r.t.\ $G_1$ and $G_2$.
In this case, we also say that $G$ is a \emph{$(v,G_1,G_2)$-polar-cat}.
\end{definition}

We first provide in Prop.\ \ref{prop:polcat-pc} a characterization of polar-cats
in terms of an ordering of their vertex sets. Moreover, polar-cats are closed
under complementation (Lemma \ref{lem:PolarCat-complement-CC}), are always
connected and never cographs (Cor.\ \ref{cor:CographNotPC}). In particular,
polar-cats are primitive graphs, i.e., they have at least four vertices and
consist of trivial modules only (Lemma \ref{lem:primitive-1} and
\ref{lem:primitive-2}). Polar-cats are precisely those graphs that can be
explained by a strong elementary quasi-discriminating network (Thm.\
\ref{thm:StrongElementary}). As a main results, Thm.\ \ref{thm:CharPolarCatPC}
shows that a graph can be explained by quasi-discriminating elementary network
if and only if it is a caterpillar-explainable cograph on at least two vertices
or a polar-cat. Taking the latter results together, a graph $G$ is a polar-cat
if and only if it is primitive and can be explained by a (strong
quasi-discriminating elementary) level-1 network (see Thm.\ \ref{thm:CharPolCat}
and \ref{thm:CharPCPC}).

	We collect here the main results that we obtained for polar-cats.
	\begin{summary}[Polar-Cats]
		Polar-cats \dots \smallskip
		\begin{itemize}[noitemsep,nolistsep]
			\item[\dots]  do not form a hereditary  graph class.
			\item[\dots]  are closed under complementation 	 
			(Lemma \ref{lem:PolarCat-complement-CC}), always connected and
			never cographs (Cor.\ \ref{cor:CographNotPC}) and, in particular, 
			primitive (Thm. \ref{thm:CharPolCat}).
			\item[\dots] are characterized in terms of a vertex ordering (Prop.\ 		
			\ref{prop:polcat-pc}).
			
		\end{itemize}\smallskip

		\noindent
		Moreover, a graph $G$ is a polar-cat 
		if and only if $G$ can be explained by a strong
		\emph{elementary} quasi-discriminating network (Thm.\ \ref{thm:StrongElementary}). 
		Polar-cats that can be explained by a unique network $(N,t)$ (up to the label of
		hybrid-vertices) are characterized in terms of the size of particular vertex
		sets 	(Prop.\ \ref{prop:uniqueNt}).

		In linear-time, polar-cats $G$ can be recognized and, in the affirmative
		case,  an (elementary) level-1 network that explains $G$ can be constructed
		(Thm.\ \ref{thm:polar-recognition}).
	\end{summary}

So far, we have only investigated the structure of graphs that can be explained
by particular level-1 networks that contain at most one cycle. Nevertheless,
these results will help us to define the class of graphs that can be explained
by general level-1 networks. As explained in much more detail in Section
\ref{sec:general}, for every prime module $M$ in the MDT $(\MDT_G,\tau_G)$ of
$G$ one can compute a quotient graph $G[M]/\Mmax(G[M])$ which is always a
primitive graph. This observation leads to the following \begin{definition} The
class of \emph{prime polar-cats}, in symbols $\PrimeCat$, denotes the set of all
graphs $G$ for which $G[M]/\Mmax(G[M])\in \PolarCat$ for all prime modules $M$
of $G$. \end{definition} Note that $\PrimeCat$ includes the class of cographs,
since they do not contain prime modules at all. Our aim is now to replace every
prime module $M$ in the MDT of $G$ by a suitable choice of labeled networks (in
particular, elementary networks that explain the polar-cat $G[M]/\Mmax(G[M])$)
to obtain labeled level-1 networks that explains $G$ (cf.\ Def.\ \ref{def:pvr}).
The latter is based on the concept of prime-vertex replacement (pvr) graphs as
established in \cite{BHS:21}. These results allow us to show that $G\in
\PrimeCat$ if and only if $G$ can be explained by a level-1 network $(N,t)$, see
Thm.\ \ref{thm:CharprimeCat}. Many examples show that level-1 networks that
explain a given graph $G$ can differ in their topology and their label.
Nevertheless, we show in Prop.~\ref{prop:1-1-prime-strong} that there is a
1-to-1 correspondence between prime modules of $G$ and strong
quasi-discriminating cycles in $N$. We then show which type of edges can or
cannot be contracted such that resulting network still explains the same graph
(cf.\ Lemma \ref{lem:noEdgeContractionsAlongCycles} and
\ref{lem:LRTnoEdgeContractionsAlongCycles}). This eventually allows us to
show in Thm.\ \ref{thm:1-1-prime-strong} that every strong quasi-discriminating
level-1 network $(N,t)$ that explains $G$ can be obtained from some pvr graph
$(N^*, t^*)$ of $(\MDT_G,\tau_G)$ after a (possibly empty) sequence of edge
contractions. Note, in this way, we can derive distinct networks that differ, in
particular, in the cycles that are used to replace prime vertices in the MDT.
Nevertheless, there is a unique least-resolved strong level-1 network $(N,t)$
that explains $G$ provided that every primitive graph $G[M]/\Mmax(G[M])$ is a
``well-proportioned'' polar-cat (cf.\ Def.\ \ref{def:well-propo} and Thm.\
\ref{thm:uniqueNt}).

In Section \ref{sec:algo} we provide linear-time algorithms for the recognition of
pseudo-cographs, polar-cats as well as graphs in $\PrimeCat$ and for the construction
of level-1 networks to explain them.

	We collect here the main results that we obtained for prime polar-cats.
\begin{summary}[Prime Polar-Cats]
A graph $G$ can be explained by a  labeled level-1 network  if and only if $G\in  
\PrimeCat$ (Thm. \ref{thm:CharprimeCat}). Prime polar-cats that can be explained by
a unique least-resolved  level-1 network are characterized by the structure 
of the quotients  $G[M]/\Mmax(G[M])$  of their prime modules $M$ (Thm.\ 
\ref{thm:uniqueNt}).

	In linear-time, it can be verified if a given graph $G$ can
	be explained by a labeled level-1 network and, in the affirmative case, 
	such a network that explains $G$  can
	be constructed (Thm.\ \ref{thm:AlgGeneral}).

\end{summary}

\section{Pseudo-Cographs and Level-1 Networks}
\label{sec:PsC}

In this section we investigate the structural properties of pseudo-cographs. We first
argue that none of the conditions in Def.\ \ref{def:pseudo-cograph} is dispensable.
It is easy to see that (F2) together with (F3) does not imply (F1) in general (just
consider a cograph $G=G_1\star G_2$ with $\star\in \{\union, \join\}$ where
$V(G_1)\cap V(G_2) =\emptyset$). Now assume that only (F1) and (F2) holds and
consider the graph $G$ that consists of two disjoint copies $P = 1-2-3-4$ and
$P'=1'-2'-3'-4'$ of an induced $P_4$. Now we can choose $G_1$ and $G_2$ in an
arbitrary way as long as $G_1$ and $G_2$ are cographs that additionally satisfy (F1).
In fact, putting $G_1=G[\{2,1,1'\}]$ and $G_2=G[\{2,3,4,2',3',4'\}]$ shows that such
a choice with $v=2$ is possible. Now consider the graph $G-v$. Both vertices $1$ and
$1'$ are contained $G_1-v$. However, vertex $1$ is not adjacent to any vertex in
$G_2-v$ while vertex $1'$ is adjacent to vertex $2'$ in $G_2-v$. Hence, $G-v$ can
neither be the  join nor the disjoint union of $G_1-v$ and $G_2-v$ which implies
that (F3) is violated. To see that (F1) and (F3) do not imply (F2) in general,
consider the graph $G\simeq P_7 = 1-2-\cdots-7$. To satisfy (F3), we can choose $v\in
\{2,3,\dots,6\}$. For every choice of $v$, $G-v$ is the disjoint union of the two
paths. To satisfy (F3) and (F1) and, in particular, $|V(G_1)|, |V(G_2)|>1$, $G_1-v$ must
consist of one path and $G_2-v$ of the other. But then at least one of $G_1$ and
$G_2$ contains an induced $P_4$ and thus, is not a cograph. Hence (F2) is violated.

Consider now a $(v,G_1,G_2)$-pseudo-cograph $G$. In this case $G_1-v$ and $G_2-v$ are
cographs, since $G_1$ and $G_2$ are cographs and the property of being a cograph is
hereditary by Thm.\ \ref{thm:CharCograph}. This implies that $G-v$ must be a cograph
as well, since it is the join or disjoint union of the two cographs $G_1-v$ and
$G_2-v$. Furthermore, Thm.\ \ref{thm:CharCograph} implies that either $G-v$ or its
complement $\overline{G-v}$ is disconnected. Now consider a vertex $v$ of $G$ that is
contained in all induced $P_4$s of $G$. Then, $G-v$ is a cograph. Thm.\
\ref{thm:CharCograph} implies that $G-v$ or $\overline{G-v}$ is disconnected. The
converse, however, is not satisfied in general. To see this, consider the graph that
consists of the disjoint union of two induced $P_4$s. In this case, for every $v\in
V(G)$, the graph $G-v$ is disconnected, but $v$ is not located on every $P_4$. For
later reference, we summarize the latter discussion in the following:
\begin{observation}\label{obs:G-v-Cograph}
	If $G$ is graph that contains a vertex $v$ that is contained in all induced $P_4$s
	of $G$, then $G-v$ or $\overline{G-v}$ is disconnected. The converse of the latter
	statement is not satisfied in general as $G-v$ is not necessarily a cograph.
	Moreover, if $G$ is a \emph{$(v,G_1,G_2)$-pseudo-cograph}, then $G-v$ is a cograph
	and either $G-v$ or its complement $\overline{G-v}$ is disconnected.
\end{observation}

Pseudo-cographs generalize the class of cographs.
\begin{lemma}\label{lem:cographPscograph}
Every cograph is a pseudo-cograph. 
\end{lemma}
\begin{proof}
	Let $G$ be a cograph. If $|V(G)|\leq 2$ we are done. Thus, assume that $V(G)\geq
	3$. Hence, $G$ can be written as $G = H\star H'$, $\star\in \{\join, \union\}$ of
	two cographs $H$ and $H'$ with $|V(H)|, |V(H')|\geq 1$ and $|V(H)| + |V(H')|\geq
	3$. W.l.o.g. assume that $|V(H)|\geq |V(H')|$ and thus, $|V(H)|\geq 2$. Let $v\in
	V(H)$ and put $G_1 = H$ and $G_2 = G[V(H')\cup \{v\}]$. By construction (F1) is
	satisfied for $G_1$ and $G_2$. Since $G_1$ and $G_2$ are induced subgraphs of $G$,
	Thm.\ \ref{thm:CharCograph} implies that the graphs $G_1$ and $G_2$ are cographs
	and hence, (F2) is satisfied. By construction, $G-v = (H\star H')-v = (H\star
	(G_2-v))-v = (H-v)\star (G_2-v) = (G_1-v)\star (G_2-v)$ and therefore, (F3) is
	satisfied. In summary, $G$ is a pseudo-cograph.
\end{proof}
Note, however, that not every pseudo-cograph $G$ is a cograph. By way of example, a
path $P=1-2-3-4$ on four vertices is not a cograph. Nevertheless, $G$ is a
pseudo-cograph. In fact, for every choice of $v\in \{1,2,3,4\}$ there are induced
subgraph $G_1$ and $G_2$ such that $G$ is a $(v,G_1,G_2)$-pseudo-cograph. By way of
example, $G$ is an $(1,G_1,G_2)$-pseudo-cograph for $G_1=G[\{1,3\}]$ and
$G_2=G[\{1,2,4\}]$ as well as a $(2,G_1,G_2)$-pseudo-cograph for $G_1 =G[\{1,2\}]$
and $G_2 =G[\{2,3,4\}]$. In the first case $G-1=(G_1-1)\join (G_2-1)$ and in the
second case $G-2=(G_1-2)\union (G_2-2)$. By symmetry for $v\in\{3,4\}$ there are
induced subgraphs $G_1$ and $G_2$ of $G$ such that $G$ is a
$(v,G_1,G_2)$-pseudo-cograph.

\begin{figure}[t]
		\begin{center}
			\includegraphics[scale = .8]{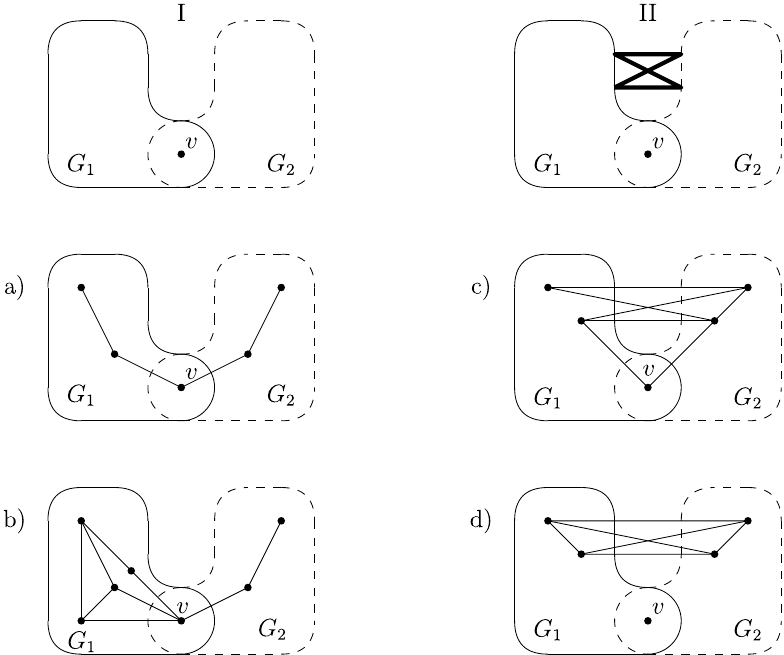}
		\end{center}
		\caption{\emph{Top:} the generic structure of pseudo-cographs where $G-v$ is
		                     either the disjoint union (I) or the join (II) of $G_1-v$ and
		                     $G_2-v$. 
							\emph{Panel a):} A polar-cat and thus, a pseudo-cograph. Note that
							                 polar-cats are always connected and primitive and
							                 thus, cannot be cographs (cf.\ Cor.\
							                 \ref{cor:CographNotPC}).
							\emph{Panel b):} A polarizing pseudo-cograph $G$ that is not a cograph.
							                 Here, $G_1$ is not caterpillar-explainable and thus,
							                 $G$ is not cat. 
						  \emph{Panel c):} A cat pseudo-cograph $G$ that is not polarizing, since
						                   $G_1$ is disconnected while $G_2$ is connected.
							\emph{Panel d)} A polarizing pseudo-cograph $G$ that is a cograph.
							                Although $G_1$ and $G_2$ are caterpillar-explainable,
							                vertex $v$ is not part of cherry in their respective
							                cotrees. Thus, $G$ is not a cat.}
		\label{fig:PseudoCograph}
\end{figure}

As the next result shows, the property of being a pseudo-cograph is
hereditary.
\begin{lemma}\label{lem:heri}
	A graph $G$ is a pseudo-cograph if and only if every induced subgraph
	of $G$ is a pseudo-cograph
\end{lemma}
\begin{proof}
	If every induced subgraph	of $G$ is a pseudo-cograph, then the fact that
	$G$ is an induced subgraph of $G$ implies that $G$ must be a pseudo-cograph.
	
	Now let $G$ be a $(v,G_1,G_2)$-pseudo-cograph. For $|V(G)|\leq 2$, every induced
	subgraph of $G$ contains one or two vertices and is, by definition, a
	pseudo-cograph. Thus, assume that $V(G)\geq 3$. Since any induced subgraph of a
	graph can be obtained by removing vertices one by one, it is sufficient to show
	that the removal of a single vertex from a pseudo-cograph yields a pseudo-cograph.
	Let $x\in V(G)$ be chosen arbitrarily and consider the graph $G-x$. If $x = v$,
	then by Obs.\ \ref{obs:G-v-Cograph}, $G-x$ is a cograph and Lemma
	\ref{lem:cographPscograph} implies that $G-x$ is a pseudo-cograph. 
	
	Hence, assume that $x\neq v$. By (F1), $x$ must be contained in either $V(G_1)$ or
	$V(G_2)$. W.l.o.g.\ assume that $x\in V(G_1)$. If $V(G_1) = \{x,v\}$ it follows
	that $G-x = G_2$ and, by (F2), $G-x$ is a cograph and thus, a pseudo-cograph.
	Assume that $V(G_1)\setminus \{x,v\}\neq \emptyset$. In this case, $G_1-x$ remains
	a cograph with at least two vertices. In particular, it holds that
	$V(G-x)=V(G_1-x)\cup V(G_2)$, $V(G_1-x)\cap V(G_2) = \{v\}$ and $|V(G_1-x)|,
	|V(G_2)|>1$. Thus (F1) is satisfied for $G-x$. Since $G_1-x$ is a cograph and $G_2$
	remained unchanged, (F2) is satisfied for $G-x$. Since $G-v$ satisfies $G-v =
	(G_1-v)\star (G_2-v)$, $\star\in \{\join, \union\}$ and since $V(G_1)\setminus
	\{x,v\}\neq \emptyset$, it is straightforward to show that $((G-x)-v=
	((G_1-x)-v)\star (G_2-v)$ is satisfied and thus (F3) holds for $G-x$. Hence, $G$ is
	a pseudo-cograph.
\end{proof}

\begin{lemma}\label{lem:diam}
If $G$ is a pseudo-cograph, then $\diam(H)\leq 4$ for every connected induced subgraph
$H$ of $G$. 
\end{lemma}
\begin{proof}
	Let $G$ be a $(v,G_1,G_2)$-pseudo-cograph. There are two cases (1) $G-v =
	(G_1-v)\union (G_2-v)$ or (2) $G-v = (G_1-v)\join (G_2-v)$. Assume first that $G$
	satisfies (1). Hence, $G$ can be obtained from $G_1$ and $G_2$ by taking copies of
	$G_1$ and $G_2$ that are identified on the single common vertex $v$. Let
	$H\subseteq G$ be a connected induced subgraph. If $H\subseteq G_i$, $i\in
	\{1,2\}$, then $\diam(H)\leq 2$ since $G_i$ is a cograph and by Thm.\
	\ref{thm:CharCograph}. Assume that $H$ is not entirely in one of $G_1$ and $G_2$.
	In this case, $H$ must contain vertex $v$ since $H$ is connected and $G-v =
	(G_1-v)\union (G_2-v)$. Hence, we can subdivide $H$ into $H_1\coloneqq H\cap G_1$
	and $H_2\coloneqq H\cap G_2$. Note, both $H_1$ and $H_2$ contain vertex $v$ and, in
	particular, must be induced connected subgraphs of $G_1$ and $G_2$, respectively.
	Again, since $H_i\subseteq G_i$ is an induced connected subgraph of the cograph
	$G_i$, we can apply Thm.\ \ref{thm:CharCograph} and conclude that $\diam(H_i)\leq
	2$, $i\in \{1,2\}$. The latter, in particular, implies that $\dist_{H_i}(v,w)\leq
	2$ for all $w\in V(H_i)$, $i\in \{1,2\}$. Since $H_1$ and $H_2$ have vertex $v$ in
	common it follows that for every two vertices $w\in V(G_1)$ and $u\in V(G_2)$ it
	holds that $\dist_{H}(w,u)\leq \dist_{H_1}(w,v) + \dist_{H_2}(v,u)\leq 4$. Hence,
	$\diam(H)\leq 4$.
	
	Assume now that $G$ satisfies (2). By similar arguments as in the previous case one
	shows that $\diam(H)\leq 4$ for all connected induced subgraphs $H$ of $G$ in case
	$H$ is entirely contained in one of $G_1$ and $G_2$. Assume that $H$ is not
	entirely in one of $G_1$ and $G_2$ and thus, it contains vertices of both $G_1$ and
	$G_2$. Let $x,y \in H$ be distinct. Assume first that $x,y\neq v$. If $x \in
	V(G_1)$ and $y \in V(G_2)$, then $\{x,y\}$ is an edge of $G$, by assumption, and
	thus $\dist_H(x,y)=1$. If $x,y \in V(G_i)$, $i\in\{1,2\}$, then there exists a
	vertex $z \in V(H) \cap V(G_j-v)$, $j\neq i$, such that both $\{x,z\}$ and
	$\{y,z\}$ are edges of $G$. Therefore, $\dist_H(x,y) \leq 2$. Finally assume that
	one $x$ and $y$ coincides with $v$, say $x=v$. Since $H$ is connected and
	$|V(H)|>1$ there is a vertex $z\in V(H)$ such that $\{z,v\}$ is an edge of $G$. If
	$y=z$, then $\dist_H(x,y)= 1$. Otherwise, if $y\neq z$ we can apply the previous
	arguments to the vertices $y$ and $z$ to conclude that $\dist_H(y,z) \leq 2$ and,
	therefore, $\dist_H(x,y) \leq 3$.
	
	In summary, $\diam(H)\leq 4$ for every connected induced subgraph $H$ of $G$.
\end{proof}

\begin{corollary}\label{cor:LongestPath}
A pseudo-cograph does not contain  induced paths $P_n$ with $n\geq 6$.
\end{corollary}

The converse of Lemma \ref{lem:diam} is not satisfied, in general. To see this
consider a cycle $C$ on five vertices. Here $\diam(C)=2$ and thus, in particular,
$\diam(H)\leq 4$ for every connected induced subgraph $H$ of $C$. However, for all
vertices $v\in V(C)$, the graph $C-v$ is isomorphic to a $P_4$ and $C-v$ is not a
cograph. Thus (F3) cannot be satisfied which implies that $C$ is not a
pseudo-cograph. 

Pseudo-cographs are  closed under complementation. 
\begin{lemma}\label{lem:complement}
	A graph $G=(X,E)$ is a pseudo-cograph if and only if  $\overline{G}$ 
	is a pseudo-cograph. In particular, if $|X|\geq 3$ and 
	$G$ is a $(v,G_1,G_2)$-pseudo-cograph, then
	$\overline{G}$ is $(v,\overline{G_1},\overline{G_2})$-pseudo-cograph. 
\end{lemma}
\begin{proof}
	Since $\overline{\overline{G}} = G$ it suffices to show the
	\emph{only-if}-direction. Let $G=(X,E)$ be a pseudo-cograph. If $|X|\leq 2$, then
	$\overline{G}$ is trivially a pseudo-cograph. Hence, suppose that $|X|\geq 3$ and
	that $G$ is a $(v,G_1,G_2)$-pseudo-cograph. Since a graph and its complement have
	the same vertex sets, (F1) must hold for $\overline{G_1} $ and $\overline{G_2}$.
	Moreover, since $G_1$ and $G_2$ are cographs, their complements are cographs and
	thus, (F2) holds for $\overline{G_1} $ and $\overline{G_2}$. Since $G$ satisfies
	(F3), we have $G-v = (G_1-v) \star (G_2-v)$, where $\star \in \{\union, \join\}$.
	It remains to show that (F3) holds for $\overline{G}$. To this end note that
	$\overline{G} - v = \overline{G - v}$ and $\overline{G-v} = \overline{(G_1-v) \star
	(G_2-v)} = \overline{(G_1-v) }\overline{\star} \overline{(G_2-v)} =
	(\overline{G_1}-v) \overline{\star} (\overline{G_2}-v)$ with $\overline{\star}\in
	\{\union, \join\}\setminus \{\star\}$. Hence, (F3) is satisfied for $\overline{G}$.
	In summary, $\overline{G}$ is $(v,\overline{G_1},\overline{G_2})$-pseudo-cograph.
\end{proof}

The next lemma provides a key result that we re-use in many
of the upcoming proofs.
\begin{lemma}\label{lem:ConnComp-v}
	Let $G$ be a $(v, G_1, G_2)$-pseudo-cograph. Furthermore, let $\mathfrak{C}$ be the
	set of connected components of $G-v$ (resp., $\overline{G-v}$), whenever $G-v$ is
	disconnected (resp.\ connected). For every $H\in \mathfrak{C}$ it holds that either
	$H\subsetneq G_1$ or $H\subsetneq G_2$. Moreover, if there are two elements
	$H,H'\in \mathfrak{C}$ such that the subgraph of $G$ induced by $V(H) \cup V(H')
	\cup \{v\}$ contains an induced $P_4$, then $H\subsetneq G_i$ and $H'\subsetneq
	G_j$ with $i,j\in \{1,2\}$ being distinct.
\end{lemma}
\begin{proof}
 Let $G$ be a $(v, G_1, G_2)$-pseudo-cograph. By Obs.\ \ref{obs:G-v-Cograph}, $G-v$
 is a cograph. By Thm.\ \ref{thm:CharCograph}, $G-v$ or its complement is
 disconnected. In particular, by (F3), $G-v$ is either the join or the disjoint union
 of $G_1-v$ and $G_2-v$. Note, if $G-v$ is connected (resp.\ disconnected) it cannot
 be the disjoint union (resp.\ join) of $G_1-v$ and $G_2-v$. This together with (F3)
 implies that if $G-v$ is connected, then $G-v=(G_1-v) \join (G_2-v)$ and if $G-v$ is
 disconnected, then $G-v=(G_1-v) \union (G_2-v)$. In the latter case, it is easy to
 see that $\mathfrak{C}$ consists of at least two elements and for every connected
 component $H \in \mathfrak C$ it must hold that either $H \subsetneq G_1$ or $H
 \subsetneq G_2$. Assume that $G-v$ is connected and thus, $G-v=(G_1-v) \join
 (G_2-v)$. In this case, we consider the connected components of $\overline{G-v} =
 \overline{(G_1-v) \join (G_2-v)} = (\overline{G_1-v}) \union (\overline{G_2-v})$.
 Again, for every connected component $H \in \mathfrak C$ it must hold that either
 $V(H) \subsetneq V(\overline{G_1}) = V(G_1)$ or $V(H) \subsetneq V(\overline{G_2})
 =V(G_2)$.
	
 Finally, by (F2), the graphs $G_1$ and $G_2$ are cographs. Thus, whenever the
 subgraph induced by $V(H) \cup V(H') \cup \{v\}$ contains an induced $P_4$, it
 cannot be entirely contained in either $G_1$ and $G_2$ and, therefore, $H\subsetneq
 G_i$ and $H'\subsetneq G_j$ with $i,j\in \{1,2\}$ being distinct.
\end{proof}

We are now in the position to provide several characterizations of pseudo-cographs
in terms of the connected components of $G-v$, resp., $\overline{G-v}$.
\begin{proposition}\label{prop:charPseudo}
A graph $G$ is a pseudo-cograph if and only if $|V(G)|\leq 2$ or there is a vertex
$v\in V(G)$ for which the following two conditions are satisfied:
\begin{enumerate}[noitemsep]
\item[(A1)] $G-v$ or $\overline{G-v}$ is disconnected with set of connected
            components $\mathfrak{C}$; and\smallskip
\item[(A2)] there is a bipartition $\{\mathfrak{C}_1, \mathfrak{C}_2\}$ of
            $\mathfrak{C}$ such that $G[V_1]$ and $G[V_2]$ are cographs where $V_i
            \coloneqq \{v\}\cup\left(\bigcup_{H\in \mathfrak{C}_i} V(H)\right)$,
            $i\in\{1,2\}$.
\end{enumerate}
In this case, $G$ is a $(v,G[V_1],G[V_2])$-pseudo-cograph.
\end{proposition}
\begin{proof}
 Let $G$ be a pseudo-cograph. If $|V(G)|\leq 2$, we are done. Hence assume that
 $|V(G)|\geq 3$ and that $G$ is a $(v, G_1, G_2)$-pseudo cograph. By Obs.\
 \ref{obs:G-v-Cograph}, $G-v$ is a cograph. Thm.\ \ref{thm:CharCograph} implies that
 either $G-v$ or $\overline{G-v}$ is disconnected. Thus, (A1) is satisfied. Let us
 denote with $\mathfrak{C} $ the set of connected components of the respective
 disconnected graph $G-v$ or $\overline{G-v}$. If $G$ is a cograph, then we put
 $\mathfrak{C}_1 = \{H\}$ for some $H\in \mathfrak{C}$ and
 $\mathfrak{C}_2=\mathfrak{C}\setminus H$ to obtain a bipartition of $\mathfrak{C}$.
 Put $V_1 = V(H)\cup\{v\}$ and $V_2 = \{v\}\cup(\cup_{H\in \mathfrak{C}_2} V(H))$.
 Both graphs $G[V_1]$ and $G[V_2]$ are induced subgraphs of $G$ and, by Thm.\
 \ref{thm:CharCograph}, $G[V_1]$ and $G[V_2]$ are cographs. Hence, (A2) is satisfied.
 If $G$ is not a cograph, we can apply Lemma \ref{lem:ConnComp-v} to conclude that
 for every $H\in \mathfrak{C}$ it holds that either $V(H)\subsetneq V(G_1)$ or
 $V(H)\subsetneq V(G_2)$. Hence, there is a quasi-bipartition $\mathfrak{C}_1\union
 \mathfrak{C}_2$ of $\mathfrak{C}$ defined by putting $H\in \mathfrak{C}_i$ whenever
 $V(H)\subsetneq V(G_i)$, $i\in \{1,2\}$. By construction $V(G_i) = V_i \coloneqq
 \{v\}\cup\left(\bigcup_{H\in \mathfrak{C}_i} V(H)\right)$ and thus, $G[V_i]=G_i$,
 $i\in \{1,2\}$ is a cograph. Since $G$ is not a cograph, it follows that neither
 $G[V_1]\simeq G$ nor $G[V_2]\simeq G$, which implies that $\mathfrak{C}_1\neq
 \emptyset$ and $\mathfrak{C}_2\neq\emptyset$. Therefore, $\{\mathfrak{C}_1,
 \mathfrak{C}_2\}$ is a bipartition of $\mathfrak{C}$ that satisfies (A2).
   	 
 For the converse, if $|V(G)|\leq 2$, then we are done. Hence, assume that
 $|V(G)|\geq 3$ and that there is a vertex $v\in V(G)$ such that (A1) and (A2) are
 satisfied. By (A1), $G-v$ or its complement is disconnected and we denote with
 $\mathfrak{C}$ the set of connected components of the respective disconnected graph.
 Since (A2) is satisfied, there is a bipartition $\mathfrak{C}_1\union
 \mathfrak{C}_2$ of $\mathfrak{C}$ such that $G_1\coloneqq G[V_1]$ and $G_2\coloneqq
 G[V_2]$ are cographs where $V_i \coloneqq \{v\}\cup\left(\bigcup_{H\in
 \mathfrak{C}_i} V(H)\right)$, $i\in\{1,2\}$. Therefore, $G_1$ and $G_2$ satisfy
 (F2). Moreover, since $\mathfrak{C}_1\neq \emptyset$ and $\mathfrak{C}_2\neq
 \emptyset$ and by construction of $V_1$ and $V_2$ it follows that (F1) is satisfied.
 We continue with showing (F3). If $G-v$ is disconnected, then $G-v =\union_{H\in
 \mathfrak{C}} H = (\union_{H\in \mathfrak{C}_1} H)\union (\union_{H\in
 \mathfrak{C}_2} H) = (G_1-v) \union (G_2-v)$ and (F3) is satisfied. Assume now that
 $G-v$ is connected, in which case $\overline{G-v}$ is disconnected and, in
 particular, $\overline{G-v} = \union_{H\in \mathfrak{C}} H$. Hence, $G-v =
 \overline{\overline{G-v}} = \overline{\union_{H\in \mathfrak{C}} H} =
 \overline{(\union_{H\in \mathfrak{C}_1} H)\union (\union_{H\in \mathfrak{C}_2} H)}
 =\overline{(\union_{H\in \mathfrak{C}_1} H)}\join \overline{(\union_{H\in
 \mathfrak{C}_2} H)} =(G_1-v)\join (G_2-v)$. Hence, (F3) is satisfied.
			 
	In summary, $G$ and $v$ satisfy (F1), (F2) and (F3) and, therefore, $G$ is a
	$(v,G[V_1],G[V_2])$-pseudo-cograph. 
\end{proof}

\begin{theorem}\label{thm:CharPsG}
The following statements are equivalent for every graph $G$. 
\begin{enumerate}
	\item $G$ is a pseudo-cograph. 
  \item Either $G$ is a cograph or there is a vertex $v\in V(G)$ that satisfies the following conditions:
	\begin{enumerate}
		\item[(B1)] $G-v$ or $\overline{G-v}$ is disconnected with set of connected components $\mathfrak{C}$; and
		\item[(B2)] $G[V(H)\cup \{v\}]$ is a cograph for all $H\in \mathfrak{C}$; and 
		\item[(B3)]  all edges of the graph $\Gamma(G,v)$ are incident to the same
		             vertex, where $\Gamma(G,v)$ is the undirected graph whose vertex set
		             is $\mathfrak{C}$ and that contains all edges $\{H,H'\}$ for which
		             the subgraph of $G$ induced by $V(H)\cup V(H')\cup\{v\}$ contains an
		             induced $P_4$. 
	\end{enumerate}
	\item Either $G$ is a cograph or there is a vertex $v\in V(G)$ that satisfies the following conditions:
	\begin{enumerate}
		\item[(C1)] $G-v$ or $\overline{G-v}$ is disconnected with set of connected components $\mathfrak{C}$; and
		\item[(C2)] There exists a component $H \in \mathfrak C$ such that $G[V(H)\cup 		
		\{v\}]$ and $G[W]$ are cographs, where $W\coloneqq
		            V(G)\setminus V(H)$. Moreover, $\mathfrak C$ contains at most two 
		            components satisfying the latter property.
		            \smallskip
					
								In this case, $G$ or $\overline G$ is a $(v,G',G'')$-pseudo-cograph 
								with $G' =  G[V(H_i)\cup \{v\}]$ and $G''=G[W_i]$, $i\in \{1,2\}$.
	\end{enumerate}
\end{enumerate}
\end{theorem}
\begin{proof}
 We start with showing that (1) implies (2). Hence, let $G$ be a pseudo-cograph. If
 $G$ is a cograph, we are done. Thus, assume that $G$ is not a cograph and hence,
 $|V(G)|\geq 4$. In particular, $G$ is a $(v,G_1,G_2)$-pseudo-cograph for some $v\in
 V(G)$ and some $G_1,G_2\subset G$. By Obs.\ \ref{obs:G-v-Cograph}, $G-v$ is a
 cograph and, by Thm.\ \ref{thm:CharCograph}, either $G-v$ or $\overline{G-v}$ is
 disconnected, i.e., Condition (B1) holds and therefore, $\Gamma(G,v)$ is
 well-defined. Let $\mathfrak{C}$ be the set of connected components of the
 disconnected graph in $\{G-v, \overline{G-v}\}$. By Lemma \ref{lem:ConnComp-v},
 either $H\subsetneq G_1$ or $H\subsetneq G_2$ for all $H\in \mathfrak{C}$. Hence,
 $G[V(H)\cup \{v\}]$ is an induced subgraph of either $G_1$ or $G_2$ and thus,
 $G[V(H)\cup \{v\}]$ must be a cograph by Thm.\ \ref{thm:CharCograph}. Therefore (B2)
 is satisfied. 
	
 	We continue with showing that (B3) holds. In the following, we assume first that
 $G-v$ is disconnected. If $|\mathfrak{C}| = 2$, then $\Gamma(G,v)$ contains at most
 one edge and the statement is vacuously true. Let $|\mathfrak{C}| \geq 3$. By Prop.\
 \ref{prop:charPseudo} there is a there is a bipartition $\{\mathfrak{C}_1,
 \mathfrak{C}_2\}$ of $\mathfrak{C}$ such that $G[V_1]$ and $G[V_2]$ are cographs
 where $V_i \coloneqq \{v\}\cup\left(\bigcup_{H\in \mathfrak{C}_i} V(H)\right)$,
 $i\in\{1,2\}$. This immediately implies that $\Gamma(G,v)$ can only contain edges
 $\{H,H'\}$ with $H\in V_1$ and $H'\in V_2$ and thus, $\Gamma(G,v)$ must be
 bipartite. Assume, for contradiction, that not all edges of $\Gamma(G,v)$ are
 incident to the same vertex. Since $\Gamma(G,v)$ is bipartite, there must be two
 vertex disjoint edges, say $\{H_1,H_2\}$ and $\{H'_1,H'_2\}$ in $\Gamma(G,v)$.
 W.l.o.g.\ we assume, by Lemma \ref{lem:ConnComp-v}, that $H_1\subset G_1$ and
 $H_2\subset G_2$. Since $H_1\cup H_2\subseteq G-v$ is a cograph and since
 $\{H_1,H_2\}$ is an edge in $\Gamma(G,v)$, there is an induced $P_4$ in $G$ that
 contains $v$ and three further vertices $x,y$ and $z$ such that two of these
 vertices are contained in $H_i$ and one vertex is contained in $H_j$, $\{i,j\} =
 \{1,2\}$. W.l.o.g.\ assume that $x,y\in V(H_1)$ and $z\in V(H_2)$. Since there are
 no edges between vertices of $H_1$ and $H_2$ in $G$, we can assume w.l.o.g.\ that
 the edges of this $P_4$ are $\{v,x\}\{x,y\}$ and $\{v,z\}$. By similar arguments and
 since $H'_1\cup H'_2\subseteq G-v$ is a cograph and $\{H'_1,H'_2\} \in
 E(\Gamma(G,v))$, there is an induced $P_4$ that contains $v$ and three further
 vertices $x',y'$ and $z'$ (that are all distinct from $x,y$ and $z$) and that has
 edges $\{v,x'\}\{x',y'\}$ and $\{v,z'\}$. Moreover, $x',y' \in V(H'_i)$ and $z'\in
 V(H'_j)$, $\{i,j\} = \{1,2\}$. It is now straightforward to verify that the subgraph
 $\Gamma(G,v)$ induced by $H_1,H_2$ and $H'_i$ must be a $K_3$ since all subgraphs of
 $G$ induced by $V(H_1)\cup V(H_2)\cup \{v\}$, $V(H_1)\cup V(H'_i)\cup \{v\}$ and
 $V(H_2)\cup V(H'_i)\cup \{v\}$ contain an induced $P_4$; a contradiction to
 bipartiteness of $\Gamma(G,v)$. Hence, all edges must be incident to the same vertex
 and thus, Statement (2) holds in case that $G-v$ is disconnected. Assume that $G-v$
 is connected. Since $G-v$ is a cograph, $\overline{G-v} =\overline{G}-v$ is
 disconnected. Moreover, Lemma \ref{lem:complement} implies that $\overline{G}$ is a
 $(v,\overline{G_1},\overline{G_2})$-pseudo-cograph and we can apply analogous
 arguments to $\overline{G}$ and $\overline{G}-v$ to conclude that statement (2)
 holds. 
	
	We continue with showing that (2) implies (3). If $G$ is a cograph, there is
	nothing to show. Hence, assume that $G$ is not a cograph and that there is a vertex
	$v\in V(G)$ such that (B1), (B2) and (B3) are satisfied. Note, (B1) and (C1) are
	equivalent. Hence, it remains to show that (C2) is satisfied. Let $\mathfrak{C}$ be
	the set of connected components of the disconnected graph in
	$\{G-v,\overline{G-v}\}$. By (B3), there is a connected component $H\in
	\mathfrak{C}$ to which all edges of $\Gamma(G,v)$ are incident. By (B2),
	$G[V(H')\cup \{v\}]$ is a cograph for all $H'\in \mathfrak{C}$. Hence, $G[V(H)\cup
	\{v\}]$ is a cograph and moreover, every component $H'\in \mathfrak{C}$ is a
	cograph, since they are induced subgraphs of the cographs $G[V(H')\cup \{v\}]$
	(cf.\ Thm.\ \ref{thm:CharCograph}). This implies that $G-v$ is a cograph since
	$G-v$ or $\overline{G-v}$ is disconnected whose connected components are cographs.
	Hence, if there is any induced $P_4$ of $G$, then it must contain vertex $v$. Now
	put $G[W]$, where $W\coloneqq V(G)\setminus V(H)$. Note that $v\in W$. By (B3), the
	subgraph of $\Gamma(G,v)$ induced by the vertex set $\mathfrak{C}\setminus \{H\}$
	does not contain any edges. This together with the fact that $v$ is contained in
	every induced $P_4$ of $G$ implies that $G[W]$ must be a cograph. 

	So-far, we have shown that there is at least one component $H\in \mathfrak{C}$ such 
	that
	$G[V(H)\cup \{v\}]$ and $G[W]$ are cographs. It remains to show that there exists at 
	most two such components. Since $G$ is not a cograph and $G-v$
	is a cograph and since every induced $P_4$ in $G$ contains vertex $v$, we can
	conclude that $\Gamma(G,v)$ contains at least one edge $\{H,H'\}$. It is easy to
	see that if $\Gamma(G,v)$ contains only the edge $\{H,H'\}$, then all of the
	previous arguments hold also for $G'_1\coloneqq G[V(H')\cup \{v\}]$ and $G'_2 =
	G[W']$, where $W'\coloneqq V(G)\setminus V(H')$. Hence, there are at least two
	components in $\mathfrak{C}$ for which (B2) is satisfied. Now assume that
	$\Gamma(G,v)$ contains at least two edges $\{H,H'\}$ and $\{H,H''\}$. Assume, for
	contradiction, that there is a further $H'''\in \mathfrak{C}$, $H\neq H'''$ such
	that $G[V(H''')\cup \{v\}]$ and $G[W''']$ are cographs where $W'''\coloneqq
	V(G)\setminus V(H''')$. Hence, $H$ must be contained in $G[W''']$ and at least one
	of $H'$ and $H''$ must be contained in $G[W''']$ as well. But since $H$ is incident
	both $H'$ and $H''$ one of the induced subgraphs $G[V(H)\cup V(H')\cup v]$ or
	$G[V(H)\cup V(H'')\cup v]$ is an induced subgraph of $G[W''']$ and contains induced
	$P_4$s. Hence, $G[W''']$ is not a cograph; a contradiction. Therefore, (C2) is
	satisfied.

 We finally show that (3) implies (1). If $G$ is a cograph, then it is a
 pseudo-cograph by Lemma \ref{lem:cographPscograph}. Assume that $G$ is not a cograph
 and thus $|X|\geq 4$. Let $v\in V(G)$ be a vertex satisfying (C1) and (C2). Let
 $\mathfrak{C}$ be the set of connected components of the disconnected graph in
 $\{G-v,\overline{G-v}\}$. By assumption, there is a component $H \in \mathfrak{C}$
 such that $G[V(H)\cup \{v\}]$ and $G[W]$ are cographs with $W\coloneqq V(G)\setminus
 V(H)$. If $G-v$ is disconnected, then put $G_1\coloneqq G[V(H)\cup \{v\}]$ and
 $G_2\coloneqq G[W]$, and otherwise, put $G_1\coloneqq \overline{G[V(H)\cup \{v\}}]$
 and $G_2\coloneqq \overline{G[W]}$. We show that $G$ is
 $(v,G_1,G_2)$-pseudo-cograph. By construction and the latter arguments (F1) and (F2)
 are satisfied. Moreover, if $G-v$ is disconnected it is, by construction, the
 disjoint union of $G_1-v$ and $G_2-v$ and, if $G-v$ is connected, then
 $\overline{G-v} = G[V(H)]\union G[W\setminus\{v\}]$ and thus, $G-v =
 \overline{\overline{G-v}} = \overline{G[V(H)]\union G[W\setminus \{v\}]} =
 \overline{G[V(H)]}\join \overline{G[W\setminus \{v\}]} = G_1-v \join G_2-v$. Hence,
 (F3) is satisfied, which completes the proof. 
\end{proof}

\begin{corollary}\label{cor:Gamma}
Let $G$ be a pseudo-cograph and $v\in V(G)$ be a vertex that is contained in every
 induced $P_4$ of $G$. Then, $\Gamma(G,v)$ is either edge-less (in which case $G$ is a
 cograph) or it contains precisely one connected component that is isomorphic to a
 star while all other remaining components (if there are any) are single vertex
 graphs (in which case $G$ is not a cograph).
\end{corollary}
\begin{proof}
 Let $G$ be a pseudo-cograph and $v\in V(G)$ be a vertex that is contained in every
 induced $P_4$ of $G$. In this case, $G-v$ must be a cograph. If $|V(G-v)|\in
 \{0,1\}$, then $\Gamma(G,v)$ is trivially edge-less. Assume that $|V(G-v)|\geq 2$.
 By Thm.\ \ref{thm:CharCograph}, either $G-v$ or $\overline{G-v}$ must be
 disconnected. Thus, $\Gamma(G,v)$ is well-defined and has as vertex set
 $\mathfrak{C}$ the connected components of the disconnected graph in
 $\{G-v,\overline{G-v}\}$. In particular, $|\mathfrak{C}|\geq 2$ and thus,
 $\Gamma(G,v)$ has at least two vertices. If $G$ is a cograph, then $G$ does not
 contain any induced $P_4$ and one easily verifies that $\Gamma(G,v)$ is edge-less.
 Assume now that $G$ contains induced $P_4$s. By Thm.\ \ref{thm:CharPsG}(B3), all
 edges of the graph $\Gamma(G,v)$ are incident to the same vertex and thus,
 $\Gamma(G,v)$ contains precisely one connected component that is isomorphic to a
 star while all other remaining components are $K_1$s.
\end{proof}

We investigate now in some detail to what extent the choice of the vertices $v$
and subgraphs $G_1$ and $G_2$ are unique for $(v,G_1,G_2)$-pseudo-cographs.

\begin{lemma}\label{lem:v-in-allP4}
	Let $G$ be a $(v,G_1,G_2)$-pseudo-cograph. Then, every induced $P_4\subseteq G$
	must contain vertex $v$ and, if $G$ is not a cograph, there are at most four
	vertices $v'\in V$ such that $G$ is a $(v',G'_1,G'_2)$-pseudo-cograph. In particular,
	if $\{P^1,\dots,P^k\}$, $k\geq 2$ is the set of all induced $P_4$s 
	in $G$, then,
	the number $\ell$ of choices for $v$ is $1\leq \ell = |\cap_{i=1}^kV(P^i)|\leq 3$
	Moreover, if $G$ contains a $P_5$ $x_1-x_2-x_3-x_4-x_5$ or its complement
	$\overline{P_5}$ as an induced subgraph, then $v = x_3$ is uniquely determined.
\end{lemma}
\begin{proof}
	If $G$ is a cograph, the statement is vacuously true. Assume that $G$ is not a
	cograph. Thus, $G$ contains an induced $P_4$. By Obs.\ \ref{obs:G-v-Cograph}, $G-v$
	is a cograph which immediately implies that every $P_4$ must contain $v$. Since,
	for a given induced $P_4$ there are four possible choices for $v \in V(P_4)$, there
	are at most four vertices $v\in V$ such that $G$ is a $(v,G_1,G_2)$-pseudo-cograph.
	Clearly the number $|V(P)\cap V(P')|$ of vertices in the intersection of distinct
	induced $P_4$s $P$ and $P'$ must be less than four and restrict the number of
	possible choices for $v$ such that $G$ is $(v,G_1,G_2)$-pseudo-cograph.
	
	Now assume that $G$ is a $(v,G_1,G_2)$-pseudo-cograph that contains an induced
	$P_5$ $x_1-x_2-x_3-x_4-x_5$. Clearly neither $v=x_1$ nor $v=x_5$ is possible since,
	otherwise, $G-v$ contains still an induced $P_4$; violating Obs.\
	\ref{obs:G-v-Cograph}. Now assume that $v=x_4$. By Obs.\ \ref{obs:G-v-Cograph},
	$G-v$ is a cograph and, by Thm.\ \ref{thm:CharCograph}, either $G-v$ or
	$\overline{G-v}$ must be disconnected. Assume first that $G-v$ is disconnected. In
	this case, there is one connected component $H$ of $G-v$ that contains the vertices
	$x_1,x_2,x_3$. By Lemma \ref{lem:ConnComp-v}, $H\subset G_i$ for one $i\in
	\{1,2\}$. By (F1), both $G_i$ contains vertex $v=x_4$ and thus, the particular
	graph $G_i$ contains all vertices $x_1,\dots,x_4$ and therefore, an induced $P_4$;
	a contradiction to (F2). Hence, $G-v$ must connected, in which case
	$\overline{G-v}$ must be disconnected. In particular, $\overline{G}$ contains the
	complement $\overline{P_5}$ as an induced subgraph. Note that $\overline{P_5}$ has
	edges $\{x_1,x_3\}$, $\{x_3,x_5\}$ and $\{x_2,x_5\}$ (among others). Hence,
	$\overline{P_5}-v$ is contained in a connected component $H$ of $\overline{G-v}$.
	Lemma \ref{lem:ConnComp-v} implies that $H\subset G_i$ for one $i\in \{1,2\}$.
	Since $v\in V(G_i)$, the graph $\overline{G_i}$ contains an induced
	$\overline{P_5}$ and thus, in particular, the induced $P_4$ $x_3-x_1-x_4-x_2$.
	Therefore, $\overline{G_i}$ and thus, $G_i$ cannot be a cograph and  (F2) is
	violated; a contradiction. Thus, $v=x_4$ is not possible. By analogous
	arguments, one shows that $v=x_4$ is not possible in case that $G$ is a
	$(v,G_1,G_2)$-pseudo-cograph that contains the complement $\overline{P_5}$ as an
	induced subgraph.
	By the same arguments, $v=x_2$ is not possible. Hence, if $G$ is a
	$(v,G_1,G_2)$-pseudo-cograph that contains such an induced $P_5$, then $v=x_3$ is
	uniquely determined. 
\end{proof}

\begin{lemma}\label{lem:NONunique-cograph}
If $G$ is a cograph with at least three vertices, then $G$ is a
$(v,G_1,G_2)$-pseudo-cograph for every $v\in V(G)$ and all graphs $G_1$ and $G_2$
that satisfy $G_1 = G[V_1]$ and $G_2=G[V_2]$ with $V_i \coloneqq
\{v\}\cup\left(\bigcup_{H\in \mathfrak{C}_i} V(H)\right)$, $i\in \{1, 2\}$ for an
arbitrary bipartition $\{\mathfrak{C}_1, \mathfrak{C}_2\}$ of the vertex set
$\mathfrak{C}$ of $\Gamma(G,v)$, i.e., the connected components in the disconnected
graph in $\{G-v,\overline{G-v}\}$.
\end{lemma}
\begin{proof}
	Let $G$ be a cograph with at least three vertices and $v\in V(G)$ be an arbitrary
	vertex. By Thm.\ \ref{thm:CharCograph}, $G-v$ is a cograph and either $G-v$ or
	$\overline{G-v}$ must be disconnected. Thus, $\Gamma(G,v)$ is well-defined and has
	as vertex set $\mathfrak{C}$ the connected components of the disconnected graph in
	$\{G-v,\overline{G-v}\}$. In particular, since $G-v$ has at least two vertices, it
	holds that $|\mathfrak{C}|\geq 2$. By Cor.\ \ref{cor:Gamma}, $\Gamma(G,v)$ is
	edge-less. Let $\{\mathfrak{C}_1 , \mathfrak{C}_2\}$ be any bipartition of
	$\mathfrak{C}$ and put $V_i \coloneqq \{v\}\cup\left(\bigcup_{H\in \mathfrak{C}_i}
	V(H)\right)$, $1\leq i\leq 2$ and $G_1\coloneqq G[V_1]$ and $G_2\coloneqq G[V_2]$.
	By construction, $G_1$ and $G_2$ satisfy (F1) and $G-v$ is still the disjoint union
	or join of $G_1-v$ or $G_2-v$ and thus, (F3) is satisfied. Since both $G_1$ and
	$G_2$ are induced subgraphs of $G$, Thm.\ \ref{thm:CharCograph} implies that $G_1$
	and $G_2$ are cographs. Therefore, (F2) is satisfied. 
\end{proof}

\begin{lemma}\label{lem:star-center-gamma}
If $G$ is a $(v,G',G'')$-pseudo-cograph, then $G$ is a $(v,G_1,G_2)$-pseudo-cograph
for all graphs $G_1$ and $G_2$ that satisfy $G_1 = G[V_1]$ and $G_2=G[V_2]$ with $V_i
\coloneqq \{v\}\cup\left(\bigcup_{H\in \mathfrak{C}_i} V(H)\right)$, $1\leq i\leq 2$
for an arbitrary graph-bipartition $\{\mathfrak{C}_1 , \mathfrak{C}_2\}$ of the
vertex set $\mathfrak{C}$ of $\Gamma(G,v)$.  

In particular, if $G$ is $(v,G_1,G_2)$-pseudo-cograph but not a cograph, then the
center $H$ of the star in $\Gamma(G,v)$ is in $G_i$; all $H'\in \mathfrak{C}$ that
are adjacent to $H$ are in $G_j$, $\{i,j\}=\{1,2\}$; and every $K_1\in \mathfrak{C}$
is in either $G_1$ or $G_2$. If $\Gamma(G,v)$ is connected, then $G_1$ and $G_2$ are
uniquely determined. 
\end{lemma}
\begin{proof}
	Assume that $G$ is a $(v,G',G'')$-pseudo-cograph. By (F3), $G-v$ is the join or
	disjoint union $G'-v$ and $G''-v$ and thus, $\Gamma(G,v)$ is well-defined and has
	as vertex set $\mathfrak{C}$ the connected components of the disconnected graph in
	$\{G-v,\overline{G-v}\}$. If $G$ is a cograph, then the assertion follows from
	Lemma \ref{lem:NONunique-cograph}. Assume that $G$ is not a cograph. Since $G-v$ is
	a cograph, the vertex $v$ must be contained in contained in every induced $P_4$ of
	$G$. Hence, we can apply Cor.\ \ref{cor:Gamma} to conclude $\Gamma(G,v)$ has at
	least two vertices and contains precisely one connected component that is
	isomorphic to a star while all other remaining components (if there are any) are
	single vertex graphs.
	
	In particular $\Gamma(G,v)$ is bipartite. Let $\{\mathfrak{C}_1 , \mathfrak{C}_2\}$
	be any graph-bipartition of $\mathfrak{C}$. Let $H\in \mathfrak{C}$ be the
	component that is contained in every edge of $\Gamma(G,v)$ and assume w.l.o.g.\
	that $H\in \mathfrak{C}_1$. Thus for every edge $\{H,H'\}$ in $\Gamma(G,v)$ it must
	hold that $H'\in \mathfrak{C}_2$. Now assign every remaining component in $
	\mathfrak{C}$ that is isomorphic to a $K_1$ in an arbitrary way to either
	$\mathfrak{C}_1$ or $\mathfrak{C}_2$. Finally put $V_i \coloneqq
	\{v\}\cup\left(\bigcup_{H\in \mathfrak{C}_i} V(H)\right)$, $1\leq i\leq 2$ and
	$G_1\coloneqq G[V_1]$ and $G_2\coloneqq G[V_2]$. By construction $G_1$ and
	$G_2$ satisfy (F1) and $G-v$ is still the disjoint union or join of $G_1-v$ or
	$G_2-v$ and thus, (F3) is satisfied. It remains to show
	that both $G_1$ and $G_2$ are cographs. By construction of $V_1$ and since
	$\{\mathfrak{C}_1 , \mathfrak{C}_2\}$ is any graph-bipartition of
	$\mathfrak{C}$, it follows that $H''$ and $H'''$ cannot be adjacent for all
	$H'',H'''\in \mathfrak{C}_1$ and thus, $G[V(H'')\cup V(H'')\cup \{v\}]$ does not
	contain any induced $P_4$. Since the latter holds for all $H'',H'''\in
	\mathfrak{C}_1$ it follows that $G[V_1]=G_1$ must be a cograph.
	 Analogously, $G_2$
	is a cograph and therefore, (F2) is satisfied. 
	
	In summary, $G$ is a $(v,G_1,G_2)$-pseudo-cograph. By the latter construction and
	arguments, it is easy to verify that the last part of the statement is satisfied. 
\end{proof}

An example that shows the different construction of subgraphs $G_1$ and $G_2$ based
on $\Gamma(G,v)$ such that $G$ is a $(v,G_1,G_2)$-pseudo-cograph is provided in Fig
\ref{fig:gamma}.

\begin{figure}[t]
		\begin{center}
			\includegraphics[width = 1.\textwidth]{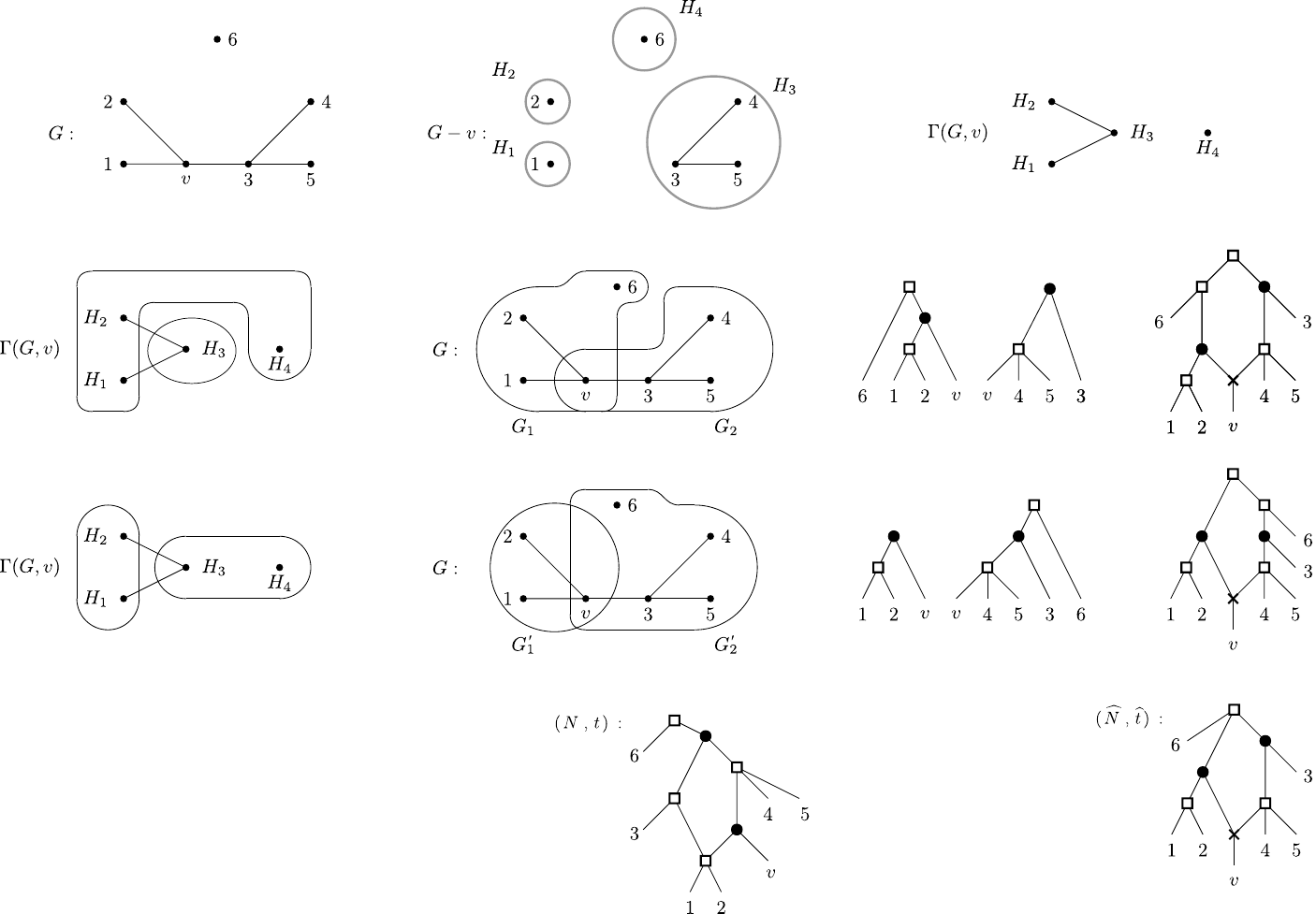}
		\end{center}
		\caption{
		\emph{Upper Row:} 
		Shown is a pseudo-cograph $G$. The connected components $H_1,\dots,H_4$ of $G-v$
		are the vertices of the graph $\Gamma(G,v)$. Here, $\Gamma(G,v)$ has two
		connected components, one is a star with center $H_3$ and the other a single
		vertex. 
		\emph{2nd and 3rd Upper Row:} 
		There are two possible graph-bipartitions of $\Gamma(G,v)$ (shown left) which
		yield two different ways to write $G$ as a $(v,G_1,G_2)$- and
		$(v,G'_1,G'_2)$-pseudo-cograph. The different networks
		$(N(v,G_1,G_2),t(v,G_1,G_2))$ and $(N(v,G'_1,G'_2),t(v,G'_1,G'_2))$ are
		constructed as specified in Def.\ \ref{def:prop:PsC-l1N}. 
		\emph{Bottom Row:}
		The quasi-discriminating versions $(\widehat N,\widehat t)$ of
		$(N(v,G_1,G_2),t(v,G_1,G_2))$ and $(N(v,G'_1,G'_2),t(v,G'_1,G'_2))$ i.e., the
		network obtained after contracting the edge $(\rho_N,\parent_N(6))$ are identical
		(shown right). Left: shown is an additionally quasi-discriminating level-1
		network $(N,t)$ that explains $G$ but that cannot be obtained by means of the
		construction in Def.\ \ref{def:prop:PsC-l1N} since the root of $N$ is not part of
		a cycle and the hybrid-vertex of $N$ has two children. We explain in Section 
		\ref{sec:general} how $(N,t)$ can be constructed.
			}
		\label{fig:gamma}
\end{figure}

There is a quite simple construction to obtain a labeled level-1 network on $X$ that explains
a $(v,G_1,G_2)$-pseudo-cograph $G=(X,E)$. To this end, consider the 
cotrees $(T_1,t_1)$ and $(T_2,t_2)$ for the cographs $G_1=(X_1,E_1)$
and $G_2=(X_2,E_2)$, respectively. Note that, by definition of pseudo-cographs,
$X_1\cap X_2 =\{v\}$, $X_1\cup X_2 = X$ and $|X_1|,|X_2|>1$. Thus, $T_1$ is a tree on
$X_1$, $T_2$ is a tree on $X_2$ and both trees contain $v$ as a leaf. Based on this
we provide	
		\begin{definition}\label{def:prop:PsC-l1N}
			Let $G$ be $(v,G_1,G_2)$-pseudo-cograph and $(T_i,t_i)$ be a cotree for $G_i$, $i\in \{1,2\}$. 
			We construct now a directed graph \emph{$N(v,G_1,G_2)$ with labeling $t(v,G_1,G_2)$} as follows:
			\begin{enumerate}[noitemsep]
				\item Modify  the trees $T_1$ and $T_2$ by adding a new vertex $\eta_i$ along the edge
						  $(\textrm{parent}_{T_i}(v),v)$ in $T_i$, $i\in\{1,2\}$. \smallskip
				\item Take these modified trees,  add a root $\rho_N$ and two new edges $(\rho_N,\rho_{T_1})$ and
						  $(\rho_N,\rho_{T_2})$. \smallskip
				\item identify $\eta_1$ and $\eta_2$ to obtain the vertex $\eta_N$
						  in $N$. \smallskip
				\item remove one copy of the leaf $v$ and its incident edge. \smallskip
			\end{enumerate}
			The labeling $t\coloneqq t(v,G_1,G_2)$ of $N(v,G_1,G_2)$ is defined as follows:
			\begin{enumerate}[noitemsep]
				\item Put $t(u) \coloneqq  t_i(u)$ for all $u\in V^0(T_i)$ with $i\in \{1,2\}$. \smallskip
				\item Choose  $t(\eta_N)\in \{0,1\}$ arbitrarily. \smallskip
				\item  Put  
				$t(\rho_N)=\begin{cases}
											1, & \text{if $G-v$ is the join of $G_1-v$  and $G_2-v$}\\
					            0, & \text{otherwise}
									 \end{cases} $				
			\end{enumerate}
		\end{definition}

Examples for such networks $(N(v,G_1,G_2),t(v,G_1,G_2))$ are provided in Fig.\
\ref{fig:gamma} and Fig.\ \ref{fig:polarCatN}.

\begin{proposition}\label{prop:PsC-l1N}
	Every  pseudo-cograph can be explained by a labeled level-1 network. 
			In particular, if $G$ is a $(v,G_1,G_2)$-pseudo-cograph and $|V(G)|\geq 3$, 
		then $(N(v,G_1,G_2),t(v,G_1,G_2))$ is a labeled level-1 network $G$ that
		explains $G$. 
\end{proposition}
\begin{proof}
	Let $G=(X,E)$ be a pseudo-cograph. If $|X|\leq 2$, then $G$ is a 
	cograph and it can
	be explained by a tree on $X$. Let $|X|\geq 3$ and assume that $G$ is
	$(v,G_1,G_2)$-pseudo-cograph. Consider the two cotrees $(T_1,t_1)$ and $(T_2,t_2)$
	of $G_1=(X_1,E_1)$ and $G_2=(X_2,E_2)$, respectively. By definition of
	pseudo-cographs, $X_1\cap X_2 =\{v\}$, $X_1\cup X_2 = X$ and $|X_1|,|X_2|>1$. Thus,
	$T_1$ is a tree on $X_1$, $T_2$ is a tree on $X_2$ and both trees contain $v$ as a
	leaf. 

	Consider now the network $(N,t)$ with $N\coloneqq N(v,G_1,G_2)$ and $t\coloneqq
	t(v,G_1,G_2)$. It is straightforward to see that $N$ is a level-1 network whose
	single cycle $C$ consists of sides $P^1$ and $P^2$ that are composed of $\rho_N,
	\eta_N$ and the $\rho_{T_1}v$-path in $T_1$ and the $\rho_{T_2}v$-path in $T_2$,
	respectively. Moreover, since $V^0(T_1)\cup V^0(T_2) = V^0(N)\setminus
	\{\rho_N,\eta_N\}$ and $V^0(T_1)\cap V^0(T_2)= \emptyset$, the labeling $t$ is
	well-defined.

	It remains to show that $G=\mathscr{G}(N,t)$. To this end let $x,y\in X$ be chosen
	arbitrarily. Consider first the case that $x,y \in X_1$. Since $G_1$ is explained
	by $(T_1,t_1)$, it holds that $t_1(\lca_{T_1}(x,y)) = 1$ if and only if $\{x,y\}\in
	E$. It is straightforward to verify that $\lca_{T_1}(x,y) = \lca_{N}(x,y)$ and
	thus, by construction of $t$, that $t(\lca_{N}(x,y)) = 1$ if and only if
	$\{x,y\}\in E$. Note, the latter covers also that case that one of $x$ or $y$ is
	$v$, since $v\in X_1$. Similarly, the case $x,y \in X_2$ is shown. Suppose now
	that $x\in X_1\setminus\{v\}$ and $y\in X_2\setminus\{v\}$. The latter arguments
	together with the fact that the paths in $N$ from $\rho_{T_1}$ to $x$ and
	$\rho_{T_2}$ to $y$ are vertex disjoint and $\rho_N$ is the only vertex that is
	adjacent to both $\rho_{T_1}$ and $\rho_{T_2}$ imply that $\lca_N(x,y) = \rho_N$.
	By (F3), $G-v$ is either the join or disjoint union of $G_1-v$ or $G_2-v$. By
	construction, $(\rho_N)=1$ if $G-v$ is the join of $G_1-v$ or $G_2-v$ and,
	otherwise, $t(\rho_N)=0$. Hence, $t(\rho_N)=1$ if and only if $G-v$ is the join of
	$G_1-v$ and $G_2-v$, if and only if $\{x,y\}\in E$. In summary, $\{x,y\}\in E$ if
	and only if $t(\lca_N(x,y))=1$ for all $x,y\in X$ which implies that
	$G=\mathscr{G}(N,t)$. Hence, $(N,t)$ is a labeled level-1 network that explains the
	pseudo-cograph $G$.
\end{proof}

We emphasize that not all graphs that can be explained by a level-1 network are
pseudo-cographs. To see this consider a graph $G$ that is the disjoint union of two
induced $P_4$s. Each $P_4$ can be explained by a level-1 network as shown in Fig.\
\ref{fig:nonUniqueN}. A level-1 network that explains $G$ can be obtained by joining
the two networks that explain the individual $P_4$s under a common root with label
``$0$''. However, $G$ is not a pseudo-cograph, since for every choice of the vertex
$v$, the graph $G-v$ is not a cograph and thus, by Obs.\ \ref{obs:G-v-Cograph}, $G$
is not a pseudo-cograph. In particular, pseudo-cographs are characterized in terms of
level-1 networks $(N,t)$ that contain one cycle that is rooted at $\rho_N$, 
and whose hybrid-vertex has a unique child which is a leaf, see Thm.\
\ref{thm:CharPsC-Network-Cycle}. A characterization of general graphs that can be
explained by level-1 networks is provided in Thm.\ \ref{thm:CharprimeCat}.

\section{Cographs, Quasi-discriminating Strong and Weak Networks}
\label{sec:Cog}

To recap, a network $(N,t)$ is \emph{quasi-discriminating} if for all $(u,v)\in E^0$
with $v$ not being a hybrid-vertex we have $t(u)\neq t(v)$. Of course, not all
labeled level-1 networks $(N,t)$ are quasi-discriminating. Suppose that $(N=(V,E),t)$
is not quasi-discriminating. In this case, there must exist some edge $e = (u,v) \in
E^0$ such that $t(u)=t(v)$ and where $v$ is not a hybrid-vertex. 
For such an edge $e=(u,v)$, we define the network  $(N_e,t_e)$ obtained from
$(N,t)$ by contraction of $e$ as follows:
\begin{enumerate}
	\item Let $N'_e=	(V'_e,E'_e)$ be  the directed graph with vertex set $V'_e = V 
	\setminus \{u,v\} \cup \{v_e\}$, edge set
	$E'_e = E \setminus \{e\} \cup \{(v_e,w) \colon (v,w) \in E \text{ or } (u,w) \in E
	\} \}\cup \{(w,v_e) \colon (w,v) \in E \text{ or } (w,u) \in E \}$.
	
	Note, $N'_e$ is a directed	graph with leaf set $X$ since neither $v$ 
	nor $u$ can be leaves. In the following, we refer to $v_e$ as the 
		vertex in $N$ obtained by contracting the edge $e$.
	
	\item Now suppress all vertices with indegree 1 and outdegree 1 in $N'_e$ to
	obtain the directed graph $N_e = (V_e,E_e)$.
	
	\item To obtain a labeling of $N_e$, we
	define the map $t_e\colon V_e \to \{0,1,\odot\}$ by putting, for all $w\in V_e$,
		$t_e(w) = t(w) \textrm{ if } w \neq v_e  \textrm{ and } t_e(v_e) = t(u)$, 
		otherwise.
\end{enumerate}

Clearly, this construction can be repeated, with $(N_e,t_e)$ now playing the
role of $(N,t)$, until a directed graph $\widehat N=(\widehat V,\widehat E)$ on
$X$ is obtained together with a quasi-discriminating map $\widehat t$ on
$\widehat N$; see Fig.\ \ref{fig:hatN} for an example of such a construction. We
refer to $(N_e,t_e)$ as the directed graph that is \emph{obtained from $(N,t)$
by contraction of $e$} without explicitly mentioning each time that in addition
indegree 1 and outdegree 1 vertices have been suppressed. We emphasize that our
edge contraction operation used to build $N_e$ from $N$ (i.e., Steps 1 and 2 of
the construction above) is an edge contraction in the sense of
\cite{HSS:22cluster}. In \cite{HSS:22cluster}, edge contractions are defined in
a slightly more general way to deal with networks that may contain so-called
shortcut edges $e=(u,v)$ which is based on the properties of cycles in $N$ and
requires $v$ to be a hybrid-vertex. However, we do not contract edges $(u,v)$
with $v$ being a hybrid-vertex at all. This, in particular, allows us to re-use
some of the results established in \cite{HSS:22cluster}.

\begin{lemma}\label{lem:discriminatingN}
	For every labeled level-1 network $(N,t)$ on $X$, the labeled directed graph
	$(\widehat N, \widehat t)$ is a quasi-discriminating level-1 network on $X$. 
\end{lemma}
\begin{proof}
	Clearly if (N0) holds, i.e., $|V(N)| = 1$ then $|\widehat{V}(N)| = 1$ and we are
	done. Thus, assume that $|V(N)| > 1$. By construction, the labeled directed graph
	$(\widehat N, \widehat t)$ is quasi-discriminating. To prove that $(\widehat N,
	\widehat t)$ is a level-1 network, it suffices to show that for $e=(u,v) \in
	E^0(N)$ such that $t(u)=t(v)$ and $v$ is not a hybrid-vertex of $N$, the labeled
	directed graph $(N_e,t_e)$ obtained by contraction of $e$ remains a level-1 network
	on $X$, since $(\widehat N, \widehat t)=(\widehat{N_e}, \widehat{t_e})$.
			
	Cor.\ 3.28 in \cite{HSS:22cluster}, in particular, implies that 
	   $N_e$ is a DAG. Now, 
	let $v_e$ be the vertex obtained by 
	contraction of $e=(u,v)$.
	By assumption, $v$ is not a hybrid-vertex of $N$. Moreover, $v$ is not a leaf of
	$N$, as $e \in E^0(N)$. In particular, $v$ has indegree 1 and outdegree at least
	$2$ in $N$. By construction, all children of $v$ in $N$ become  children of
	$v_e$ after contraction of $e$. It follows that $v_e$ has outdegree at least $2$ in
	$N_e$. In particular, $v_e$ is not an indegree 1 and outdegree 1 vertex, i.e.,
	$v_e$ will not be suppressed and thus, $v_e \in V(N_e)$. Moreover, since $v$ has
	indegree 1 in $N$ and is, therefore, only adjacent to $u$, it follows that all
	in-neighbors of $u$ in $N$ are precisely the in-neighbors of $v_e$ in $N_e$.
	Consequently, the indegree of $v_e$ in $N_e$ is precisely the indegree of $u$ in
	$N$.
					
	To see that (N1) holds, assume first that $u \neq \rho_N$. In this case, $\rho_N
	\in V(N_e)$ and $\rho_N$ is the only vertex of indegree 0 in $N_e$. Moreover, the
	outdegree of $\rho_N$ remains unchanged and thus, (N1) is satisfied. Assume that
	$u=\rho_N$. In this case, $v_e$ is the only vertex of indegree $0$ in $N_e$.
	Moreover, as argued above, the outdegree of $v_e$ in $N_e$ is at least $2$. Hence,
	(N1) is satisfied. Moreover, since $v$ is not a leaf and since contraction of the
	edge $e$ does not create new leaves in $N_e$ it follows that the leaf set of $N_e$
	must be $X$. It is now easy to verify that (N2) holds.
	
	We continue with showing that (N3) is satisfied. To this end, let $w \in V^0(N_e)$
	such that $w\neq \rho_{N_e}$. Assume first that $w=v_e$. As already argued, $w$ has
	outdegree at least $2$ in $N_e$. Moreover, the indegree of $w$ in $N_e$ is
	precisely the indegree of $u$ in $N$. Since $w$ is distinct from $\rho_{N_e}$,
	vertex $u$ must be distinct from $\rho_N$ and, therefore, the indegree of $u$ in
	$N$ is at least 1. It follows that (N3) is satisfied for $w$. Now, assume that $w
	\neq v_e$. In this case, $w$ must be a vertex of $N$ that is distinct from $u$ and
	$v$ since $w\in V(N_e)\setminus \{v_e\} \subseteq V(N)\setminus \{u,v\}$. If $w$
	has indegree $1$ in $N_e$, then $w$ has outdegree at least $2$ in $N_e$, since
	vertices with indegree 1 and outdegree $1$ are suppressed when constructing $N_e$.
	Moreover, contraction of $e$ does not increase the indegree of vertices of $N$
	distinct from $u$ and $v$, so $w$ has indegree at most $2$ in $N_e$. Hence, (N3) is
	satisfied. In summary, $N_e$ is a phylogenetic network on $X$.
			
  It remains to show that $N_e$ is a level-1 network. To this end, we must show that
  every cycle $C$ in $N_e$ contains at most one 
  hybrid-vertex distinct from $\rho_C$.
  If
  $e=(u,v)$ is not part of a cycle in $N$, then it is easy to see that $N_e$ remains
  a level-1 network. Assume that $e$ is contained in a cycle $C$ in $N$ and let $C_e$
  be the biconnected component in $N_e$ that contains $v_e$. There are cases in which
  $C_e$ is not a cycle in $N_e$, in particular if $u=\rho_C$ and both $(u,\eta_C)$
  and $(v,\eta_C)$ are edges of $N$. If $C_e$ is not a cycle, then it cannot contain
  any hybrid-vertex. Now assume that $C_e$ remains a subgraph in $N_e$ that is
  distinct from a single vertex or an edges. In this case, \emph{none} of the
  vertices of $N$ that are distinct from $u$ and $v$ have been suppressed and thus,
  are still present in $N_e$. This, in particular, implies that $C$ and $C_e$ differ
  only in the edge $e$ and the vertex $v_e$, that is, the topology of $C$ and $C_e$
  remains the same up to the single contraction of $e$. Note that $v\neq \eta_C$ by
  assumption. Moreover, since $v\prec_N u$, we also have $u\neq \eta_C$. Hence, the
  only remaining case we need to show is that $v_e$ is not a hybrid-vertex of $C_e$.
  If $v_e$ has indegree 1 in $N_e$ we are done. Thus, assume $v_e$ has indegree 2 in
  $N_e$. Assume for contradiction $C_e$ has $v_e$ and $\eta_C$ as hybrid-vertices.
  Note, removal of any vertex $w$ in $C_e$ keeps $C_e-w$ connected in $N_e$. Hence,
  removal of any vertex $w$ in $C$ keeps $C-w$ connected. This together with the
  facts that the topology of $C$ and $C_e$ remains the same up to the single
  contraction of $e$ and that $v$ has indegree 1 in $N$ implies that $u\neq \eta_c$
  must be a hybrid-vertex of $C$ in $N$; a contradiction. In summary, every
  biconnected components of $N_e$ contains at most one hybrid-vertex and thus, $N_e$
  is a level-1 network.
\end{proof}

\begin{figure}[t]
		\begin{center}
			\includegraphics[width = .8\textwidth]{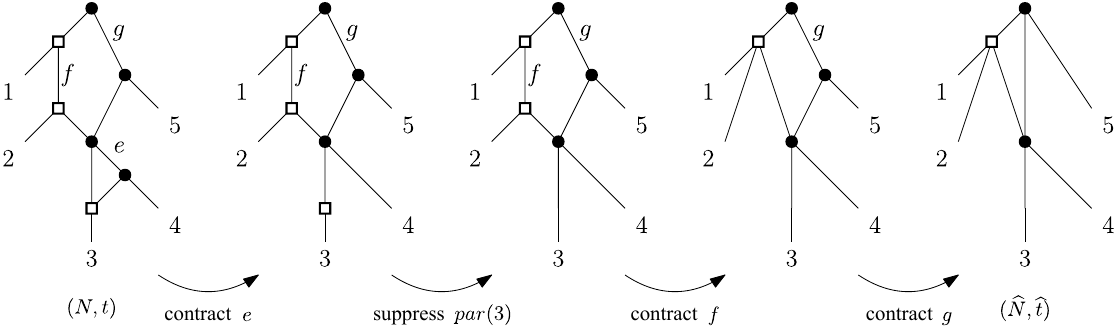}
		\end{center}
		\caption{A level-1 network $(N,t)$ and the resulting quasi-discriminating level-1
		         network $(\widehat N, \widehat t)$. Here, $(\widehat N, \widehat t)$ is
		         weak. Note, although $\widehat N$ contains only one cycle $C$, it is not
		         elementary since there are inner vertices adjacent to more than one
		         leaf and vertex $5$ is adjacent to the root. 
		         Observe that the edge $(\rho_{\widehat N},\eta_{\widehat N})$
		         satisfies $\widehat t(\rho_{\widehat N}) = \widehat t(\eta_{\widehat
		         N})$ but will not be contracted since $\eta_{\widehat N}$ is a
		         hybrid-vertex.}
		\label{fig:hatN}
\end{figure}

 We continue with showing that   $(N,t)$ and $(\widehat N, \widehat t)$ explain
 the same graph. To this end, we provide first the following results which 
 	is slightly adjusted to fit with our notation. 
 \begin{lemma}[{\cite[Prop.\ 7.13]{HSS:22cluster}}]
 		\label{lem:lca-contract}
		Let $N$ be a level-1 network on $X$ and $e = (u,v)\in E^0(N)$ be such that
		$v$ is not a hybrid-vertex and $(N_e,t_e)$ be obtained from $(N,t)$ 
		by contraction of $e$. 
		Then, for all $x,y\in X$, we have $\lca_{N_e}(x,y) =\lca_N(x,y)$ whenever 
		$\lca_N(x,y)\notin \{u,v\}$ and, otherwise $\lca_{N_e}(x,y)=v_e$.
\end{lemma}

\begin{proposition}\label{prop:NhatN-sameGraph}
	If $(N,t)$ is a labeled level-1 network, then it holds that
	$\mathscr{G}(N,t) = \mathscr{G}(\widehat N, \widehat t)$.
\end{proposition}
\begin{proof}
	Let $(N,t)$ be a labeled level-1 network on $X$ and $e=(u,v) \in E^0(N)$ be an edge
	of $N$ that satisfies $t(u)=t(v)$ and where $v$ is not a hybrid-vertex. Consider
	the directed graph $(N_e,t_e)$ that is obtained from $(N,t)$ by contraction of $e$
	and let $v_e$ be the vertex in $N_e$ that corresponds to the contracted edge $e$.
	As shown in the proof of Lemma \ref{lem:discriminatingN}, $(N_e,t_e)$ is a labeled
	level-1 network on $X$. By construction all vertices of $N_e$ that are contained in
	$N$ obtained the same label as in $N$, while $v_e$ obtained label
	$t(v_e)=t(u)=t(v)$. 

  We show first that $\mathscr{G}(N,t) = \mathscr{G}(N_e, t_e)$. It suffices to
 verify that, for all distinct $x,y \in X$, we have $t(\lca_N(x,y)) =
 t_e(\lca_{N_e}(x,y))$. Let $x,y \in X$ be distinct and $w = \lca_N(x,y)$. If
 $w\notin \{u,v\}$, then Lemma \ref{lem:lca-contract} implies that $w =
 \lca_{N_e}(x,y)$. Hence, in particular, $w\neq v_e$ must hold. The latter two
 arguments together with the definition of $t_e$ imply $t(\lca_N(x,y))
 =t_e(\lca_{N_e}(x,y))$. Suppose now that $w\in \{u,v\}$. By Lemma
 \ref{lem:lca-contract}, we have $\lca_{N_e}(x,y)=v_e$. Since $t(u)=t(v)$ and by
 construction of $t_e$, we have $t_e(v_e) = t(u)=t(v)$. In summary, for all
 distinct $x,y \in X$, we have $t(\lca_N(x,y)) = t_e(\lca_{N_e}(x,y))$.
 Consequently, $\mathscr{G}(N,t) = \mathscr{G}(N_e, t_e)$. The latter arguments
 can be repeated, with $(N_e,t_e)$ now playing the role of $(N,t)$, until we
 eventually obtain $(\widehat N, \widehat t)$, which completes the proof. 
\end{proof}

\begin{figure}[t]
		\begin{center}
			\includegraphics[width = .4\textwidth]{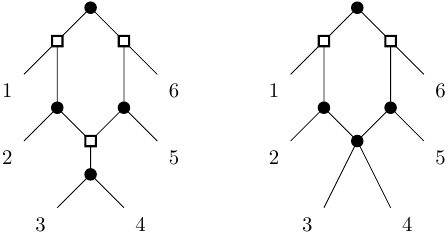}
		\end{center}
		\caption{The left network $(N,t)$ is discriminating but not least-resolved, since 
						the right network $N'$ can be obtained from $N$ by a single edge contraction
						and there is a labeling $t'$ such that $\mathscr{G}(N,t) = \mathscr{G}(N',t')$.}
		\label{fig:LRT}
\end{figure}

We emphasize that discriminating or quasi-discriminating networks are not necessarily
least-resolved, see Fig.\ \ref{fig:LRT}.
We show now that every level-1 network that contains weak cycles can be
`transformed'' into a strong level-1 network that explains the same graph by
replacing all weak cycles locally by trees. 

\begin{lemma}\label{lem:Weak-Reduce-Cycle}
Let $(N,t)$ be a level-1 network that contains $l$ cycles for which $k$ of them are
weak. Then, there is a strong level-1 network $(N',t')$ on $X$ that contains $l-k$
cycles and such that $\mathscr{G}(N',t') = \mathscr{G}(N,t)$.
\end{lemma}

\begin{proof}
Let $(N,t)$ be a level-1 network that contains $l$ cycles. Assume that $C$ is a weak
cycle of $N$. Hence, we either have that (a) $(\rho_C,\eta_C)$ is an edge of $N$, or
(b) there exists two vertices $u$ and $v$ in $N$ and edges $(\rho_C,u)$,
$(u,\eta_C)$, $(\rho_C,v)$, $(v,\eta_C)$. We continue with showing that we can
``locally replace'' $C$ by a tree to obtain a level-1 network $(N',t')$ on $X$ such
that $\mathscr{G}(N',t') = \mathscr{G}(N,t)$. In this case, $(N',t')$ contains
contains $l-1$ weak cycles and the statement follows by induction.

Consider first Case (a) and let $N'$ be the network obtained from $N$ by removing the
edge $(\rho_C,\eta_C)$, and suppressing $\rho_C$, resp., $\eta_C$ in case they have
now indegree 1 and outdegree 1 or indegree 0 and outdegree 1. Note, these are the
only vertices in $N$ that may have a different indegree and outdegree than in $N'$.
It is now easy to verify that $N'$ remains a level-1 network with $V(N') \subseteq
V(N)$. If $\rho_C$ has been suppressed in $N'$, then we define $\rho'$ as the unique
child $c$ of $\rho_C$ that is located on $C$ and distinct from $\eta_C$ which is
feasible, since $c$ was not suppressed. In all other cases, we put $\rho'\coloneqq
\rho_C$. In either case, $\rho'\in V(N')$, however, $\rho'\neq \rho_{N'}$ may be
possible. The latter arguments also imply that, for all distinct $x.y\in X$, the
vertex $\lca_N(x,y)$ is not suppressed in $N'$ whenever $\lca_N(x,y)\notin
\{\rho_C,\eta_C$\}, $x,y\in X$. Moreover, if $\lca_N(x,y)=\rho_C$, 
then there is no $v\in V(C)\setminus \{\rho_C\}$ with
	  $x\prec_N v$ and $y\prec_N v$  since then $\lca_N(x,y)\preceq_N v\prec_N \rho_N$. Hence,
at least one of
the vertices $x$ and $y$ must be incomparable to every vertex $v\in V(C)\setminus 
\{\rho_C\}$ which implies
that, after the removal of the edge $(\rho_C,\eta_C)$, vertex $\rho_C$ must have
outdegree at least two and thus, was not suppressed. By similar arguments, if
$\lca_N(x,y)=\eta_C$, then $\eta_C$ must have outdegree at least two in $N'$ and thus,
was not suppressed to obtain $N'$. Hence, in all cases the vertex $\lca_N(x,y)\in V(N)$ is
present in $N'$ for all distinct $x.y\in X$.

We show now that for all distinct $x,y \in X$, we have indeed
$\lca_{N'}(x,y)=\lca_N(x,y)$. Let $x,y\in X$, $x\neq y$ be chosen arbitrarily. If
$\lca_N(x,y)$ and $\rho_C$ are incomparable in $N$, then it is easy to verify that
$\lca_N(x,y) = \lca_{N'}(x,y)$ since we only changed the topology of $N$ ``below''
$\rho_C$. Thus, assume that $\lca_N(x,y)$ and $\rho_C$ are comparable in $N$. We
distinguish the following case: $\lca_N(x,y)\prec_N \rho_C$, $\rho_C\prec_N
\lca_N(x,y)$ and $\rho_C = \lca_N(x,y)$.

Assume first that $\lca_N(x,y)\prec_N \rho_C$. Since we only removed
$(\rho_C,\eta_C)$ and since the $\rho_C\eta_C$-path in $N$ that is distinct from the
single edge $(\rho_C,\eta_C)$ defines a $\rho'\eta_C$-path in $N'$, we still have
$\lca_{N'}(x,y)\prec_{N'} \rho'$ in case $\rho' = \rho_C$ and
$\lca_{N'}(x,y)\preceq_{N'} \rho'$ if $\rho_C$ was suppressed. Moreover, since the
topology along all vertices $v$ with $v \prec_N \rho_C$ remained unchanged, it holds
that $\lca_N(x,y) = \lca_{N'}(x,y)$.

Assume now that $\rho_C\prec_N \lca_N(x,y)$. If $x$ and $y$ are not
descendants of $\rho_C$ in $N$, then the paths from $\lca(x,y)$ to $x$ and $y$
respectively remain unchanged in $N'$. In particular, $\lca_N(x,y) = \lca_{N'}(x,y)$
holds. Otherwise, we can assume w.l.o.g.\ that $x\prec_N \rho_C$ and that $y$ and
$\rho_C$ are incomparable in $N$ and thus, $\lca_N(x,y) = \lca_N(\rho_C,y)$. Since we
have only removed the edge $(\rho_C,\eta_C)$, the vertices $y$ and $\rho'$ remain
incomparable in $N'$ and $x\prec_{N'} \rho'$. Therefore, $\lca_{N'}(x,y) =
\lca_{N'}(\rho',y) = \lca_{N}(\rho_C,y) = \lca_{N}(x,y)$. 

Finally assume that $\rho_C = \lca_N(x,y)$, in which case $\rho'=\rho_C$. 
As argued above, in this
case, at least one of the vertices $x$ and $y$ must be incomparable to every vertex
in $V(C)\setminus \{\rho_C\}$ and $\rho'=\rho_C$. It is easy to verify that $\rho'$ remains the
$\preceq_{N'}$-minimal ancestor of $x$ and $y$, since we did not create new paths in
$N'$. Hence, $\rho' = \lca_{N'}(x,y)$ and therefore, $\lca_{N}(x,y) =
\lca_{N'}(x,y)$.

In summary, for all vertices $x,y \in X$, we have $\lca_{N'}(x,y)=\lca_N(x,y)$. Let
$t'$ be the restriction of $t$ to $V(N')$, that is, $t'(v)=t(v)$ for all $v\in
V(N')\subseteq V(N)$. Since $\lca_{N'}(x,y)=\lca_N(x,y)$ and $t'$ retains the vertex
labels for all such vertices, it follows that $(N',t')$ is a labeled level-1 network
such that $\mathscr{G}(N',t') = \mathscr{G}(N,t)$ and that has one weak
cycle less than $N$.

Consider now Case (b). By Prop.\ \ref{prop:NhatN-sameGraph}, we can assume w.l.o.g.\
that $(N,t)$ is quasi-discriminating and thus, $t(u)\neq t(\rho_C)$ and $t(v)\neq
t(\rho_C)$. Let $N'$ be the network obtained from $N$ by applying the following
steps: Remove the edges $(\rho_C,u)$, $(\rho_C,v)$, $(u,\eta_C)$ and $(v,\eta_C)$;
add new vertices $\rho'$ and $w_0$; add the edges $(\rho_C, \rho')$,
$(\rho',\eta_C)$, $(\rho', w_0)$, $(w_0,u)$ and $(w_0,v)$; and suppress resulting
vertices of indegree and outdegree 1. By construction we have $V(N')\setminus \{w_0,
\rho'\} \subseteq V(N)$. Moreover, we have $\rho_C \in V(N')$ if and only if $\rho_C$
has outdegree three or more in $N$. We next define $t'$ by putting
$t'(w_0)=t(\rho_C)$, $t'(\rho')=t(v)$ and $t'(w)=t(w)$ for all $w \in V(N')\setminus
\{w_0,\rho'\}$. Note, $t'(\rho')=t(v) = t(u)$, since $(N,t)$ is quasi-discriminating.
Note that a vertex $w \in \{\rho_C, u, v\}$ is suppressed if and only if $w$ has
outdegree exactly two in $N$, while $\eta_C$ is suppressed if it has outdegree one in
$N$. All other vertices $w\in V(N)\setminus\{\rho_C, u, v, \eta_C\}$ remain vertices
of $N'$.

We next show that $\mathscr G(N',t')=\mathscr G(N,t)$. To this end, we must verify
that $t'(\lca_{N'}(x,y))=t(\lca_N(x,y))$ holds for all $x,y \in X$. Let $x,y\in X$ be
two arbitrarily chosen vertices. If $\lca_N(x,y)$ and $\rho_C$ are incomparable in
$N$, then the subnetwork of $N$ rooted at $\lca_N(x,y)$ remained unchanged in $N'$.
In particular, we have $\lca_N(x,y)=\lca_{N'}(x,y)$ and, therefore,
$t'(\lca_{N'}(x,y))=t(\lca_N(x,y))$ by definition of $t'$. By similar arguments, if
$\lca_N(x,y) \prec_N \rho_C$ with $\lca_N(x,y) \notin \{u,v\}$, then
$t'(\lca_{N'}(x,y))=t(\lca_N(x,y))$. If $ \rho_C \prec_N \lca_N(x,y)$, then the path
from $\lca_N(x,y)$ to the parent $p$ of $\rho_C$ (possibly of length 0) is preserved
in $N'$. In particular, $\lca_{N'}(x,y)=\lca_N(x,y)$ holds in this case and thus,
$t'(\lca_{N'}(x,y))=t(\lca_N(x,y))$ by definition of $t'$.

It remains to consider the cases $\lca_N(x,y) \in \{\rho_C,v,u\}$. Assume first that
$\lca_N(x,y)=\rho_C$, then two cases can occur: (i) at least one of $x,y$ is not a
descendant of $u$ and not a descendant of $v$ in $N$, and (ii) $x$ is a descendant of $u$ 
and $y$ is a
descendant of $v$ in $N$ or \emph{vice versa}. In Case (i), $\rho_C$ has outdegree at
least three in $N$, so $\rho_C \in V(N')$. By construction, $x$ and $y$ are
descendant of distinct children of $\rho_C$ in $N'$, so
$\lca_{N'}(x,y)=\rho_C=\lca_N(x,y)$. By definition of $t'$ it holds that
$t'(\lca_{N'}(x,y))=t(\lca_N(x,y))$. In Case (ii), assume w.l.o.g.\ that 
$x\prec_N u$ and $y\prec_N v$. We first remark that neither $x$
nor $y$ are descendant of $\eta_C$ in $N$, since then $\lca_{N}(x,y)\prec_N
\rho_C$; a contradiction. This means that $x$ (resp., $y$) is descendant of one child of
$u$ (resp., $v$) that is distinct from $\eta_C$. Note, that these particular
children cannot be located on $C$ since $(u,\eta_C)$ and $(v,\eta_C)$ are edges of
$C$. Hence, by construction, $x$ and $y$ are descendants of two distinct children
of $w_0$ in $N'$ and thus, $\lca_{N'}(x,y)=w_0$. Since $t'(w_0)=t(\rho_C)$ it follows
that $t'(\lca_{N'}(x,y))=t(\lca_N(x,y))$.

Finally, if $\lca_N(x,y) \in \{u,v\}$, then two cases may occur: (i') None of $x$ and
$y$ is a descendant of $\eta_C$ and (ii') exactly one of $x,y$, say $x$, is a
descendant of $\eta_C$ in $N$. In Case (i'), $\lca_N(x,y)$ has outdegree three or
more in $N$. In particular, $\lca_N(x,y)$ has outdegree two or more in $N'$, so
$\lca_N(x,y) \in V(N')$. It follows that $\lca_{N'}(x,y)=\lca_N(x,y)$ and thus,
$t'(\lca_{N'}(x,y))=t(\lca_N(x,y))$. In Case (ii'), $x$ remains a descendant of
$\eta_C$ in $N'$ if $\eta_C$ was not suppressed and, otherwise, a descendant of a
child of $\rho'$ in $N'$ that is distinct from $w_0$. Moreover, $y$ is a descendant
of $w_0$ in $N'$ . Hence, $\lca_{N'}(x,y)=\rho'$. Since $t'(\rho')=t(v)=t(u)$, we
have $t'(\lca_{N'}(x,y))=t(\lca_N(x,y))$.

In summary, $N'$ contains one weak cycle less than $N$ and
$t'(\lca_{N'}(x,y))=t(\lca_N(x,y))$ for all $x,y \in X$ which implies that $\mathscr
G(N',t')=\mathscr G(N,t)$.
\end{proof}

As a direct consequence of Lemma \ref{lem:Weak-Reduce-Cycle}, we obtain a new 
characterization of cographs.

\begin{theorem}\label{thm:WeakIffCograph}
A graph $G$ is a cograph if and only if $G$ can be explained by a weak labeled
level-1 network $(N,t)$.
\end{theorem}
\begin{proof}
If $G$ is a cograph, then $G$ can be explained by a labeled phylogenetic tree $(N,t)$
which is trivially a weak level-1 network since it does not contain cycles. 
Conversely, if $G$ can be explained by a weak labeled level-1 network $(N,t)$, then
all $l$ cycles of $N$ are weak and, by Lemma \ref{lem:Weak-Reduce-Cycle}, there is a
strong level-1 network $(N',t')$ that contains no cycles and still explains $G$.
Hence, $(N',t')$ is a tree and thus, $G$ is a cograph.
\end{proof}

Based on the latter results we can provide an additional characterization of 
pseudo-cographs.

\begin{theorem}\label{thm:CharPsC-Network-Cycle}
 A graph $G$ is a pseudo-cograph if and only if $|V(G)|\leq 2$ or 
 $G$ can be explained by a level-1 network $(N,t)$ that contains precisely
 one cycle $C$ such that $\rho_C = \rho_N$ and $\child_N(\eta_C)=\{x\}$ with $x\in 
 L(N)$.
\end{theorem}
\begin{proof}
	Assume first that $G$ is a pseudo-cograph. If $|V(G)|\leq 2$ there is nothing to
	show. Assume that $|V(G)|\geq 3$. In this case there is a vertex $v$ and subgraphs
	$G_1$ and $G_2$ of $G$ such that $G$ is a $(v,G_1,G_2)$-pseudo-cograph. By Prop.\
	\ref{prop:PsC-l1N}, $G$ can be explained by the level-1 network
	$(N(v,G_1,G_2),t(v,G_1,G_2))$. This network contains precisely one cycle $C$ with
	$\rho_C=\rho_N$, and $\child_N(\eta_C)=\{v\}$ and $v\in L(N)$.
	
	Conversely, if $|V(G)|\leq 2$, then $G$ is a pseudo-cograph by definition. Assume
	that $G$ can be explained by a level-1 network $(N,t)$ that contains precisely one
	cycle $C$ with $\rho_C=\rho_N$ and $\child_N(\eta_C)=\{x\}$, $x\in 
    L(N)$. If $C$ is weak, then $N$ is weak and Thm.\
	\ref{thm:WeakIffCograph} implies that $G$ is a cograph. By Lemma
	\ref{lem:cographPscograph}, $G$ is a pseudo-cograph. Thus, assume that $C$ is
	strong. Since $N$ is level-1, there is a unique hybrid-vertex $\eta_C$ in $C$.
	Consider the two sides $P^1$ and $P^2$ of $C$. Since $C$ is strong it holds, in
	particular, that $1\leq |V(P^1)\setminus\{\rho_C,\eta_C\}|$ and $1\leq
	|V(P^2)\setminus\{\rho_C,\eta_C\}|$ where at least one inequality is strict.  
	Let $u$ be the child of $\rho_C$ that is
	contained in $P^1$ and let $x$ be the leaf that is the unique children of $\eta_C$.
	Consider the subnetwork $(N_1,t_1)$ that is rooted at $u$ and where $\eta_C$ is
	suppressed. Since $N$ contains only one cycle it follows that $(N_1,t_1)$ is a
	labeled (phylogenetic) tree. Hence, $G_1 \coloneqq G[L(N_1)]$ must be a cograph.
	Note that $V(G_1) = L(N_1)$ and thus, $|V(G_1)|\geq 2$ since $|V(P^1)\setminus
	\{\rho_C\}|\geq 2$ and each vertex in $P^1-\rho_C$ must be an ancestor of at least
	one leaf in $L(N_1)$. Consider now the subnetwork $(N_2,t_2)$ that is rooted at
	$\rho_C$ and that contains none of the vertices of $N_1$ except $\eta_C$ and $x$
	and where $\eta_C$ is finally suppressed and $\rho_C$ is removed in case it has
	outdegree 1 in $N_2$. Roughly speaking, $(N_2,t_2)$ consists of the other side
	$P^2$ of $C$ and all vertices that are descendants of the vertices in $P^2$ except
	for $u$ an its descendants that are different from $\eta_C$ and $x$. By similar
	arguments as before, $(N_2,t_2)$ is a labeled (phylogenetic) tree that explains the
	cograph $G_2 \coloneqq G[L(N_2)]$ with $|V(G_2)|\geq 2$. Hence, $|V(G_1)|,|V(G_2)|>
	1$ and, by construction, $V(G_1)\cup V(G_2) = V(G)$ and $V(G_1)\cap V(G_2) =
	\{x\}$. Thus, (F1) is satisfied. Since both $G_1$ and $G_2$ are cographs (F2)
	holds. Moreover, for every $y\in V(G_1)\setminus \{x\}$ and $z\in
	V(G_2)\setminus\{x\}$ it holds that $\lca_N(y,z)=\rho_N$. Hence, depending on the
	label $t(\rho_N)$ it either holds that $\{y,z\}\in E$ or $\{y,z\}\notin E$ for all
	$y\in V(G_1)\setminus \{x\}$ and $z\in V(G_2)\setminus\{x\}$. Thus, $G-x$ is either
	the join or disjoint union of $G_1-x$ and $G_2-x$ which implies that (F3) is
	satisfied. Hence, $G$ is a $(x,G_1,G_2)$-pseudo-cograph.
\end{proof}

Note that not all graphs that can be explained by level-1 networks are
pseudo-cographs, see Fig.\ \ref{fig:non-pseudo-cograph}. Furthermore, observe that
not every level-1 network that explains a pseudo-cograph contains a cycle $C$ such
that $\rho_C = \rho_N$ and $\child_N(\eta_C)=\{x\}$ with $x\in L(N)$ as shown in
Fig.\ \ref{fig:gamma}. Thus,  there can be distinct
quasi-discriminating level-1 networks that explain the same graph $G$. Non-uniqueness
of labeled level-1 networks that explain a given graph is supported by further
examples: Consider a network $N$ that contains a hybrid-vertex $\eta$ that has
outdegree 1. In this case, $\eta$ cannot be the lowest common ancestor of any pair of
leaves. Hence, the label of $\eta$ can be chosen arbitrarily. Let $v$ denote the
unique child of $\eta$ in a quasi-discriminating level-1 network $(\widehat N,
\widehat t)$ and assume that $\widehat{t}(\eta)\neq \widehat{t}(v)$ and that $v$ is
not a leaf. In this case, there is another labeling $t$ of $\widehat N$ that keeps
all vertex labels except for $\eta$ where we put $t(\eta)=t(v)$. In this case,
$(\widehat N, t)$ is not quasi-discriminating any more and we can contract the edge
$e=(\eta,v)$ to obtain a different quasi-discriminating level-1 network $(N_e,t_e)$
for which $\mathscr{G}(N,t) = \mathscr{G}(N_e,t_e)$. Further examples (incl.\
cographs) that show that different labeled level-1 networks can explain the same
graph are provided in Fig.\ \ref{fig:nonUniqueN}. 

\begin{figure}
		\begin{center}
			\includegraphics[scale = .9]{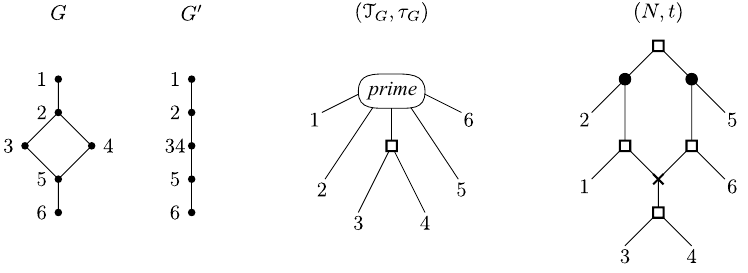}
		\end{center}
		\caption{Shown is a graph $G$ that is not a pseudo-cograph and thus, not a
		         polar-cat. To see this, assume for contradiction that $G$ is a
		         $(v,G_1,G_2)$-pseudo-cograph. There are only two choices for $v$ such
		         that $G-v$ is a cograph, namely $v\in \{2,5\}$. By Lemma
		         \ref{lem:ConnComp-v}, $G_1-v$ must contain one component of $G-v$ and
		         $G_2-v$ the other one. In this case, however, $G_1$ or $G_2$ is not a
		         cograph; a contradiction to (F2). The graph $G' \simeq G/\Mmax(G)$ is a
		         pseudo-cograph and, in particular, a polar-cat that is obtained from $G$
		         by identifying the two vertices $3$ and $4$; see Section
		         \ref{sec:general} for more details. Replacing the prime vertex in the
		         modular-decomposition tree of $(\MDT_G,\tau_G)$ of $G$ by a strong
		         quasi-discriminating elementary network that explains $G'$ yields the
		         level-1 network $(N,t)$ that explains $G$.
		         The only modules of $G$ that are distinct from the singletons
				are $\{3,4\}$ and $M=V(G)$. While $\{3,4\}$ is a parallel module, $M$ is a
				prime. This together with $G'\in \protect{\PolarCat}$ implies that $G\in
				\protect{\PrimeCat}$.
		         }
		\label{fig:non-pseudo-cograph} 
\end{figure}

\begin{figure}
		\begin{center}
			\includegraphics[scale = .7]{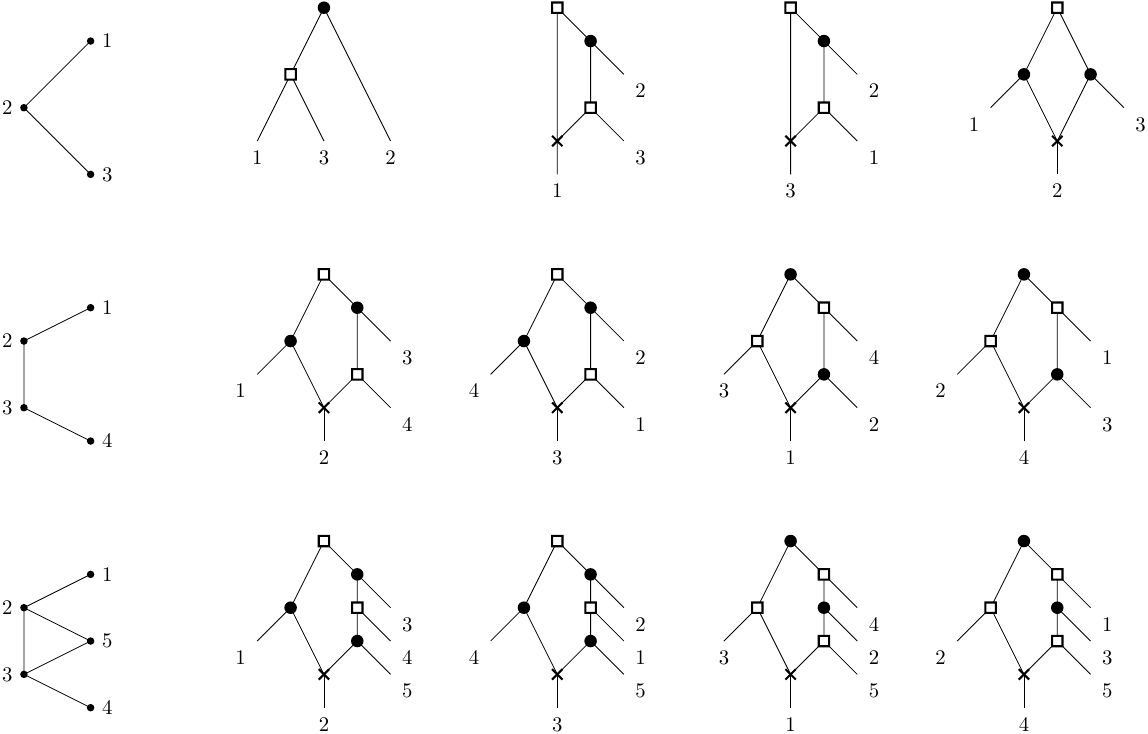}
		\end{center}
		\caption{Three different graphs that can be represented by non-isomorphic
		         quasi-discriminating elementary level-1 networks.
		         By Thm.\ \ref{thm:CharPsC-Network-Cycle}, all graphs shown here are 
		         pseudo-cographs. Moreover, the networks in the upper row are weak, while  
	         		all other networks are strong. By Lemma \ref{thm:WeakIffCograph}, the 
	         		graph in the upper row must be a cograph, while
         			the other two graphs are polar-cats (according to Thm.\ 
         			\ref{thm:CharPCPC}).	
         }
		\label{fig:nonUniqueN}
\end{figure}

\begin{lemma}\label{lem:WeakElementary}
	A graph $G$ can be explained by a weak quasi-discriminating elementary network 
	if and only if $G$ is a  caterpillar-explainable cograph and $|V(G)|\geq 2$. 
\end{lemma}
\begin{proof}
	Suppose that $G=(X,E)$ can be explained by a weak quasi-discriminating elementary
	network $(N,t)$ on $X$. By Thm.\ \ref{thm:WeakIffCograph}, $G$ is a cograph. Since
	$N$ is elementary it contains a cycle and, in particular, more than one leaf.
	Hence, $|X|\geq 2$. Let $C$ be the unique cycles in $N$ of length $|X|+1$ and $P^1$
	and $P^2$ be the sides of $C$. Since $N$ is weak, $C$ must be weak. Hence,
	either one of $P^1$ and $P^2$ consists of $\rho_C$ and $\eta_C$ only or both $P^1$
	and $P^2$ contain only one vertex that is distinct from $\rho_C$ and $\eta_C$.

	Assume first that one of $P^1$ and $P^2$ consists of $\rho_C$ and
	$\eta_C$ only. Now remove $\rho_C$ and its two incident edges from $(N,t)$,
	suppress $\eta_C$ and keep the labels of all remaining vertices to obtain a tree
	$(T,t')$ on $X$. It is straightforward to verify that $(T,t')$ satisfies $\lca_N(x,y) = \lca_T(x,y)$ for all distinct
	$x,y\in X$ and therefore, $G = \mathscr{G}(N,t) = \mathscr{G}(T,t')$. By
	construction and since $N$ is elementary, $T$ is a caterpillar. Moreover,
	since $\eta_C$ was suppressed in $T$ and since $t$ is a quasi-discriminating
	labeling of $N$, the labeling $t'$ must be a discriminating labeling of $T$. Hence,
	$(T,t')$ is the unique discriminating cotree of $G$. Thus, $G = (X,E)$ is
	caterpillar-explainable. Assume now that both $P^1$, resp., $P^2$ contain only one
	vertex $u$, resp., $v$ distinct from $\rho_C$ and $\eta_C$. Hence, $C$ must be
	cycle of length $4 = |X|+1$ and thus, $|X|=3$. Since $(N,t)$ is
	quasi-discriminating it must hold $t(u)=t(v) \neq t(\rho_N)$. It is now easy to
	verify that $G$ is either isomorphic to a $P_3$ or $K_2+K_1$ whose discriminating
	cotree is a caterpillar. Hence, $G = (X,E)$ is caterpillar-explainable.
	
	Suppose now that $G = (X,E)$ is a caterpillar-explainable cograph and $|X|\geq 2$.
	Let $(T,t)$ be its unique discriminating cotree which is, by definition, a
	caterpillar. Since $|X|\geq 2$ the root of $T$ is distinct from the leaves in $X$
	and thus, an inner vertex. Let $\rho_T$ be the root of $T$, $x\in X$ be one of the
	leaves that is part of the unique cherry in $T$ and $u$ be the parent of $x$ in
	$T$. To construct a weak quasi-discriminating elementary network $(N,t')$ take
	$(T,t)$, remove the edge $(u,x)$, add a vertex $\eta_N$ (the hybrid-vertex in $N$),
	add a vertex $\rho_N$ (the root of $N$) and add edges $(u,\eta_N)$, $(\eta_N,x)$,
	$(\rho_N,\rho_T)$ and $(\rho_N,\eta_N)$. It is straightforward to see that $N$
	consists of a single cycle $C$ for which one side consists of $\rho_N$ and $\eta_N$
	only. Hence, $N$ is a weak elementary network. To obtain a quasi-discriminating
	network, we define $t'\colon V^0(N)\to \{0,1\}$ by putting $t'(v) = t(v)$ for all
	$v\in V^0(N)\setminus \{\rho_N,\eta_N\}$ and put $t'(\rho_N) \neq t(\rho_T)$ and
	choose $t'(\eta_N)\in \{0,1\}$ arbitrarily. It is easy to verify that for all
	distinct $x,y\in X$ the vertex $\lca_N(x,y)$ is located at the non-empty side of
	$C$ and in particular, $\lca_N(x,y)\neq \rho_N$. In other words, for all distinct
	$x,y\in X$ the vertex $\lca_N(x,y)$ is located in the copy of $T$. Hence, $G =
	\mathscr{G}(N,t')$. Therefore, $G$ can be explained by a weak quasi-discriminating
	elementary network.
\end{proof}

\section{Primitive Graphs, Polar-Cats and Elementary Networks}
\label{sec:PolarCat}

In this section, we first characterize cographs that are caterpillar-explainable
		and graphs that are polar-cats. These characterizations are based on particular 
		vertex-orderings. We then show that
		polar-cats are precisely the primitive graphs that can be explained by
		strong quasi-discriminating elementary networks.

\begin{lemma}\label{lem:cograph-caterpillar}
A graph $G=(X,E)$, $|X|=n\geq 1$ can be explained by a discriminating caterpillar
tree $(T,t)$ on $X$ if and only if we can index the vertices of $X$ from $1$ to $n$
such that one of the following holds:
\begin{itemize}[noitemsep]
\item[(a)] $G$ is connected and the edges of $G$ are $\{x_i,x_j\}$ with $1 \leq i<j
           \leq n$ and $i$ being odd.\smallskip
\item[(b)] $G$ is disconnected and the edges of $G$ are $\{x_i,x_j\}$ with $1 \leq
           i<j \leq n$ and $i$ being even.
\end{itemize}
In this case, the indexing of the vertices in $X$ is unique up to permutation of the
indices $n-1$ and $n$. In particular, the vertices $x_{n-1}$ and $x_n$ form the
cherry in the underlying caterpillar.
\end{lemma}
\begin{proof}
It is an easy task to verify that the statements are satisfied in case $n\leq 2$.
Hence, let us assume that $n\geq 3$.

Assume first that $G$ can be explained by a discriminating caterpillar tree
$(T,t)$ on $X$. W.l.o.g.\ assume that the inner vertices are indexed in such a
way that the index $i$ in $v_i$ implies that $\dist_T(\rho_T,v_i)= i-1$. Hence,
$\rho_T = v_1$ and, since there are $|X|=n$ leaves and $T$ is a caterpillar,
$V^0(T) = \{v_1,\dots,v_{n-1}\}$. We index the leaves in such a way that vertex
$v_i$ is adjacent to $x_i$, $1\leq i\leq n-2$ while $v_{n-1}$ is adjacent to
$x_{n-1}$ and $x_n$. We consider now the two different cases: (i) $t(v_1)=1$ and
(ii) $t(v_1)=0$. In Case (i) it holds that $t(v_1)=1$ and thus, $G$ is
connected. Moreover, $t(v_i)=1$ if and only if $i$ is odd ($1 \leq i \leq n-1$),
since $(T,t)$ is discriminating. By definition of caterpillars and the choice of
respective vertex indices, we have $\lca(x_i,x_j)=v_i$, $1 \leq i<j \leq n$. The
latter two arguments imply that $\{x_i,x_j\}$ is an edge of $G$ if and only if
$i$ is odd, $1 \leq i \leq n$. Hence, Condition (a) is satisfied. Now consider
Case (ii). The fact that $t(v_1)=0$ implies that $G$ is disconnected. By
analogous arguments and by interchanging the term ``odd'' by ``even'' one shows
that Condition (b) is satisfied as well.

Conversely, assume that we can index the vertices of $X$ from $1$ to $n$ such that
one of the Conditions (a) and (b) holds. We show that there exists a discriminating
caterpillar tree $(T,t)$ on $X$ explaining $G$.

Assume first that Condition (a) is satisfied. Let $M\subseteq X$ be a module of
$G$ with $|M|\geq 2$. Let $x_i,x_j\in M$, $1\leq i<j\leq n$. In this case, every
vertex $x_k\in X$ with $i<k<j$ must also be contained in $M$ . 
To see this, assume for contradiction that
there exists some $k$ with $i < k < j$ such that $x_k \notin M$.
Let $k$ be the smallest such integer for which $x_k \notin M$. 
Hence, $x_{k-1} \in M$. By (a), exactly one of
the pairs $\{x_{k-1},x_k\}$, $\{x_k,x_j\}$ is an edge of $G$. Since $x_{k-1},
x_j \in M$ and $x_k \notin M$, it follows that $M$ is not a module; a
contradiction. Thus, $\{x_k \mid i\leq k\leq j\}\subseteq M$ with $1\leq
i<j\leq n$. Since $|M|\geq 2$, there are vertices $x_{j},x_{j-1}$ in $M$ and, by
definition of the edges in Case (a), one of them is adjacent to $x_n$ while the
other is not. This implies that $x_n\in M$. Taking the latter arguments
together, every putative module of size at least two must be of the form
$M_i\coloneqq \{x_k \mid i\leq k\leq n\}$, $1\leq i\leq n-1$. By definition,
every $x\in M_i$ is adjacent to $x_{\ell}\in X\setminus M$ if $\ell$ is odd and
non-adjacent of $\ell$ is even. Hence, indeed $M_i$ is module for all $i\in
\{1,\dots,n\}$. It is easy to see that $M_k\cap M_{k'} \in \{M_k,M_{k'}\}$ and
thus these modules do not overlap. Hence $\MDstrong(G) = \{M_1=X,M_2,\dots,
M_n\} \cup (\cup_{x\in X} \{x\})$. Moreover, $G[M_i]$ is disconnected if $i$ is
even (since $x_i\in M_i$ is not connected to any $x_j\in M_i\setminus \{x_i\}$)
and $\overline{G[M_i]}$ is disconnected if $i$ is odd (since $x_i\in M_i$ is
connected to all $x_j\in M_i\setminus \{x_i\}$). Hence, none of the modules
$M_i$ is a prime module. Its now easy to verify that the MDT for $G$ must be a
discriminating caterpillar $(T,t)$. In particular, $M_{n-1} = \{x_{n-1},x_{n}\}$
is the smallest non-trivial module and thus $x_{n-1}$ and $x_n$ form a cheery in
$T$. By similar arguments one shows that $G$ is explained by a discriminating
caterpillar $(T,t)$ in case Condition (b) is satisfied. 
\end{proof}

\begin{proposition}\label{prop:polcat-pc}
Let $G=(X,E)$ be a graph with $|X|\geq 4$. Then, $G$ is a polar-cat if and only if
there exists a vertex $v \in X$ and two ordered sets $Y=\{y_1, \ldots, y_{k-1},
y_{k}=v\}$ and $Z=\{z_1, \ldots, z_{m-1}, z_{m}=v\}$, $k,m \geq 2$ such that $Y \cap
Z=\{v\}$, $Y \cup Z=X$, and one of the following conditions hold:
\begin{itemize}
	\item[(a)] $G[Y]$ and $G[Z]$ are connected and the edges of $G$ are 
		\begin{itemize}[noitemsep, nolistsep]
			\item[(I)] $\{y_i,y_j\}$,		where $1 \leq i<j \leq k$ and $i$ is odd; and
			\item[(II)] $\{z_i,z_j\}$, where $1 \leq i<j \leq		m$ and $i$ is odd.
		\end{itemize}
	\item[(b)] $G[Y]$ and $G[Z]$ are disconnected and the edges of $G$ are 
		\begin{itemize}[noitemsep, nolistsep]
			\item[(I)] $\{y_i,y_j\}$, where $1 \leq i<j \leq k$ and $i$ is even; and
			\item[(II)] $\{z_i,z_j\}$, where $1 \leq i<j \leq m$ and $i$ is even; and 
			\item[(III)] $\{y,z\}$ for all $y\in Y\setminus\{v\}$ and $z\in Z\setminus\{v\}$ 
	  \end{itemize}
	\end{itemize}
In this case, $G$ is a $(v,G[Y],G[Z])$-polar-cat
\end{proposition}
\begin{proof}
For the \emph{only-if}-direction, 
assume first that $G$ is a $(v,G_1,G_2)$-polar-cat. In this case, $G_1$ and $G_2$ are
caterpillar-explainable such that $v$ is part of a cherry. We can put $Y=V(G_1)$ and
$Z=V(G_2)$. Since $G_1$ and $G_2$ intersect only in vertex $v$ it holds that $G_1 =
G[Y]$ and $G_2=G[Z]$. By Lemma \ref{lem:cograph-caterpillar} and since $v$ is part of
a cherry in the cotrees of $G_1$ and $G_2$, we can order $Y$ and $Z$ such that
$Y=\{y_1, \ldots, y_{k-1}, y_{k}=v\}$ and $Z=\{z_1, \ldots, z_{m-1}, z_{m}=v\}$ and
(a.I) and (a.II), resp., (b.I) and (b.II) are satisfied. Moreover, since $G$ is
polarizing, there are two cases either (A) $G[Y]$ and $G[Z]$ are connected and $G-v =
(G_1-v)\union (G_2-v)$ or (B) $G[Y]$ and $G[Z]$ are disconnected and $G-v =
(G_1-v)\join (G_2-v)$. In Case (A), all edges of $G$ are contained in $G_1$ and
$G_2$ and are, thus, precisely the edges of $G$ as
specified in (a).  In Case (B), $G$ must contain
all edges as specified in (b.I) and (b.II) and, in addition, all edges $\{y,z\}$ with
$y\in Y\setminus\{v\}$ and $z\in Z\setminus\{v\}$ as specified in (b.III) since $G-v 
=(G_1-v)\join (G_2-v)$. Hence, Condition (b) must hold.

For the \emph{if}-direction, assume that there is a vertex $v\in X$ and subsets 
$Y,Z\subseteq X$ as
specified in the statement of the lemma.
We put $G_1=G[Y]$ and $G_2=G[Z]$. We
start with showing that $G$ is a $(v,G_1,G_2)$-pseudo-cograph. (F1) holds by
definition of the vertex sets $V(G_1)=Y$ and $V(G_2)=Z$. Moreover, the edges of
$G_1=G[Y]$ satisfy (a.I) in case $G_1$ is connected and (b.I) in case $G_1$ is
disconnected. Lemma \ref{lem:cograph-caterpillar} implies that $G_1$ is a
caterpillar-explainable cograph such that $v$ is part of a cheery. By similar
arguments, the latter holds also for $G_2$ and thus, (F2) is satisfied. To see that
(F3) holds, note that in case (a) there are no edges between distinct vertices $y\in
Y\setminus\{v\}$ and $z\in Z\setminus\{v\}$ and thus, $G-v = (G_1-v)\union (G_2-v)$.
In case (b), all $y\in Y\setminus\{v\}$ and $z\in Z\setminus\{v\}$ are adjacent and
thus, $G-v = (G_1-v)\join (G_2-v)$. Hence, (F3) is satisfied. Therefore, $G$ is a
pseudo-cograph.

		It remains to show that $G$ is polarizing and cat. As argued above, if (a) holds,
then $G-v = (G_1-v)\union (G_2-v)$ and thus, $G-v$ is disconnected, while, by
definition, $G[Y]=G_1$ and $G[Z]=G_2$ are connected. Similarly, in case (b) $G-v =
(G_1-v)\join (G_2-v)$ is connected while $G[Y]=G_1$ and $G[Z]=G_2$ are disconnected.
Thus, $G$ is polarizing. As argued above, $G_1$ and $G_2$ are caterpillar-explainable
cograph such that $v$ is part of a cherry. This together with the fact that $|X|\geq
4$ implies that $G$ is cat. In summary, $G$ is a polar-cat. 
\end{proof}

Note that the property of being polar-cat is not  hereditary since the property of
subgraphs being caterpillar-explainable is not hereditary in general. However,
polar-cats are closed under complementation. 
\begin{lemma}\label{lem:PolarCat-complement-CC}
	A graph $G$ is a $(v,G_1,G_2)$-polar-cat if and only if $\overline{G}$ is a
	$(v,\overline{G_1},\overline{G_2})$-polar-cat. Moreover, if $G$ is a polar-cat,
	then the number of connected components in the disconnected graph in $\{G-v,
	\overline{G-v}\}$ is exactly two. 
\end{lemma}
\begin{proof}
	Let $G$ be a $(v,G_1,G_2)$-polar-cat. Hence, $G$ is, in particular, a
	$(v,G_1,G_2)$-pseudo-cograph. By Lemma \ref{lem:complement}, $\overline{G}$ is a
	$(v,\overline{G_1},\overline{G_2})$-pseudo-cograph. Note that $G-v$ is the disjoint
	union (resp., join) of $G_1-v$ and $G_2-v$ if and only if
	$\overline{G-v}=\overline{G}-v$ is the join (resp., disjoint union) of
	$\overline{G_1-v}=\overline{G_1}-v$ and $\overline{G_2-v}=\overline{G_2}-v$. In
	other words, $G-v$ is disconnected (resp., connected) if and only if
	$\overline{G}-v$ is connected (resp., disconnected). Now consider $G_1$ and $G_2$.
	Since $G$ is polarizing and $G_1,G_2$ are cographs, both $G_1$ and $G_2$ are
	connected (resp., disconnected) if and only if $\overline{G_1}$ and
	$\overline{G_2}$ are disconnected (resp., connected). It is now easy to verify that
	$G$ is a $(v,G_1,G_2)$-polar-cat if and only if $\overline{G}$ is a
	$(v,\overline{G_1},\overline{G_2})$-polar-cat.
	
	We continue with proving the second statement. By the latter arguments, we can
	assume w.l.o.g.\ that $G-v$ is disconnected with set of connected components
	$\mathfrak{C}$. Since $G$ is polarizing and $G-v$ is disconnected, it follows that
	$G_1$ must be connected. By Prop.~\ref{prop:polcat-pc}~(a), there exists an
	ordering $y_1, \ldots, y_{k-1}, y_k=v$, $k=|V(G_1)|$ of the elements of $V(G_1)$
	such that, for all $1 \leq i<j \leq k$, $\{y_i,y_j\}$ is an arc of $G_1$ if and
	only if $i$ is odd. In particular, vertex $y_1$ of $G_1$ must be adjacent to all
	vertices of $G_1-y_1$. Since, $v=y_k \neq y_1$, it follows that $G_1 -v$ is
	connected. In particular, $G_1-v$ is a connected component of $G-v$. By analogous
	argumentation, $G_2 -v$ is a connected component of $G-v$. Hence, $G-v$ has exactly
	two connected components, $G_1-v$ and $G_2-v.$
\end{proof}

Note that the converse of the second statement in Lemma
\ref{lem:PolarCat-complement-CC} is not satisfied in general. To see this, consider
the $(v,G_1,G_2)$-pseudo-cograph $G$ that consists of two disjoint edges $\{a,b\}$ and $\{x,y\}$.
In fact, $G$ is a cograph and for every $v\in \{a,b,x,y\}$ the graph $G-v$ has
precisely two connected components. However, $G$ is disconnected and at least one
of $G_1$ and $G_2$ must be disconnected as well. Thus $G$ is not
polarizing and, therefore, not a polar-cat.

We provide now a simple characterization of polar-cats.
\begin{theorem} \label{thm:StrongElementary}  
	A graph $G$ is a polar-cat if and only if $G$ can be explained by a strong
	elementary quasi-discriminating network.
\end{theorem}
\begin{proof}
	For the \emph{if}-direction, let $G =(X,E)$ be a graph that can explained by a
	strong quasi-discriminating elementary network $(N,t)$ on $X$ with root $\rho_N$
	and hybrid-vertex $\eta_N$. Since $N$ is strong, its underlying unique cycle $C$
	consists of sides $P^1$ and $P^2$ such that both contain at least one vertex
	distinct from $\rho_N$ and $\eta_N$, and at least one of $P^1$ and $P^2$ must
	contain two vertices distinct from $\rho_N$ and $\eta_N$. The latter implies that
	$|X|\geq 4$. 
	
	In the following let $i\in \{1,2\}$ and $v\in X$ be the unique leaf adjacent to
	$\eta_N$. Denote with $X_i\subseteq X$ the leaves that are adjacent to the vertices
	in $P^i$ and let $G_1=G[X_1]$ and $G_2=G[X_2]$. We next show that $G$ is a
	$(v,G_1,G_2)$-pseudo cograph. We first remark that $X_1\cap X_2 =\{v\}$, $X_1\cup
	X_2 = X$ and, by the aforementioned arguments, $|X_1|,|X_2|>1$. Hence Condition
	(F1) is satisfied. Consider the tree $T_i$, $i\in \{1,2\}$ that is induced by the
	vertices $P^i -\rho_N$ and its adjacent leaves and for which the
	vertex $\eta_N$ is suppressed. Put $t_i(v) \coloneqq t(v)$ for all $v\in V(T_i)$ to
	obtain a vertex-labeling $t_i\colon V(T_i)\to\{0,1\}$. Since $(N,t)$ is
	quasi-discriminating and the vertex $\eta_N$ is suppressed, $(T_i,t_i)$ is a
	discriminating tree. In particular, $(T_i,t_i)$ is a caterpillar
	and $v$ is part of a cherry in $T_i$. By construction,
	for all $x,y\in X_i$ we have $\lca_N(x,y) = \lca_{T_i}(x,y)$ and thus,
	$\mathscr{G}(T_i,t_i)$ is precisely the graph induced by $X_i$, i.e.,
	$\mathscr{G}(T_i,t_i)=G_i$. Since $G_i$ is explained by $(T_i,t_i)$ it follows that
	$G_i$ is a cograph and thus, Condition (F2) is satisfied. 
 To see that (F3) is satisfied, consider
	the graph $G-v$. Since $G$ is explained by $(N,t)$ and since for all $x\in
	X_1\setminus \{v\}$ and $y\in X_2\setminus \{v\}$ we have $\lca_N(x,y) = \rho_N$ it
	holds that $\{x,y\}\in E$ (resp., $\{x,y\}\notin E$) if and only if $t(\rho_N)=1$
	(resp., $t(\rho_N)=0$) if and only if $G-v$ is the join (resp., disjoint union) of
	$G_1-v$ and $G_2-v$. Therefore, (F3) is satisfied. Moreover, since $(N,t)$ is
	quasi-discriminating and $\rho_{T_1}$ and $\rho_{T_2}$ are the unique children of
	$\rho_N$ in $N$, we have $t(\rho_N)\neq t_{1}(\rho_{T_1})$ and $t(\rho_N)\neq
	t_{2}(\rho_{T_2})$ and thus, $t_{1}(\rho_{T_1}) = t_{2}(\rho_{T_2})$. Hence, if
	$G-v$ is the join (resp., disjoint union) of $G_1-v$ and $G_2-v$, then
	$t_{1}(\rho_{T_1}) = t_{2}(\rho_{T_2}) = 0$ (resp., $t_{1}(\rho_{T_1}) =
	t_{2}(\rho_{T_2}) = 1$) and $G_1$ and $G_2$ are both disconnected (resp.,
	connected). Thus, $G$ is polarizing. 
	Moreover, since $(T_1,t_1)$ and $(T_2,t_2)$
			are discriminating caterpillars,  $G_1$ and $G_2$ are caterpillar-explainable
			cographs. The latter together with the fact that
			$|X|\geq 4$ and that $v$ is part of a cherry in $T_1$ and $T_2$  
			implies that $G$ is a cat.
	In summary $G$ is a $(v,G_1,G_2)$-polar-cat.

	For the \emph{only-if}-direction, let $G=(X,E)$ be a $(v,G_1,G_2)$-polar-cat with
	$G_1=(X_1,E_1)$ and $G_2 = (X_2,E_2)$. Since $G$ is cat we have $|X|\geq 4$.
	Moreover, by definition of pseudo-cographs, $X_1\cap X_2 =\{v\}$, $X_1\cup X_2 = X$
	and $|X_1|,|X_2|>1$. Let $(T_1,t_1)$ and $(T_2,t_2)$ be the unique discriminating
	caterpillars that explain $G_1$ and $G_2$, respectively. Consider now the labeled
	level-1 network $(N,t)$ with $N\coloneqq N(v,G_1,G_2)$ and $t\coloneqq
	t(v,G_1,G_2))$ as specified in Def.\ \ref{def:prop:PsC-l1N}. Indeed, by Prop.\
	\ref{prop:PsC-l1N}, $(N,t)$ is a labeled level-1 network that explains $G$. Since
	both $T_1$ and $T_2$ are caterpillars for which $v$ is cherry, the unique cycle $C$
	in $N$ consists of $\rho_N, \eta_N$ and all inner vertices of $T_1$ and $T_2$. It
	is now easy to see that $(N,t)$ is elementary and since $|X_1|,|X_2|>1$ and $|X_1|
	+ |X_2| = |X| + 1 \geq 5$, $N$ must be strong. Since by assumption, $G_1$ and $G_2$
	are connected (resp., disconnected) while $G-v$ is the disjoint union (resp.\ join)
	of $G_1-v$ and $G_2-v$ we have, by construction of $t$, $t(\rho_N) \neq
	t_1(\rho_{T_1}) = t_2(\rho_{T_2})$. Since $(T_1,t_1)$ and $(T_2,t_2)$ are
	discriminating and $t(\rho_N) \neq t_1(\rho_{T_1}) = t_2(\rho_{T_2})$, the labeled
	network $(N,t)$ is quasi-discriminating. Hence, $(N,t)$ is a strong elementary
	quasi-discriminating network that explains the polar-cat $G$. 
\end{proof}

Note that an elementary network is either weak or strong. This together with Lemma 
\ref{lem:WeakElementary} and Thm.~\ref{thm:StrongElementary} 
implies
\begin{theorem}\label{thm:CharPolarCatPC}
	A graph $G$ can be explained by a quasi-discriminating elementary network
	if and only if $G$ is a caterpillar-explainable cograph on at least two vertices
	or a polar-cat.
\end{theorem}

To recall, a graph is primitive it if it has at
	least four vertices and contains only trivial modules, i.e., 
	 only the singletons and $X$.
In the following, we show that polar-cats are precisely the primitive graphs
that can be explained by a level-1 network. To this end, we need the following
\begin{lemma}\label{lem:primitive-1}
If $G$ is a polar-cat, then $G$ is primitive and there is a level-1 network $(N,t)$
that explains $G$. 
\end{lemma}
\begin{proof}
Let $G=(X,E)$ be a polar-cat. By Thm.~\ref{thm:StrongElementary}, $G$ can be
explained by a strong quasi-discriminating elementary network $(N,t)$ on $X$ and,
thus, by a level-1 network.

Therefore, it remains to show that $G$ is primitive. To this end, let $(N,t)$ be a
strong quasi-discriminating elementary network on $X$ that explains $G$. 
We emphasize that application of Prop.~\ref{prop:polcat-pc} shows that for every integer 
$n\geq 4$
there is a polar-cat on $n$ vertices. Hence, in order to show that $G$ is primitive,
we can proceed by induction on $|X|$. In the following, let $C\subseteq N$ be the 
underlying unique cycle and $P^1$
and $P^2$ be its sides. As base case, let $|X|=4$. Since $N$ is strong, one side of
$C$ must contain one and the other side two vertices that are distinct from $\eta_N$
and $\rho_N$. There are precisely two possible quasi-discriminating labelings $t_1$
and $t_2$ of $N$ up to the choice of the labeling of the unique hybrid-vertex 
$\eta_N$; cf.\
Fig. \ref{fig:nonUniqueN}. It is easy to see that for both labelings $t_1$ and $t_2$,
we have $G\simeq \mathscr{G}(N,t_1) \simeq\mathscr{G}(N,t_2) \simeq P_4$ and thus,
$G$ is primitive. Assume now that $G=(X,E)$ is primitive for all $X$ with $4\leq
|X|<n$ vertices. 
 
Let $G=(X,E)$ be a $(v,G_1,G_2)$-polar-cat with $|X|=n>4$ that is explained by the
strong quasi-discriminating elementary network $(N,t)$. Put $P^{1-}_{-} \coloneqq
P^1\setminus \{\rho_N,\eta_N\}$ and $P^2_{-} \coloneqq P^2\setminus
\{\rho_N,\eta_N\}$. Since $N$ is strong, we can assume w.l.o.g.\ that $1\leq
|P^{1}_{-}|\leq |P^{2}_{-}|$. Moreover, since $N$ is strong and $|X|>4$, we have
$2\leq |P^{2}_{-}|$. Let $x'\in P^2$ be adjacent to unique hybrid-vertex
$\eta_N$ and $x\in X$ be the leaf
adjacent to $x'$ in $N$. Now, remove $x$ and suppress $x'$ to get $(N',t')$ on
$X\setminus \{x\}$ where $t'$ is obtained from $t$ by keeping the label $t$ of all
remaining vertices. Then $(N',t')$ is still a strong quasi-discriminating elementary
network and by induction hypothesis, $G'=\mathscr{G}(N',t')$ is primitive and, by
construction, $G-x = G'$. Since $G'$ is primitive, for all $M\subsetneq X\setminus
\{x\}$ with $|M|>1$, there are vertices $a,b\in M$ such that $\{a,c\}\in E$ and
$\{b,c\}\notin E$ for some $c\in X\setminus M$. The latter clearly 
remains true for $M\cup \{x\}$ with
 $M\subsetneq X\setminus \{x\}$ in $G$ and $|M|>1$. Hence, any 
possible new
non-trivial module in $G$ must be of size two and contain $x$. Let
$M=\{x,y\}\subseteq X$ for some $y\in X$ distinct from $x$ and let $y'$ be the vertex
in $C$ adjacent to $y$. We show that $M$ cannot be a module in $G$. There are three
cases: (i) $y'\in P^{2}_{-}$, (ii) $y' = \eta_N$ and (iii) $y'\in P^{1}_{-}$. In the
following let $h\in X$ be the unique leaf adjacent to $\eta_N$.

In Case (i), we have $y'\in P^{2}_{-}$ and, in particular, $x',y'\neq \eta_N$ and
$x'\prec_N y'$. Therefore, $\lca_{N}(x,h) = x'$ and $\lca_{N}(y,h) = y'$. Since
$M=\{x,y\}$ is a module in $G$, $x$ and $y$ are either both adjacent to $h$ or not.
The latter two arguments imply $t(x') = t(y')$. Since $(N,t)$ is quasi-discriminating
there must be a vertex $z'\in P^2$ with corresponding leaf $z\in X$ such that
$x'\prec_N z' \prec_N y'$ and $t(z') \neq t(x') = t(y')$. Since $x'\prec_N z' \prec_N
y'$ we have $\lca_N(x,z)=z'\prec_N y'= \lca_N(y,z)$. Hence, $\{x,z\}\in E$ and
$\{y,z\}\notin E$ or \emph{vice versa}. Hence, $M$ cannot be a module in $G$. 

In Case (ii), we have $y' = \eta_N$ and $y=h$. Since $1\leq |P^{1}_{-}|$ and $(N,t)$
is quasi-discriminating there must be a vertex $z'\in P^{1}_{-}$ such that $t(z')\neq
t(\rho_N)$. Let $z\in X$ be the leaf adjacent to $z'$ in $N$. Since $y' = \eta_N$, we
have $\lca_N(y,z) = z'$ and since $x'$ and $z'$ are located on different sides of
$C$, it holds that $\lca_N(x,z) = \rho_N$. Since $t(z')\neq t(\rho_N)$, we have
$\{x,z\}\in E$ and $\{y,z\}\notin E$ or \emph{vice versa}. Hence, $M$ cannot be a
module in $G$. 

In Case (iii), $y'\in P^{1}_{-}$. Since $2\leq |P^2|$ and $(N,t)$ is
quasi-discriminating and since $x'$ is adjacent to $\eta_N$, there must be a vertex
$z'\in P^2\setminus \{\rho_N, \eta_N,x'\}$ such that $t(z')\neq t(\rho_N)$. Let $z\in
X$ be the leaf adjacent to $z'$ in $N$. Since $x'$ is adjacent to $\eta_N$, we have
$x'\prec_N z'$. This and $x',z'\in P^2$ implies that $\lca(x,z)=z'$. Since $y'$ and
$z'$ are located on different sides of $C$ and $y',z'\neq \eta_N$, it holds that
$\lca_N(y,z) = \rho_N$. Since $t(z')\neq t(\rho_N)$, we have $\{x,z\}\in E$ and
$\{y,z\}\notin E$ or \emph{vice versa}. Hence, $M$ cannot be a module in $G$. 

In summary, no subset $M\subseteq X$ with $1<|M|<|X|$ can be a module in $G$ and
thus, $G$ contains only trivial modules. Consequently, $G$ is primitive and, as
argued above, $G$ is level-1 explainable. 
\end{proof}

\begin{figure}[t]
		\begin{center}
			\includegraphics[scale = .7]{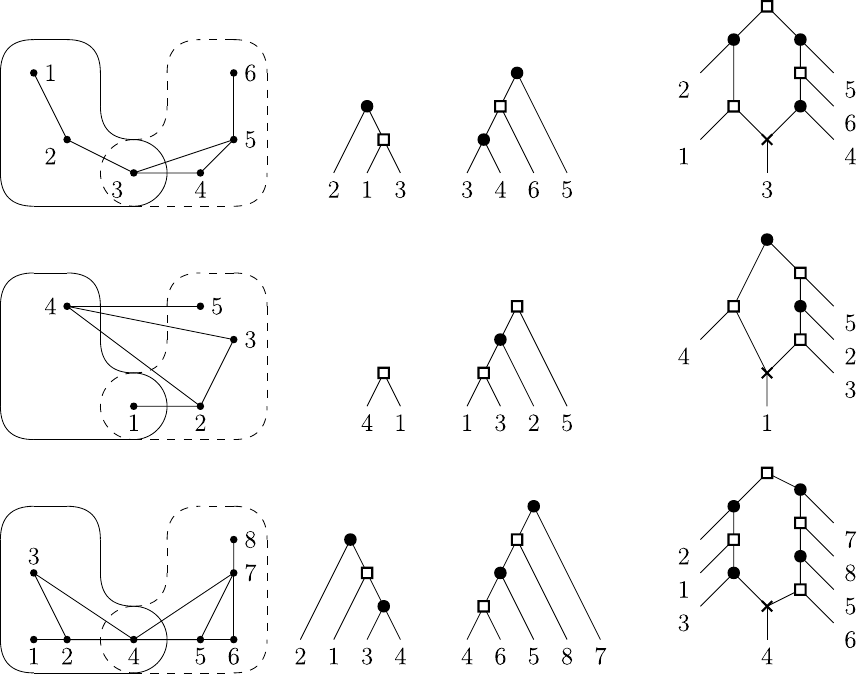}
		\end{center}
		\caption{Shown are three $(v,G',G'')$-polar-cats $G_1,G_2,G_3$ (from top to down)
		         together with the discriminating cotree $(T',t')$ of $G'$ and
		         $(T'',t'')$ of $G''$ and quasi-discriminating elementary networks
		         $(\widehat{N},\widehat{t})$ that explains the respective graph $G_i$.
		         The graph $G'$ and $G''$ are enclosed by dashed and solid-line curves,
		         respectively. The vertex $v$ is the unique vertex that is contained in
		         both $G'$ and $G''$. It is easy to see that the level-1 networks $(N,t)$
		         (right) are strong quasi-discriminating elementary networks that are
		         obtained from the respective cotrees $(T',t')$ and $(T'',t'')$ according
		         to Def.\ \ref{def:prop:PsC-l1N}.}
		\label{fig:polarCatN}
\end{figure}

Since polar-cats are primitive, we obtain
\begin{corollary} \label{cor:CographNotPC}
	All polar-cats are connected and not cographs.
	In particular, the smallest polar-cat is an
	induced $P_4$.
\end{corollary}

We proceed with showing that the converse of Lemma \ref{lem:primitive-1}
is satisfied as well.

\begin{lemma}\label{lem:primitive-2}
If $G$ is primitive and there is a level-1 network $(N,t)$ that explains $G$, then
$G$ is a polar-cat. In particular, if 	$(N,t)$ is a level-1 network
that explains a primitive graph $G$, then $(N,t)$ must be a strong
quasi-discriminating elementary network.
\end{lemma}
\begin{proof}
Let $G=(X,E)$ be primitive and assume that there is a level-1 network $(N,t)$ on $X$
that explains $G$. 
To show that $G$ is a polar-cat it suffices, by 
Thm.~\ref{thm:StrongElementary}, to show that $(N,t)$ is a strong  elementary 
quasi-discriminating network. 

Note, $(N,t)$ must contain cycles, since otherwise $G$ would be a cograph and thus, a
non-primitive graph. Let $C\subseteq N$ be such a cycle and let $\rho_C$ be its root and $\eta_C$ its unique hybrid-vertex. We show first that $\rho_C=\rho_N$ and
that the outdegree of $\rho_N$ must be 2. Assume, for contradiction, that $\rho_C\neq
\rho_N$ and therefore, $\rho_C\prec_N \rho_N$. Let $M\subseteq X$ be the set of all
leaves $x\in X$ with $x\prec_N \rho_C$. Using the definition of networks, it is
straightforward to verify that $X\setminus M\neq \emptyset$ and, therefore,
$|M|<|X|$. Note, $1<|M|$ since the cycle $C$ contains the vertex 
$\eta_C$ and at
least one vertex $v$ that is distinct from $\rho_C$ and $\eta_C$ and thus, there are
at least two leaves $x$ and $x'$ with $x\prec_N \eta_C \prec_N \rho_C$ and $x'\prec_N
v \prec_N \rho_C$. Let $y\in X\setminus M$. Hence, $y$ and $\rho_C$ are incomparable
and, in particular, $\lca_N(x,y) = \lca_N(\rho_C,y)$. Hence, for all $x,x'\in X$ we
have $\lca_N(x,y) = \lca_N(x',y)$ and therefore, $t(\lca_N(x,y)) = t(\lca_N(x',y))$.
Hence, $G$ contains a non-trivial module $M$ and is, thus, not primitive; a
contradiction. Therefore, $\rho_C=\rho_N$. 

Assume now, for contradiction, that the outdegree of $\rho_C$ is distinct from 2
and thus, at least 3. Let $u,v\in
C$ be the two vertices in $C$ that are adjacent to $\rho_C$ and $M\subseteq X$ be the
set of all leaves $x\in X$ with and $x\prec u$ or $x\prec v$. Since 
$\rho_C$ has outdegree at
least three, $|M|<|X|$. Moreover, at least one of $u$ and $v$ must be distinct from
$\eta_C$ which implies that $1<|M|$. By construction, for all $x\in M$ it holds that
$x\prec_N \rho_C$. Thus, for every $y\in X\setminus M$ and $x,x'\in M$ it holds that
$\lca_N(x,y) = \lca_N(x',y)$ and therefore, $t(\lca_N(x,y)) = t(\lca_N(x',y))$.
Hence, $M$ is a non-trivial module of $G$; a contradiction. Hence, the outdegree of
$\rho_C$ must be 2. 

In summary, for every cycle $C\subseteq N$ it must hold that $\rho_C=\rho_N$ and the
outdegree of $\rho_N$ is 2. Therefore, $N$ cannot contain a further cycle different
from $C$. By definition of networks, in particular by Condition (N3), every vertex
$u\neq \rho_C$ in $C$ has at least one child $u'$ that is not located in $C$. If
every such vertex $u$ has precisely one such child $u'\in X$ then $C$ is an
elementary network. Assume, for contradiction, that the latter condition is violated.
Hence, there is a vertex $u\neq \rho_C$ in $C$ that has either at least two children
that are not located in $C$ or precisely one child $u'\notin V(C)$ that is not a
leaf. Suppose first that $u$ has at least two children that are not located in $C$.
Let $W$ be the set of vertices $v$ in $C$ with $v\prec_N u$ and $M$ be the collection
of leaves $x\in X$ with $x\prec_N u$ and $x\not\prec_N w$ for any $w\in W$. Since $u$
has at least two children that are not located in $C$, we have $1<|M|$ and since
$u\neq \rho_C$, it holds that $|M|<|X|$. Since there are no other cycles in $N$ than
$C$, it is easy to verify that for all $z\in X\setminus M$ and
$x,y\in M$  it holds
that $u\preceq_N\lca_N(x,z) = \lca_N(y,z)$. Again, $M$ is a non-trivial module of
$G$; a contradiction. Assume now that $u$ has precisely one child $u'\notin V(C)$
that is not a leaf and let $M$ be the collection of leaves $x\in X$ with $x\prec_N
u'$. By similar arguments as in the previous case one shows that $M$ is a non-trivial
module; a contradiction. In summary, every vertex $u\neq \rho_C$ in $C$ has at least
one child $u'$ that is not located in $C$ and this child must be a leaf. Thus, $N$ is
an elementary network. Note that the latter immediately implies that 
$(N,t)$ must be  quasi-discriminating, since otherwise, we could contract
an edge  $(u,v)\in E^0$ with $t(u)=t(v)$ and obtain the network  $(N_e,t_e)$
that still explains $G$ and for which the vertex $v_e$ that corresponds to the edge $e$ in $N_e$ would
have two leaves as children; a case that cannot occur as argued above.
Moreover, $N$ must be strong, since otherwise Lemma \ref{lem:WeakElementary} implies
that $G = \mathscr{G}(N,t)$ is a cograph and thus, not primitive; a contradiction. In
summary, we have shown that $(N,t)$ is a quasi-discriminating strong elementary
network, which completes the proof. 
\end{proof}

\begin{theorem}\label{thm:CharPolCat}
	A graph $G$ is a polar-cat if and only if $G$ is a primitive graph that can be
	explained by a level-1 network.
\end{theorem}
\begin{proof}
	The \emph{only-if}-direction is a direct consequence of Lemma \ref{lem:primitive-1},
		while the  \emph{if}-direction follows from \ref{lem:primitive-2}.	
\end{proof}

Moreover, Thm.\ \ref{thm:CharPolCat} together with Thm.\ \ref{thm:StrongElementary}
implies
\begin{theorem}\label{thm:CharPCPC}
For a primitive graph $G$ the following statements are equivalent:\smallskip
\begin{enumerate}[nolistsep,noitemsep]
\item	$G \simeq \mathscr{G}(N,t)$ for some labeled level-1 network $(N,t)$.
\item $G \simeq \mathscr{G}(N,t)$ for some strong elementary quasi-discriminating
      network $(N,t)$.
\item $G$ is a polar-cat.
\end{enumerate}
\end{theorem}
\begin{proof}
	Let $G$ be a primitive graph. If Condition (1) is satisfied, then Thm.\ 
	\ref{thm:CharPolCat} implies that $G$ is a polar-cat. Thus we can apply Thm.\
	\ref{thm:StrongElementary} to conclude that 
	there is a strong elementary quasi-discriminating
	network $(N,t)$ with $G \simeq \mathscr{G}(N,t)$. Hence (1) implies (2). 
	If Condition (2) is satisfied, then  Thm.\ \ref{thm:StrongElementary} implies  
	$G$ is a polar-cat. Hence (2) implies (3). 
	Finally assume that $G$ is a polar-cat. By Thm.\ \ref{thm:CharPolCat}, 
	$G \simeq \mathscr{G}(N,t)$ for some labeled level-1 network $(N,t)$.
	Hence, (3) implies (1).
	In summary, the conditions 	(1), (2) and (3) are equivalent. 
\end{proof}

In the following, we show under which conditions a polar-cat can be explained by a
unique labeled level-1 network. To this end, we need
\begin{definition}[well-proportioned]\label{def:well-propo}
	Let $G$ be a $(v,G_1,G_2)$-pseudo-cograph. Then, 
	$(v,G_1,G_2)$  is \emph{well-proportioned} if $|V(G_1)| \geq 3$ and $|V(G_2)| 
	\geq 3$, or $|V(G_i)| = 2$ and $|V(G_j)|\geq 5$ with $i,j\in\{1,2\}$
	distinct.
\end{definition}

Examples of well-proportioned pseudo-cographs are provided in Fig.\
\ref{fig:polarCatN} (upper and lower row). Note that, by definition,
well-proportioned pseudo-cographs $G$ must satisfy $|V(G)|\geq 5$. Consequently,
$G\simeq P_4$ is not well-proportioned. The smallest well-proportioned
$(v,G_1,G_2)$-pseudo-cograph is the path $P_5 = 1-2-3-4-5$ where $G_1 =
G[1,2,3]$, $G_2 = G[3,4,5]$ and $v=3$. It is easy to verify that $G$ is, in
particular, a $(v,G_1,G_2)$-polar-cat.	

As we shall see, for a given $(v,G_1,G_2)$-polar-cat, the choice of $G_1$ and
$G_2$ cannot be changed as long as $v$ is fixed. However, a pseudo-cograph might
be $(v,G_1,G_2)$-polar-cat and $(w,G'_1,G'_2)$-polar-cat with $w\neq v$ at the
same time. By way of example consider the polar-cat $G$ as shown in Fig.\
\ref{fig:nonUniqueN} (lower row). Based on the different networks that explain
$G$ it is easy to verify that $G$ is a $(1,G[\{1,3\}],
G[\{1,2,4,5\}])$-polar-cat as well as a $(2,G[\{1,2\}],
G[\{2,3,4,5\}])$-polar-cat. In both cases, however, $(v,G_1,G_2)$ is not
well-proportioned. As we show in the following, for a well-proportioned
polar-cat $G$ there is no ambiguity for the choice of $v$ and thus, of $G_1$ and
$G_2$ such that $G$ is a $(v,G_1,G_2)$-polar-cat. This, in particular, can then
be used to show that the networks that explain a well-proportioned polar-cat are
uniquely determined.

\begin{lemma}\label{lem:special1}
If $G$ is a $(v,G_1,G_2)$-polar-cat, then there are no subgraphs $G_1', G_2'
\subseteq G$ that are distinct from $G_1$ and $G_2$ such that $G$ is a
$(v,G_1',G_2')$-polar-cat. Moreover, $v$ is the only vertex such that $G$ is a 
$(v,G_1,G_2)$-polar-cat if and only if $(v,G_1,G_2)$ is well-proportioned.
\end{lemma}
\begin{proof}
Let $G$ be a $(v,G_1,G_2)$-polar-cat. By Lemma~\ref{lem:PolarCat-complement-CC}, the
disconnected graph in $\{G-v,\overline{G-v}\}$ contains exactly two connected
component. Lemma \ref{lem:star-center-gamma} implies that there are no subgraphs $G_1',
G_2' \subseteq G$ that are distinct from $G_1$ and $G_2$ and such that $G$ is a
$(v,G_1',G_2')$-polar-cat.

We continue with showing that the second statement is satisfied. Note, since $G$ is a
$(v,G_1,G_2)$-polar-cat, we have $|V(G_1)|, |V(G_2)|\geq 2$. For the
\emph{only-if}-direction assume, by contraposition, that $(v,G_1,G_2)$ is not
well-proportioned. Hence, we can w.l.o.g.\ assume that $|V(G_1)|= 2$ and
$|V(G_2)|\leq 4$. By Lemma \ref{lem:primitive-1}, $G$ is primitive and thus,
$|V(G)|\geq 4$. This together with $|V(G_1)| = 2$ and $|V(G_2)|\leq 4$ implies that
$|V(G_2)|\in \{3,4\}$. In view of Lemma~\ref{lem:PolarCat-complement-CC}, we can
w.l.o.g.\ assume that $G-v$ is disconnected. Moreover, by
Lemma~\ref{lem:PolarCat-complement-CC}, $G-v$ has exactly two connected components
$H_1=G_1-v$ and $H_2 = G_2-v$. By Prop.~\ref{prop:polcat-pc}(a), the
vertices of $G$ can be ordered and, based on the particular ordering and the number
of vertices of $G$, the graph $G$ is either isomorphic to an induced $P_4$ of the form
$y_1-v-z_1-z_2$ (in case $|V(G_2)|=3$), or to a graph that contains an induced $P_4$
of the form $y_1-v-z_1-z_2$ and a vertex $z_3$ that is adjacent to $v$ and $z_1$ (in
case $|V(G_2)|=4$). It is easy to verify that, in case $|V(G_2)|=3$, $G$ is a
$(v,G[\{y_1,v\}],G[\{v,z_1,z_2\}])$ and a $(z_1,G[\{z_1,z_2\}],G[\{z_1,v,y_1\}])$
polar-cat and, in case $|V(G_2)|=4$, $G$ is a $(v,G[\{y_1,v\}],G[\{v,z_1,z_2,z_3\}])$-
and a $(z_1,G[\{z_1,z_2],G[\{z_1,z_3,v,y_1\}])$-polar-cat. Hence, $v$ is not the only
vertex such that $G$ is a $(v,G_1,G_2)$-polar-cat.

For the \emph{if}-direction assume that $(v,G_1,G_2)$ is well-proportioned and
thus, $|V(G_1)|, |V(G_2)|\geq 3$ or $|V(G_i)| = 2$ and $|V(G_j)|\geq 5$,
$\{i,j\}=\{1,2\}$. Again, we assume w.l.o.g.\ that $G-v$ is disconnected, that
$H_1=G_1-v$ and $H_2=G_2-v$ are the two connected components of $G-v$ and that
the vertices are ordered according to Prop.~\ref{prop:polcat-pc}(a). Assume first
that $|V(G_1)|, |V(G_2)|\geq 3$. Based on the vertex ordering, there 
exists vertices $y_1, y_2 \in V(H_1)$
such that $\{y_1,v\}$, $\{y_1,y_2\}$ are edges of $G$, while $\{y_2,v\}$ is not.
Similarly, there exists vertices $z_1,z_2 \in V(H_2)$ such that $\{z_1,v\}$,
$\{z_1,z_2\}$ are edges of $G$, while $\{z_2,v\}$ is not an edge of $G$. Moreover,
since $G-v$ is disconnected, there are no edges of the form $\{y_i,z_j\}$, $i,j\in
\{1,2\}$. Hence, $y_2-y_1-v-z_1-z_2$ is an induced $P_5$ of $G$. By
Lemma~\ref{lem:v-in-allP4}, the vertex $v$ is uniquely determined.

Assume now that $|V(G_i)| = 2$ and $|V(G_j)|\geq 5$ and let w.l.o.g. $i=1$ and $j=2$.
Let $x$ be the unique vertex of $H_1=G_1-v$. By assumption, the vertices of
$G_2$ are indexed as $y_1, y_2,
y_3, y_4, \dots,y_k=v$, $k \geq 5$ such that $\{y_i,y_j\}$
is an edge of $G$ if and only if $i$ is odd, for $1 \leq i<j \leq k$. It is now easy
to verify that $G$ contains two induced $P_4$s: 
$x-v-y_1-y_2$ and
$x-v-y_3-y_4$. Both induced $P_4$s have only vertices $x$ and $v$ in 
common.
Hence, by Lemma~\ref{lem:v-in-allP4}, there is only one further choice for $G$ being
a $(v',G_1',G_2')$-polar-cat, namely $v'=x$. We continue with showing that $G$ is not
an $(x,G'_1,G'_2)$-polar-cat for any two subgraphs $G'_1$ and $G'_2$ of $G$. To this
end, consider the graph $G-x$. Note that $G-x=G_2$. In particular, $G-x$ is
connected, and $\overline{G-x}$ has exactly two connected components,
$(\{y_1\},\emptyset)$ and $G_2-y_1$. By Prop.~\ref{prop:charPseudo}, it follows
that the only choices for $G_1'$ and $G_2'$ are $G_1'=G[\{x,y_1\}]$ and $G_2'=G-y_1$.
However, $G_2'$ is not a cograph, since it contains the induced $P_4$ $x-v-y_3-y_4$.
Hence, $G$ is not an $(x,G_1',G_2')$-pseudo-cograph. In particular, $G$ is not a
$(x,G_1',G_2')$-polar-cat and $v$ is, therefore, uniquely determined.
\end{proof}

We are now in the position to show that  well-proportioned polar-cats
are characterized by the existence of a unique network that explains them.

\begin{proposition}\label{prop:uniqueNt}
Let $G$ be a $(v,G_1,G_2)$-polar-cat. Then, there is a unique network $N$ with unique
labeling $t$ (up to the label of the hybrid-vertices) that explains $G$ if and
only if $(v,G_1,G_2)$ is well-proportioned. In this case, $(N,t)$ is
in particular a strong quasi-discriminating elementary network.
\end{proposition}
\begin{proof}
Let $G$ be a $(v,G_1,G_2)$-polar-cat. For the \emph{only-if}-direction assume,
by contraposition, that $(v,G_1,G_2)$ is not well-proportioned. Since $G$ is a
polar-cat, Cor.\ \ref{cor:CographNotPC} implies that $G$ is not a cograph and
thus, it must contain an induced $P_4$. Consequently, $|V(G)|\geq 4$. We can now
apply the same arguments as used to verify the \emph{only-if}-direction of Lemma
\ref{lem:special1} and conclude that $G$ is a $(v,G_1,G_2)$- and a
$(w,G'_1,G'_2)$-polar-cat with $v,w\in V(G)$ being distinct. By definition, $G$
is a pseudo-cograph and thus, by Prop.\ \ref{prop:PsC-l1N}, we can construct a
network $(N(v,G_1,G_2),t(v,G_1,G_2))$ and a network
$(N(w,G'_1,G'_2),t(w,G'_1,G'_2))$ according to Def.\ \ref{def:prop:PsC-l1N} that
both explain $G$. Since $v$, resp., $w$ is, by construction, the unique child of
the unique hybrid in $N(v,G_1,G_2)$, resp., $N(w,G'_1,G'_2)$ it follows that
$N(v,G_1,G_2)\neq N(w,G'_1,G'_2)$. Thus, there are two labeled network that
explain $G$. 

For the \emph{if}-direction assume that $G$ is well-proportioned. Let $(N,t)$ be
an arbitrary level-1 network that explains $G$. By Lemma
\ref{lem:primitive-1}, $G$ is primitive and Lemma \ref{lem:primitive-2} implies
that $(N,t)$ must be a strong elementary quasi-discriminating network. We
proceed with showing that $(N,t)$ is uniquely determined by $G$.

Let $C$ be the unique cycle of $N$ with sides $P^1$ and $P^2$. Put $P^{1}_{-}
\coloneqq P^1\setminus \{\rho_N\}$ and $P^2_{-} \coloneqq P^2\setminus \{\rho_N\}$.
Note, since $C$ is strong, it holds that neither $P^1_{-}$ nor $P^2_{-}$ is empty. We
denote by $W_1$ and $W_2$ the set of leaves of $N$ whose parent is a vertex of
$P^1_{-}$ and $P^2_{-}$, respectively. Moreover, $h$ denotes the unique child of
$\eta_C$. Note that by definition, $W_1 \cup W_2=V(G)$ and $W_1 \cap W_2=\{h\}$ hold.
The same arguments as in the proof for the \emph{if}-direction of Thm.\
\ref{thm:StrongElementary} show that $G$ is a $(h,G[W_1],G[W_2])$-polar-cat. Since
$G$ is a well-proportioned $(v,G_1,G_2)$-polar-cat, we can apply
Lemma~\ref{lem:special1} to conclude that $v=h$, and that
$\{G[W_1],G[W_2]\}=\{G_1,G_2\}$. Note, since $(N,t)$ was chosen arbitrarily and since
$v$ as well as $G_1$ and $G_2$ are uniquely determined by $G = \mathscr{G}(N,t)$, it
follows that for any network $(N,t)$ that explains $G$ it must hold that all vertices
of $G_1$, resp., $G_2$ must be adjacent to vertices in $P^{1}_{-}$, resp.,
$P^{2}_{-}$ or \emph{vice versa}.

W.l.o.g.\ assume that $G[W_1]=G_1$ and $G[W_2]=G_2$. Let $y_1, \ldots, y_k=v$ be the
ordering of the elements of $W_1$ as postulated by Prop.~\ref{prop:polcat-pc}. Since
$N$ is elementary and by the latter arguments, we have in particular that
$|V(P^1_-)|=k$. Consider now the induced directed path $P^1_-=u_1 - \ldots -u_k$ in
$N$. By definition, $u_k =\eta_C$. We next show that $y_i$ must be a child of $u_i$
in $N$ for all $i\in \{1,\dots,k\}$. To this end, assume, for contradiction, that
there is some vertex $u_i$ with child $y_j$ in $N$, with $1 \leq i \neq j \leq k$.
Without loss of generality, we may assume that $i$ is the smallest element of $\{1,
\ldots k\}$ with that property which, in particular, implies that $j>i$, and $y_i$ is
the child of some vertex $u_l$ satisfying $l>i$. Note that $l=k$ does not hold, as we
have already established that $v=y_k$ is the unique child of $\eta_C=u_k$. Note that
$u_l\prec_N u_i$. Since $(N,t)$ explains $G$ and $\lca(y_j,v)=u_i=\lca(y_j,y_i)$, it
follows that $\{y_i,y_j\}$ is an edge of $G$ if and only if $\{y_j,v\}$ is an edge of
$G$. By choice of the labels, we also have that $\{y_i,y_j\}$ is an edge of $G$ if
and only if $\{y_i, v\}$ is an edge of $G$ (depending on whether $j$ is odd or even).
In particular, $t(u_i)=t(u_l)$ must hold. Since $N$ is quasi-discriminating, is
follows that $l>i+1$. Now, let $y_m$ be the child of $u_{i+1}$, $m=i+1$ is possible.
By choice of $i$, $i<m$ must hold. Using the same arguments as above with $y_m$
playing the role of $x_j$, it follows that $t(u_{i+1})=t(u_l)$. However, this
equality, together with $t(u_i)=t(u_l)$, implies that $t(u_i)=t(u_{i+1})$, a
contradiction since $(N,t)$ is discriminating.

Hence, the ordering of leaves along the path $P^1_-$ in $N$ is uniquely determined by
$G$. By the same arguments, the latter holds also for the ordering of the leaves
along the path $P^2_-$. Therefore, $N$ is uniquely determined by $G$. Since the label of
a tree-vertex of $N$ is uniquely determined by $N$ and $G$, it follows that, up to
the choice of $t(\eta_C)$, $(N,t)$ is the unique labeled level-1 network explaining
$G$.
\end{proof}

\section{General Graphs, the Class $\protect\PrimeCat$ and Level-1 Networks}
\label{sec:general}

In this section, we characterize the class $\PrimeCat$ of graphs that can
be explained by level-1 networks. 
For later reference, we show first that the property of a graph being explainable by a
level-1 network is hereditary.

\begin{lemma}\label{lem:level-1-inher}
	A graph $G$ can be explained by a labeled level-1 network if and only if every
	induced subgraph of $G$ can be explained by a level-1 network.	
\end{lemma}
\begin{proof}
	The \emph{if}-direction immediately follows from the fact that $G$ is an induced
	subgraph of $G$. For the \emph{only-if}-direction, assume that $G=(X,E)$ can be
	explained by a labeled level-1 network $(N,t)$ on $X$. If $|X|\leq 3$, 
	then $G$ and each of its
	induced subgraphs are cographs and the statement follows from Thm.\
	\ref{thm:WeakIffCograph}. Hence, let $|X|\geq 4$. Since any induced subgraph of a
	graph can be obtained by removing vertices one by one, it is sufficient to show
	that the removal of a single vertex from $G$ yields a graph that can be explained
	by a level-1 network. Let $x\in X$ and put $G'\coloneqq G-x$. Let $N'$ be the
	network that is obtained from $N-x$ after repeating the following four steps until
	no vertices and edges that satisfy (1), (2), (3) and (4) exist: (1) suppression
	of indegree 1 and outdegree 1 vertices; (2) removal of indegree 2 and outdegree 0
	vertices and its two incident edges; (3) removal of indegree 0 and outdegree 1
	vertices and its incident edge; and (4) removal of all but one of possible
	resulting multi-edges. Since $V(N')\subseteq V(N)$, we can put $t'(v)=t(v)$ for all
	$v\in V(N')$. 
  
 We show now that $(N', t')$ is a level-1 network that explains $G'$. We start
	with showing that $N'$ is a level-1 network. Since $|X|>1$, we need to verify
	(N1), (N2) and (N3).  Note first that Step (1) and (4) preserve connectedness
	of $N'$. Moreover, if we remove in Step (2) indegree 2 and outdegree 0 vertices
	and its two
 incident edges, then $x$ must have been incident to a hybrid-vertex $\eta_C$ in $N$.
 Hence, if Step (2) was applied, then $x$ was the only vertex that is incident to
 $\eta_C$ in $N$ and one easily verifies that Step (2) preserves connectedness of
 $N'$. Furthermore, assume we have applied Step (3). For the first application of
 Step (3), there can only be one vertex in $N-x$ that has indegree 0 and outdegree 1,
 namely the root $\rho_N$ of $N$. In this case, the outdegree of $\rho_N$ must be 2
 in $N$. If $x$ is adjacent to $\rho_N$, then $\rho_N$ cannot be the root of a cycle
 in $N$ and it follows that the unique child $u\neq x$ of $\rho_N$ must have
 outdegree at least 2. After applying Step (3) on $\rho_N$, vertex $u$ becomes the
 root of $N'$. If $x$ is not adjacent to $\rho_N$, then $x$ is adjacent to a vertex
 $v$ that is located on a weak cycles $C$ in $N$ and $C$ consists of edges
 $(\rho_N,v)$, $(v,\eta_C)$ and $(\rho_N,\eta_C)$ (otherwise, if $C$ is not weak,
 $C'$ obtained after possible suppression of $v$ would remain a cycle $C'$ in $N'$
 since no multiple-edges occur and thus, Step (3) would not have been applied on
 $\rho_N$). After removal of $x$, Step (1) is applied on $v$ and we obtain two
 multi-edges $(\rho_N,\eta_C)$ of which one is removed by Step (2) and now, Step (3)
 is applied on $\rho_N$ and $\eta_C$ becomes the new root of the resulting network.
 It could be that we apply the last Step (3) again on $\eta_C$, in case it has now
 indegree 0 and outdegree 1. However, in each of these cases it is ensured that there
 remains a vertex $u$ from which each leaf$x'\in X\setminus\{x\}$
 can be reached by a directed path in $N'$ and, thus, $N'$ remains connected. 
 It is now easy to verify that $N'$
 remains a DAG that satisfies (N1) and (N2). To see that (N3) is satisfied, let $v\in
 V^0(N')$. This vertex cannot have indegree 1 and outdegree 1 since, otherwise, it
 would have been suppressed. Hence, if $v$ has indegree 1 it must have outdegree 0 or 
 at least  2, i.e., (N3.a) holds for all indegree 1 vertices. Moreover, $v$ can also not 
 have indegree 2 and outdegree 0 due to Step (2). Hence, if $v$ has indegree 2 it must have 
 outdegree least 1, i.e., (N3.b) holds for all indegree 2 vertices. Moreover, 
 there cannot be vertices in $N'$ with indegree at least 3, 
 since we only suppressed
 vertices and removed edges (and possibly vertices) to obtain $N'$ and thus, never
 increased the indegree of any vertex. Hence, $N'$ satisfies (N1), (N2) and (N3) and
 is, therefore, a level-1 network.
  
  It remains to show that $(N',t')$ explains $G'$. Let $y,z\in X\setminus \{x\}$.
  Consider $u\coloneqq \lca_N(y,z)$. Since $u$ is a $\preceq_N$-minimal ancestor of
  both $y$ and $z$ it follows that there are two children $u_y$ and $u_z$ of $u$ such
  that $y\preceq_N u_y$, $z\preceq_N u_z$ and $z\not\preceq_N u_y$, $y\not\preceq_N
  u_z$. This implies that $u$ must have outdegree at least 2 in $N$. Since $u
  =\lca_N(y,z)$, any two directed paths $P_y$ and $P_z$ from $u$ to $y$ and $u$ to
  $z$ in $N$, respectively, can only intersect in vertex $u$. As argued above, there
  are still directed paths $P'_y$ and $P'_z$ from $u$ to $y$ and $u$ to $z$ in $N'$,
  respectively. Since Step (1)-(4) can never identify vertices of disjoint paths in
  $N$ it follows that all such directed paths $P'_y$ and $P'_z$ can only intersect in
  $u$. This implies that $u$ must have outdegree at least 2 in $N'$ and thus, will
  never be suppressed and thus $u\in V(N')$. In particular, the latter arguments
  imply that $u= \lca_{N'}(y,z)$ for all $y,z\in X\setminus \{x\}$. By the choice of
  $t'$, we have $t'(u)=t(u) = t(\lca_N(x,y))$ and thus, $(N',t')$ explains $G'$.
\end{proof}

Now, let $G=(X,E)$ be a graph. Recall, the set of strong modules $\MDstrong(G)$ forms a
hierarchy and gives rise to a unique tree representation with leaf set $X$, the MDT
$(\MDT_G,\tau_G)$ of $G$. We aim at extending the MDT of $G$ to a labeled level-1
network that explains $G$. Uniqueness and the hierarchical structure of
$\MDstrong(G)$ implies that there is a unique partition $\Mmax(G) = \{M_1,\dots,
M_k\}$ of $X$ into maximal (w.r.t.\ inclusion) strong modules $M_j\ne X$ of $G$
\cite{ER1:90,ER2:90}. Since $X\notin \Mmax(G)$ the set $\Mmax(G)$ consists of $k\ge
2$ strong modules, whenever $|X|>1$.

Let $M, M'\in \MD(G)$ be disjoint modules. We say that $M$ and $M'$ are adjacent (in
$G$) if each vertex of $M$ is adjacent to all vertices of $M'$; the sets are
non-adjacent if none of the vertices of $M$ is adjacent to a vertex of $M'$. By
definition of modules, every two disjoint modules $M, M'\in \MD(G)$ are either
adjacent or non-adjacent \cite[Lemma 4.11]{ER1:90}. One can therefore define the
\emph{quotient graph} $G/\Mmax(G)$ based on $G$ and the inclusion-maximal subsets in
$\MDstrong(G)\setminus \{X\}$. The quotient graph $G/\Mmax(G)$ has $\Mmax(G)$ as its
vertex set and $\{M_i,M_j\}\in E(G/\Mmax(G))$ if and only if $M_i$ and $M_j$ are
adjacent in $G$.
\begin{observation}[\cite{HP:10}]\label{obs:quotient} 
The quotient graph $G/\Mmax(G)$ with $\Mmax(G) = \{M_1 , \dots , M_k\}$ is isomorphic
to any subgraph induced by a set $W\subseteq V$ such that $|M_i \cap W | = 1$ for all
$i \in \{1, \dots,k\}$.
\end{observation}

By \cite[Lemma 3.4]{HFWS:20}, every prime module $M$ is strong and thus, associated with
a unique vertex $v$ that satisfies $L(\MDT_G(v))=M$ in the MDT of $G$. Recall that q such 
vertices are called ``prime
vertices'' of $(\MDT_G, \tau_G)$. If $M$ is a prime module, the graph $G[M]$ is not
necessarily primitive, however, its quotient $G[M]/\Mmax(G[M])$ is always primitive. In order to infer $G$ from
$(\MDT_G, \tau_G)$ we need to determine the information as whether $x$ and $y$ are
adjacent or not for all pairs of distinct leaves $x,y\in X$. In the case of prime
modules, however, we must therefore drag the entire information of the quotient graphs.
An alternative idea is to replace prime vertices in $\MDT_G$ by suitable graphs and
to extend the labeling function $\tau_G$ that assigns the ``missing information'' to
the inner vertex of the new graph. This idea has been fruitful for median graphs
\cite{BHS:21}. We will apply this idea now on the MDT to obtain level-1 networks and 
replace prime modules in the MDT by a suitable choice of labeled
cycles to obtain labeled level-1 networks. This is achieved by the following
\begin{definition}[prime-vertex replacement (pvr) graphs]
  \label{def:pvr}
  Let $G\in \PrimeCat$ and $\mathcal{P}$ be the set of all prime vertices in
	$(\MDT_G,\tau_G)$. The \emph{prime-vertex replacement} (\emph{pvr}) graph $(N^*,	
	t^*)$ of $(\MDT_G,\tau_G)$ is obtained by the following procedure:
\begin{enumerate}[noitemsep] 
\item For all $v\in \mathcal{P}$, let $(N_v,t_v)$ be a
			 strong quasi-discriminating elementary network with root $v$
			 that explains $G[M]/\Mmax(G[M])$ with $M = L(\MDT_G(v))$.			
  \label{step:Gv} \smallskip
\item For all $v\in \mathcal{P}$, remove all edges $(v,u)$ with
  $u\in \child_{\MDT_G}(v)$ from $\MDT_G$ to obtain the forest
  $(T',\tau_G)$ and \label{step:T'}  
	add $N_v$ to $T'$ by identifying the root
  of $N_v$ with $v$ in $T'$ and each leaf $M'$ of $N_v$ with the
  corresponding child $u\in \child_{\MDT_G}(v)$ for which $M' = L(\MDT_G(u))$.  \smallskip

\noindent
 \emph{This results in the pvr graph $N^*$.}\smallskip
\item \label{step:color} 
 Define the labeling $t^*\colon V(N^*)\to \{0,1,\odot\}$ by putting, for
  all $w\in V(N^*)$,
  \begin{equation*}
    t^*(w) =
    \begin{cases} 
      \tau_G(v) &\mbox{if } v\in V(\MDT_G)\setminus \mathcal P \\
      t_v(w) &\mbox{if } w \in V(N_v)\setminus X \text{ for some } v\in \mathcal P
    \end{cases}
  \end{equation*}
\end{enumerate}
\end{definition}

The construction of pvr graphs has first been studied in the context of median graphs
by Bruckmann et al.\ \cite{BHS:21}. If there is a unique vertex $\med(x,y,z)$
for vertices $x,y,z$ that belongs to shortest paths between each pair of them, 
it is called \emph{median of} $x,y,z$. Median graphs are connected graphs in which
all three vertices have a unique median.
Our definition of
pvr graphs is just a special case of \cite[Def.\ 3.7]{BHS:21}), see also Remark 4.2
in \cite{BHS:21}. To be more precise, pvr graphs have been defined in \cite{BHS:21}
for modular decomposition trees of symmetric maps $\delta\colon X\times X \to
\Upsilon$ with $\Upsilon$ being a set of arbitrary colors. In our context, these maps
$\delta$ can be simplified to $\delta_G\colon X\times X \to \{0,1\}$ such that
$\delta_G(x,y) = 1$ if and only if $\{x,y\}\in E$. In other words, $\delta_G$
characterizes the adjacencies in $G=(X,E)$, i.e., every graph $G$ is uniquely determined by
$\delta_G$. Thus, we can focus on graphs $G$ rather than on such maps $\delta$. It
has been shown that every graph $G$ can be explained by a ``general'' pvr graph that
is obtained by replacing prime modules $M$ in the modular decomposition tree by
labeled graphs $(N_v,t_v)$ for which $\med(\rho_{N_v}, x,y)$ is uniquely determined
and that explain $G[M]/\Mmax(G[M])$. In our specialized definition, we use strong
quasi-discriminating elementary network $(N_v,t_v)$ instead. Observe first that, for
all prime modules $M$ of $G$, the quotient graph $G[M]/\Mmax(G[M])$ is supposed to be
a polar-cat since $G \in \PrimeCat$. By Thm.\ \ref{thm:CharPCPC}, $G[M]/\Mmax(G[M])$
is explained by a strong quasi-discriminating elementary network $(N_v,t_v)$ with
leaf set $\Mmax(G[M])$ and where $M=L(\MDT_G(v))$. The results established in
\cite{BHS:21} rely on the fact that there is a unique vertex $z$ in $N_v$ that
satisfies $M_i,M_j\prec_{N_v}z\preceq_{N_v} \rho_{N_v}$ and whose label $t_v(z)$
determines as whether $M_i$ and $M_j$ are adjacent or not  in $G[M]/\Mmax(G[M])$, for
all $M_i, M_j\in \Mmax(G[M]$. This implies that, in the pvr graph $(N^*,t^*)$, vertex
$z \in V^0(N_v) \subseteq V^0(N^*)$ satisfies $x,y\prec_{N^*}z\preceq_{N^*}
\rho_{N_v} \preceq_{N^*} \rho_{N^*}$ and the label $t_v(z)$ determines if $x$ and $y$
are adjacent or not in $G$,  for all $x,y\in X$. By construction, $t^*$ retains all
labels of non-prime vertices in the MDT and uses labels $t_v$ for the vertices
contained in $N_v - \Mmax(G[M]) \subseteq N^*$. Instead of employing a vertex $z = \med(\rho_{N_v}, x,y)$
(which may not exist in $N_v$ e.g.\ if $N_v$ contains an odd-length cycle), we use $z
= \lca_{N_v}(M_i,M_j)$ which is uniquely determined for all $M_i,M_j \in L(N_v) = \Mmax(G[M])$
and thus, $z= \lca_{N^*}(x,y)$ for all $x\in M_i$, $y\in M_j$ and all $M_i,M_j \in \Mmax(G[M])$. This
makes it possible to reuse the results established in \cite{BHS:21}  which eventually shows that
$(N^*,t^*)$ explains $G$. Note that all
prime vertices are ``replaced'' in $\MDT_G$ by strong elementary networks, which
implies that all cycles in $N^*$ are strong. Thus, $N^*$ is strong. Moreover,
different prime vertices are replaced by different elementary networks and thus,
cycles in $N^*$ are pairwise vertex disjoint. In summary, $(N^*, t^*)$ is a strong
level-1 network. Even more, adjusting \cite[Thm.\ 3.11]{BHS:21} to our special case
we obtain
\begin{proposition}\label{prop:well-defPVR}
	If $G\in \PrimeCat$, then a pvr graph $(N^*, t^*)$ of $(\MDT_G,\tau_G)$ is
	well-defined and, in particular, a strong level-1 network that explains $G$, i.e.,
	$\mathscr{G}(N^*, t^*) = G$. 
\end{proposition}

Note that a pvr graph $(N^*, t^*)$ of $G$ is not necessarily quasi-discriminating.
Nevertheless, by Lemma \ref{lem:discriminatingN}, it can easily be transformed into a
quasi-discriminating network that still explains $G$, see Fig.\ \ref{fig:pvr} for an
illustrative example. Note, there can be different pvr graphs that explain the same
graph $G$ since the choice of the elementary networks $(N_v,t_v)$ that are used to
replace prime vertices $v$ in the MDT $\MDT(G)$ of $G$ are not unique in general, see
Fig.\ \ref{fig:nonUniqueN}. 

\begin{figure}[t]
		\begin{center}
			\includegraphics[width = .9\textwidth]{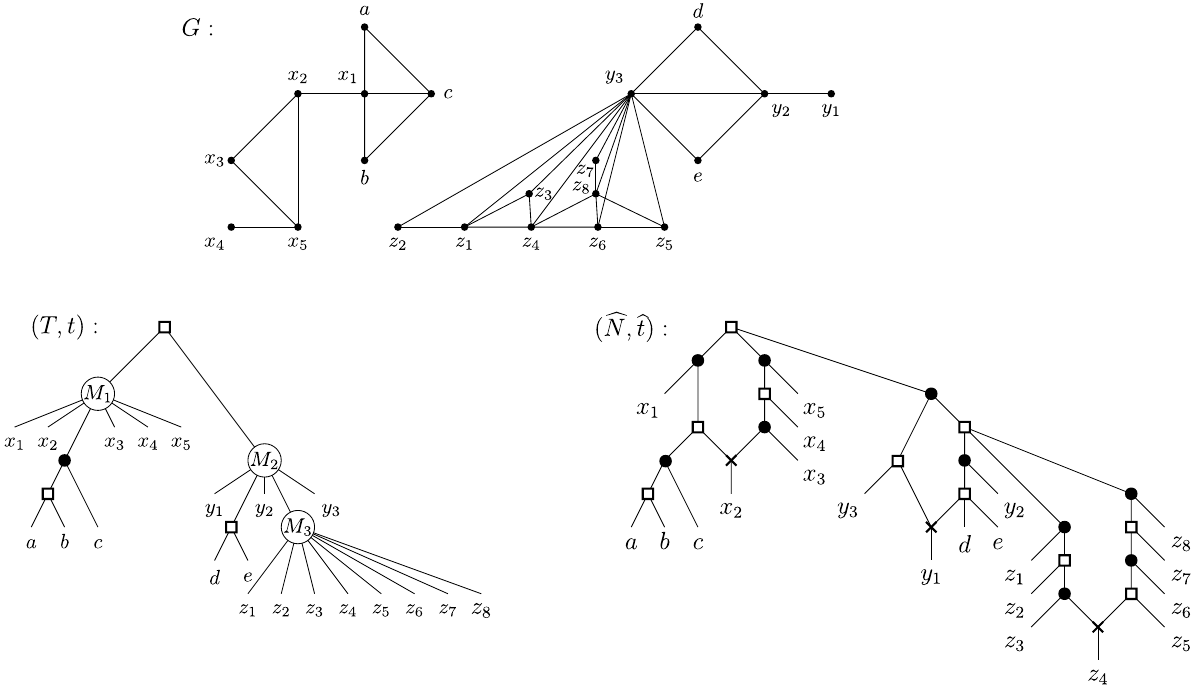}
		\end{center}
		\caption{Shown is a graph $G$ together with its MDT $(T,t)$ that has prime
		         modules $M_1, M_2$ and $M_3$. For each module $M_i$, the quotient
		         $G[M_i]/\Mmax(G[M_i])$ is isomorphic to the graph $G_i$ as shown in
		         Fig.\ \ref{fig:polarCatN}. Replacing each prime module $M_i$ by the
		         respective elementary network that explains $G_i$ (as shown in Fig.\
		         \ref{fig:polarCatN}) and contraction of inner edges whose endpoints have
		         the same label yields the quasi-discriminating level-1 network
		         $(\widehat N, \widehat t)$ that explains $G$. }
		\label{fig:pvr}
\end{figure}

We are now in the position to characterize graphs that can be explained by labeled
level-1 networks.

\begin{theorem}\label{thm:CharprimeCat}
A graph $G$ can be explained by a  labeled level-1 network  if and only if $G\in  \PrimeCat$.
\end{theorem}
\begin{proof}
	The \emph{if}-direction is an immediate consequence of Prop.\
\ref{prop:well-defPVR}. For the \emph{only-if}-direction assume that $G$ can be
explained by a labeled level-1 network $(N,t)$. If $N$ is weak or a tree, then
Thm.\ \ref{thm:WeakIffCograph} implies that $G$ is a cograph and thus, $G\in
\PrimeCat$. Assume that $G$ is not a cograph and thus, $G$ contains at least one
prime module $M$. By Obs.\ \ref{obs:M4}, $|M|\geq 4$. Consider the quotient
$G'\coloneqq G[M]/\Mmax(G[M])$ with $\Mmax(G[M]) = \{M_1 , \dots , M_k\}$. By
Obs.\ \ref{obs:quotient}, $G'\simeq G[W]$ with $W\subseteq M$ such that $\forall
i \in \{1, \dots,k\}$ we have $|M_i \cap W | = 1$. This together with Lemma
\ref{lem:level-1-inher} implies that $G'$ can be explained by a labeled level-1
network. Since $M$ is a prime module, $G'$ is primitive. The latter two
arguments together with Lemma \ref{lem:primitive-2} imply that $G'$ is a
polar-cat. Since the latter arguments hold for every prime module of $G$ we can
conclude that $G\in \PrimeCat$, which completes this proof. 
\end{proof}

\begin{corollary}
	$G\in \PrimeCat$ if and only if   $H\in \PrimeCat$ for every induced subgraph $H$ of $G$.
\end{corollary}
\begin{proof}
	The \emph{if}-direction follows from the fact that $G$ is an induced	subgraph of $G$.
	The \emph{only-if}-direction is a consequence of Thm.\ \ref{thm:CharprimeCat}
	and Lemma \ref{lem:level-1-inher}.
\end{proof}

We summarize the relationship between the different graph-classes
$\textsc{Cograph}$, $\textsc{PseudoCograph}$,  $\PolarCat$ and 
$\PrimeCat$ (see also Fig.\ \ref{fig:PseudoCograph} and the results established above).

\begin{proposition}\label{prop:subsets}
$\textsc{Cograph} \cap \PolarCat =
\emptyset$ and $\textsc{Cograph} \union \PolarCat \subsetneq \textsc{PseudoCograph}
\subsetneq \PrimeCat$.
\end{proposition}
\begin{proof}
By definition, a cograph does not contain prime modules and is, in particular,
not primitive. Moreover, Lemma \ref{lem:primitive-1} implies that every graph
$G\in \PolarCat$ is primitive. Hence, a cograph cannot be contained in
$\PolarCat$. Moreover, Cor.\ \ref{cor:CographNotPC} implies that any graph in
$\PolarCat$ is not a cograph. Hence, $\textsc{Cograph} \cap \PolarCat =
\emptyset$.
		
		We continue with showing that $\textsc{Cograph} \union \PolarCat
\subsetneq \textsc{PseudoCograph}$. By Lemma \ref{lem:cographPscograph}, we have
$\textsc{Cograph} \subseteq \textsc{PseudoCograph}$. Moreover, by definition,
$\PolarCat \subseteq \textsc{PseudoCograph}$. Hence, $\textsc{Cograph} \union
\PolarCat \subseteq \textsc{PseudoCograph}$. Now, consider the pseudo-cograph
$G$ in Fig.\ \ref{fig:gamma}. $G$ contains the non-trivial module $\{1,2\}$ and
is, therefore, not primitive. Contraposition of Lemma \ref{lem:primitive-1}
shows that $G\notin \PolarCat$. Moreover, since $G$ contains induced $P4$s,
$G\notin \textsc{Cograph}$. Hence, $\textsc{Cograph} \union \PolarCat \subsetneq
\textsc{PseudoCograph}$.

		Furthermore, by Prop.\ \ref{prop:PsC-l1N}, every pseudo-cograph can be
explained by a labeled level-1 network. Thm.~\ref{thm:CharprimeCat} implies that
$\textsc{PseudoCograph} \subseteq \PrimeCat$. Now consider that the graph $G$ as
in Fig.\ \ref{fig:non-pseudo-cograph}. Since $G$ can be explained by a labeled
level-1 network, Thm.~\ref{thm:CharprimeCat} implies that $G\in \PrimeCat$.
However, $G\notin \textsc{PseudoCograph}$, since there is no single vertex $v\in
V(G)$ such that $G-v$ is a cograph (cf.\ Obs.\ \ref{obs:G-v-Cograph}). Thus,
$\textsc{PseudoCograph}\subsetneq \PrimeCat$.
\end{proof}

Note that many graphs are not contained in the class $\PrimeCat$. The simplest
example is an induced $P_6$. To see this, observe that Cor.\ \ref{cor:LongestPath}
implies that a $P_6$ is not a pseudo-cograph. Hence, $P_6\notin \PolarCat$. Since a
$P_6$ is primitive, the MDT of a $P_6$ consists of a single root with label ``prime''
and six leaves corresponding the vertices of this $P_6$. Since for $M=V$ the quotient
satisfies $G[M]/\Mmax(G[M])\simeq P_6$, this quotient is not a polar-cat and
therefore, $P_6\notin\PrimeCat$. Further examples of graphs that cannot be explained
by level-1 networks are shown in Fig.\ \ref{fig:forb-primitive}.
\begin{figure}[t]
		\begin{center}
			\includegraphics[width = .7\textwidth]{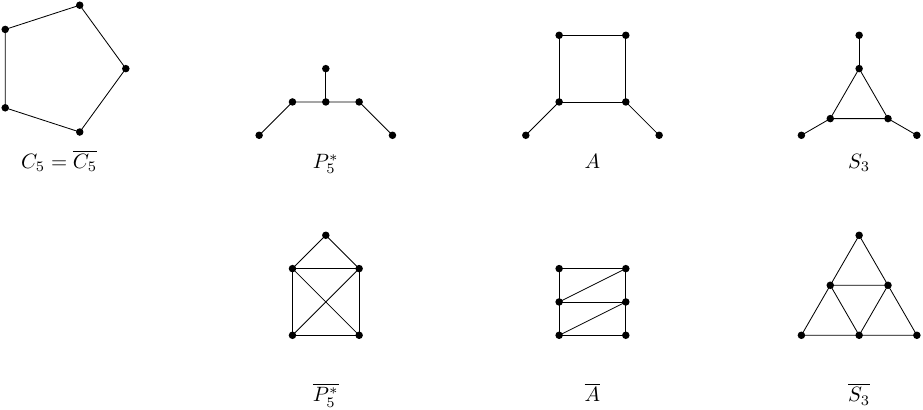}
		\end{center}
	\caption{Shown are several primitive graphs $G\in \{C_5,\overline{C_5},P_5^*,
	         \overline{P_5^*}, A, \overline{A}, S_3, \overline{S_3}\}$. In all cases,
	         application of Lemma \ref{lem:ConnComp-v} shows that none of these graph
	         is a pseudo-cograph and thus, also no polar-cat. Moreover, since every $G$ is
	         primitive, Thm.\ \ref{thm:CharPCPC} implies that none of these graph can
	         be explained by a labeled level-1 network. In particular, since these
	         graphs are primitive they cannot be contained as an induced subgraph in
	         any graph $G\in \protect\PrimeCat$.}
		\label{fig:forb-primitive}
\end{figure}

\section{Least-Resolved Level-1 Networks and Uniqueness Results}
\label{sec:least-resolved}

Thm.~\ref{thm:CharprimeCat} shows that the graphs contained in $\PrimeCat$ are
precisely the graphs that can be explained by some labeled level-1 network. In
this section, we show under which conditions such a network is uniquely
determined. To this end, we show first which (subsets of) leaves in $(N,t)$
correspond to strong modules in $\mathscr{G}(N,t)$. Moreover, we show that there
is a 1-to-1 correspondence between cycles in a network $(N,t)$ and prime modules
of $\mathscr{G}(N,t)$. We are then interested in the type of edges that can be
contracted in a network $(N,t)$ such that the resulting network still explains
$\mathscr{G}(N,t)$. This, eventually, is used to show under which conditions a
level-1 network is least-resolved and when such least-resolved networks are
uniquely determined. 

In the following, we make frequent use of sets $\widetilde L^C_N(v)$ as defined as 
follows.
\begin{definition}\label{def:widetildeL}
	Let $N$ be a network and $C\subseteq N$ be a cycle of $N$. Then, for all
	$v\in V(C)$, we denote with $\widetilde L^C_N(v)$ the set of all descendant
	leaves of $v$ in $N$ that are not descendants of any vertex $w\in V(C)$ with
	$w\prec_N v$. 		
\end{definition}
Note that, for some cycles $C$ and $C'$ it may hold that $v=\eta_C=\rho_{C'}$ in which case 
$\widetilde L^C_N(v) = L(N(v))$ and where  $\widetilde L^C_N(v)\neq \widetilde L^{C'}_N(v)$
is possible. In particular, $\widetilde L^{C'}_N(v)$ could be empty.

In the following, we show that every strong quasi-discriminating level-1 network
$(N,t)$ that explains a given graph can be obtained from a pvr graph after a (possibly
empty) sequence of edge contractions. In other words, non-uniqueness of the networks
$(N,t)$ that explain $G$ mainly depends on the choice of the elementary networks
$(N_v,t_v)$ that are used to replace prime vertices $v$ in the MDT $(\MDT_G,\tau_G)$
of $G$. To this end, we need the following Lemma \ref{lem:strongModuleInCycles} that
provides information of the location and structure of prime and strong modules along
cycles. In this context, we emphasize that every strong module of a graph $G$
coincides with the leaf set $L(\MDT_G(v))$ for some $v\in V(\MDT_G)$ in the MDT of
$G$ which, however, does not longer remain true in networks that contain strong
cycles $C$. In the latter case, some strong modules are defined by the set
$\widetilde L^C_N(v)$. Hence,
$\widetilde L^C_N(v)$ is in many cases a proper subset of $L(N(v))$ and strong modules do
not necessarily coincide with the leaf set $L(N(v))$. 

\begin{lemma}\label{lem:strongModuleInCycles}
Let $(N,t)$ be a strong level-1 network on $X$, $C$ be a quasi-discriminating
cycle in $N$ with root $\rho_C$ and $G=\mathscr{G}(N,t)$. Put $\widetilde
L=\widetilde L^C_N$. Then, $M \coloneqq L(N(\rho_C))\setminus \widetilde
L(\rho_C)$ and $\widetilde L(v)$ with $v\in V(C)\setminus \{\rho_C\}$ are strong
modules of $G$. In particular, $M$ is a prime module of $G$ with maximal modular
decomposition $\Mmax(G[M]) = \{\widetilde L(v) \mid v \in
V(C)\setminus\{\rho_C\}\}$.
\end{lemma}
\begin{proof}
Let $(N,t)$ be a strong level-1 network on $X$, $C$ be a quasi-discriminating
cycle in $N$ with root $\rho_C$ and $G=\mathscr{G}(N,t)$. Note that $M\coloneqq
L(N(\rho_C))\setminus \widetilde L(\rho_C)$ is precisely the set $L(N(u))\cup
L(N(u'))$ for the two unique children $u$ and $u'$ of $\rho_C$ in $V(C)$. It is
straightforward to verify that $M$ and $\widetilde L(v)$, $v\in V(C)\setminus
\{\rho_C\}$ are modules. 

We show now that the set $\widetilde L(v)$ is a strong module in $G[M]$ for all $v\in
V(C)\setminus \{\rho_C\}$ and that $\widetilde L(v)$ is an inclusion-maximal module
of $G[M]$ distinct from $M$ for all $v\in V(C)$, i.e., $\Mmax(G[M]) = \{\widetilde
L(v) \mid v \in V(C)\setminus\{\rho_C\}\}$. 

Let us start with showing that $\widetilde L(v)$ is a strong module in $G[M]$ for all
$v\in V(C)\setminus \{\rho_C\}$. To this end, we modify the network induced by the
vertices in $N(u), N(u')$ and $\rho_C$ with $u, u'\in V(C)$ being the two unique
children of $\rho_C$ in $C$ as follows: Take $C$ and, for all $v\in V(C)\setminus
\{\rho_C\}$, remove all paths in $N$ from $v$ to any leaf in $\widetilde L(v)$ and
add an edge $(v,\widetilde v)$. This results in a strong elementary
quasi-discriminating network $(\widetilde{N}, \widetilde{t})$. Let $\widetilde
G\coloneqq \mathscr{G}(\widetilde{N}, \widetilde{t})$ be the graph with vertex set
$\{\widetilde v\mid v\in V(C)\setminus \{\rho_C\}\}$ that is explained by
$(\widetilde{N}, \widetilde{t})$. Thm.\ \ref{thm:CharPolarCatPC} implies that
$\widetilde G$ is a polar-cat. Thus, we can apply Lemma \ref{lem:primitive-1} to
conclude that $\widetilde G$ is primitive. Hence, it contains only the modules
$V(\widetilde G)$ and the singletons $\{\widetilde v\}$ with $v\in V(C)\setminus
\{\rho_C\}\}$. In particular, all these modules are strong modules of $\widetilde G$
and the singletons are the only (inclusion-maximal) modules that are distinct from
$V(\widetilde G)$. This together with the fact that each vertex $\{\widetilde v\}$ in
$\widetilde G$ is uniquely identified with the module $\widetilde L(v)$ of $G$
implies that $\widetilde L(v)$ must be a strong module of $G[M]$ for all $v\in
V(C)\setminus \{\rho_C\}$ and that $\widetilde L(v)$ is an inclusion-maximal module
of $G[M]$ and distinct from $M$ for all $v\in V(C)$. Therefore, $\Mmax(G[M]) =
\{\widetilde L(v) \mid v \in V(C)\setminus\{\rho_C\}\}$. 

By construction, $G[M]/\Mmax(G[M]) \simeq \widetilde G$ and thus, $G[M]/\Mmax(G[M])$
is primitive. Hence, $M$ is a prime module of $G[M]$ and thus, $G[M]$ and
$\overline{G}[M]$ must be connected. This together with the fact that $M$ is a module
of $G$ implies that $M$ is a prime module of $G$. By Lemma 3.4 in \cite{HFWS:20}, $M$
is a strong module of $G$. In addition, Lemma 3.1 in \cite{HFWS:20} states that $M'$
is a strong module of $G[M]$ if and only if $M'$ is a strong module of $G$. Hence, we
can conclude that $\widetilde L(v)$ is strong module of $G$ for all $v\in
V(C)\setminus \{\rho_C\}$. 
\end{proof}

\begin{proposition}\label{prop:1-1-prime-strong}
Let  $(N,t)$ be a strong level-1 network that explain $G$
and for which all cycles are quasi-discriminating.
Then, there is a 1-to-1 correspondence between cycles in $N$ and prime modules of $G$.
\end{proposition}
\begin{proof}	
	Let $(N,t)$ be an arbitrary strong level-1 network that explains $G=(V,E)$
	and assume that all cycles of $N$ are quasi-discriminating. Assume
	first that $G$ is a cograph. In this case, $(\MDT_G,\tau_G)$ does not
	contain prime vertices. Moreover, Thm.\ \ref{thm:WeakIffCograph}
	implies that any network $(N',t')$ that explains $G$ must be weak. Since
	$(N,t)$ is strong and explains the cograph $G$, it can therefore not contain
	any cycle and  the statement is vacuously true. Assume that $G$ is not a
	cograph and let $M$ be an arbitrary prime module of $G$. Let $G'\coloneqq
	G[M]/\Mmax(G[M])$ where $\Mmax(G[M]) = \{M_1, \ldots, M_k \}$. By Obs.\
	\ref{obs:quotient}, $G'\simeq G[W]$ with $W\subseteq M\subseteq L(N)$
	such that $\forall i \in \{1, \dots,k\}$ we have $|M_i \cap W | = 1$. As
	argued in the proof of Lemma \ref{lem:level-1-inher}, a network $(N',t')$
	that explains $G[W]$ can be obtained from $(N,t)$ byremoving
	step-by-step the leaves of $N$ not in $W$ and by repeating the
	following four steps until no vertices and edges that satisfy (1), (2), (3)
	and (4) exist: (1) suppression of indegree 1 and outdegree 1 vertices; (2)
	removal of indegree 2 and outdegree 0 vertices and its two incident edges;
	(3) removal of indegree 0 and outdegree 1 vertices and its incident edge;
	and (4) removal of all but one of possible resulting multi-edges. Since
	$G[W]$ is primitive and is explained by $(N',t')$, Lemma
	\ref{lem:primitive-2} implies that $(N',t')$ must be a strong elementary
	quasi-discriminating network. Let $C'$ denote the unique (strong
	quasi-discriminating) cycle in $N'$. Since $(N',t')$ is obtained from
	$(N,t)$ it follows that $C'$ must be obtained from a cycle $C$ in $N$
	after application of a (possibly empty) sequence of Steps (1)-(4).
	However, since Steps (2)-(4) involve edge or vertex removals and since the
	maximal biconnected components in $N$ are precisely the cycles, only Step
	(1) could have been applied to obtain $C'$ from $C$. To simplify the
	notation, put $\widetilde{L'}\coloneqq \widetilde L_{N'}^{C'}$ and
	$\widetilde L\coloneqq \widetilde L_{N}^{C}$ Since $N'$ is elementary it
	holds that $\rho_{C'}=\rho_{N'}$ and $\widetilde{L'}(\rho_{C'})=\emptyset$
	and, therefore, $W = L(N'(\rho_{C'}))\subseteq L(N(\rho_C))$. Since $W$ can
	be chosen arbitrarily as long as $W \subseteq M$ and $|M_i \cap W | = 1$ for
	all $i \in \{1, \dots,k\}$ hold, and since $\cup_{i=1}^k M_i=M$, we can
	conclude that $M \subseteq K\coloneqq L(N(\rho_C))\setminus
	\widetilde{L}(\rho_C)$. We next show that these sets are, in fact, equal.
	Too see this, assume for contradiction that $M \subsetneq K$. By Lemma
	\ref{lem:strongModuleInCycles}, $K$ is a strong prime module of $G$ with
	$\Mmax(G[K]) = \{\widetilde{L}(v) \mid v \in V(C)\setminus\{\rho_C\}\}$.
	Since $M\subsetneq K$ and since $M$ is a strong module of $G$ (and thus does
	not overlap with any other module of $G$), it must hold that $M\subseteq
	\widetilde{L}(v)$ for some $v \in V(C)\setminus\{\rho_C\}$. But then $W =
	L(N'(\rho_{C'})) \subseteq M \subseteq \widetilde{L}(v)$ must hold and thus,
	$C$ and $C'$ can have at most vertex $v$ in common. In this case, however,
	$C'$ cannot have been obtained from $C$ by suppression of vertices only; a
	contradiction. Hence, $M = L(N(\rho_C))\setminus \widetilde{L}(\rho_C)$ must
	hold and thus, every prime module $M$ of $G$ is associated with a cycle
	$C\subseteq N$. Moreover, reusing the latter arguments, we have for distinct
	prime modules $M$ and $M'$ that $M = L(N(\rho_C))\setminus
	\widetilde{L}(\rho_C)\neq M' = L(N(\rho_{C'}))\setminus
	\widetilde{L}_N^{C'}(\rho_{C'})$ for some
	cycles $C,C'\subseteq N$ and thus $C\neq C'$. Thus, every prime module $M$
	of $G$ is associated with a unique cycle $C\subseteq N$. 
	
	Moreover, since every cycle of $N$ is quasi-discriminating, we can apply
	Lemma \ref{lem:strongModuleInCycles} to conclude that every cycle
	$C\subseteq N$ is associated with the unique prime module $M =
	L(N(\rho_C))\setminus \widetilde{L}(\rho_C)$. Therefore, there is a 1-to-1
	correspondence between prime modules of $G$ and cycles in $N$.
\end{proof}

The restriction to $\widetilde L^C_N(v)$ as in Lemma \ref{lem:strongModuleInCycles} 
is, however, not always required. To see this, 
we continue with characterizing modules of the form $L(N(v))$.

\begin{lemma}\label{lem:strong-nonCycle}
	Let $(N,t)$ be a strong level-1 network on $X$
	that explains $G$ and for which all cycles are quasi-discriminating. Furthermore, let $v \in V(N)$.
	Then, the set $L(N(v))$ is a module of $G$ if and only if there exists 
 	no cycle $C$ of $N$ such that $v \in V(C)\setminus \{\rho_C,\eta_C\}$. Moreover, if the 
 	latter holds and $(N,t)$ is quasi-discriminating, then $L(N(v))$ is a strong module of $C$.
\end{lemma}
\begin{proof}
	Suppose first that there exists a cycle $C$ of $N$ such that $v \in
V(C)\setminus \{\rho_C,\eta_C\}$. Let $x \in \widetilde L_N^C(v)$ and $y \in
L(N(\eta_C))$. Since $C$ is strong and quasi-discriminating, there exists a
vertex $w \in V(C)$ such that $w$ and $v$ lie on different sides of $C$, and
$t(w) \neq t(\rho_C)$. Let $z \in \widetilde L_N^C(w)$. Note that $x,y \in
L(N(v))$, while $z \notin L(N(v))$. It holds that $\lca(x,z)=\rho_C$ and
$\lca(y,z)=w$. By choice of $w$, it follows that $t(\lca(x,z)) \neq
t(\lca(y,z))$ which implies that $x$ is adjacent to $z$ in $G$ while $y$ is not,
or vice versa. Hence, $L(N(v))$ is not a module of $G$ which establishes the
\emph{only-if} direction.
	
	For the \emph{if} direction, assume that $v\in V(N)$ is a vertex such that
there is no cycle $C\subseteq N$ with $v \in V(C)\setminus \{\rho_C,\eta_C\}$.
Put $M\coloneqq L(N(v))$. We show first that $M$ is a module of $G$. If
$v=\rho_N$ and thus, $M=X$, then we are done. Similarly, if $v\in L(N)$ or $v$
is a vertex that has precisely one child and this child is a leaf then $|M|=1$
and there is nothing to show. In all other cases, $|M|>1$ must holds. Hence,
assume that $v$ is a vertex such that $|M|>1$. Let $x,y\in M$. Since $M\neq X$,
we can choose a vertex $z\in L(N)\setminus M$. We next show that $\lca_N(x,z) =
\lca_N(y,z)$. To see this, put $v_x=\lca_N(x,z)$ and $v_y=\lca_N(y,z)$. Since $z
\notin M=L(N(v))$, neither $v_x$ nor $v_y$ are descendant of $v$. Suppose now
that one of these vertices, say $v_x$, is not an ancestor of $v$. In this case,
$v_x$ and $v$ must be $\preceq_N$-incomparable. Since $x\preceq_N v_x$ and
$x\preceq_N v$ hold, there must be a hybrid-vertex $\eta_C$ of some cycle $C$
that satisfies $x \preceq_N \eta_C$, $\eta_C \preceq_N v_x $ and $\eta_C
\preceq_N v$. In particular, $v$ and $v_x$ belong to different sides of $C$ and
$\eta_C\neq v$; a contradiction to the choice of $v$. Hence, $v_x$ and $v_y$
must be ancestors of $v$. In particular, we have $y \in L(N(v_x))$ and $x \in
L(N(v_y))$. Since $z$ also belongs to $L(N(v_x))$ and to $L(N(v_y))$, it follows
that $v_x$ is an ancestor of $\lca_N(y,z)=v_y$, and that $v_y$ is an ancestor of
$\lca_N(x,z)=v_x$, so $\lca_N(x,z)= \lca_N(y,z)$ must hold. Hence,
$t(\lca_N(x,z))=t(\lca_N(y,z))$. Thus, $N_G(x)\setminus M = N_G(y)\setminus M $.
Since $x,y\in M$ and $z\in z\in L(N)\setminus M$ were chosen arbitrarily, $M$ is
a module of $G$.

	We continue with showing that $M$ is a strong module of $G$ under the
assumption that $(N,t)$ is, in addition, quasi-discriminating. To simplify
writing, we call a vertex $w\in V(N)$ \emph{qualified} if there is no cycle
$C\subseteq N$ such that $w \in V(C)\setminus \{\rho_C,\eta_C\}$. By assumption,
$v$ is qualified. Assume, for contradiction, that $M$ is not strong. By
\cite[Lemma 3.4]{HFWS:20}, $M$ cannot be a prime module. Note that
$L(N(\rho_N))=X$ and thus, $L(N(\rho_N))$ is a strong module. Hence, we can
assume w.l.o.g.\ that $v$ is a ``first top-down'' vertex in $N$ such that
$M=L(N(v))$ is not strong, i.e., for all qualified vertices $w$ with $v\prec_N
w$ the set $L(N(w))$ is a strong module. W.l.o.g.\ assume that $M$ is a parallel
module and, therefore, that $G[M]$ is disconnected (otherwise, consider
$\overline G$ in which case $\overline{G[M]}$ must be disconnected since $M$ is
not a prime module). Since $M$ is not strong, it must overlap with some other
module $M'$ and \cite[Lemma 3.1]{HFWS:20} implies that $M\cap M'$ and $M\cup M'$
must be a module of $G$. In particular, $M\cup M'$ must be a parallel module. To
see this, assume for contradiction that $M\cup M'$ is a series module. In this
case, all vertices $x\in M$ must be adjacent to at least one vertex $y\in
M'\setminus M$, since $G[M]$ is disconnected, $M$ a module and $G[M\cup M']$ is
connected. Moreover, since $G[M]$ is disconnected, there is a vertex $z\in M\cap
M'\subseteq M$ and a vertex $x\in M\setminus M'$ such that $z$ is not adjacent
to $x$. Since $y,z\in M'$ and $y$ is adjacent to all $x\in M$ and thus, to all
$x\in M\setminus M'$ but $z$ is not adjacent to all $x\in M\setminus M'$, $M'$
cannot be a module of $G$; a contradiction. Hence, $M\cup M'$ is a parallel
module. By \cite[Cor.\ 7.10]{HSS:22cluster}, there is a unique vertex $w\in
V(N)$ such that $w=\lca_N(M\cup M')$. Clearly, the outdegree of $w$ must be at
least 2, since otherwise $w$ is a leaf or a hybrid-vertex with a single child,
in which cases $w$ cannot be the last common ancestor of any subset of $X$.
Moreover, $M=L(N(v))\subsetneq (M\cup M') \subseteq L(N(w))$ implies together
with \cite[Obs.\ 4.3 + L.\ 7.1]{HSS:22cluster} that $v\prec_N w$. One easily
verifies that $t(v)=t(w)=0$, since otherwise, $G[M]$ and $G[M\cup M']$ would be
connected. Since $(N,t)$ is quasi-discriminating and $v$ not a hybrid-vertex,
$(w,v)\notin E(N)$. Hence, there is a vertex $u$ with $v\prec_N u \prec_N w$.
There are two cases: (i) $u$ is qualified or (ii) $u$ is not qualified in which
case there is a cycle $C$ such that $u\in V(C)\setminus \{\rho_C, \eta_C\}$. In
Case (i), put $u^*\coloneqq u$ and $M^*\coloneqq L(N(u))$ which is a strong
module by the choice of $v$. Now consider Case (ii). Since $N$ is level-1, for
any two vertices $a,b\in V(C)\setminus \{\rho_C\}$ it must hold that $\widetilde
L_N^C(a)\cap \widetilde L_N^C(b)=\emptyset$. In particular, $L(N(u))$ is the
disjoint union $\cupdot_{i=1}^k \widetilde L_N^C(w_i)$ with $w_i\preceq_N u$ and
$w_i\in V(C)$. Since $N$ is level-1 it is now an easy task to verify that
$L(N(v))\subseteq \widetilde L_N^C(w_i)$ for one of these vertices. Hence, in
Case (ii), $u^*\preceq_N u$ denotes the unique vertex in $C$ such that
$L(N(v))\subseteq \widetilde L_N^C(u^*)$ and we put $M^*=\widetilde L_N^C(u^*)$
which is a strong module by Lemma \ref{lem:strongModuleInCycles}. By
construction and \cite[L.\ 3.33]{HSS:22cluster}, we have $M=L(N(v))\subseteq
M^*\subseteq L(N(w))$. Note that $M\cup M' \subseteq M^*$ is not possible since
then $\lca_N(M\cup M')\preceq_N u^*$; a contradiction to $u^* \prec_N
w=\lca_N(M\cup M')$ (cf.\ \cite[Obs.\ 6.4]{HSS:22cluster}). This together with
$M\subseteq M^*$ implies that there is a $x'\in M'$ such that $x'\not\in M^*$
and thus, $x'\in M'\setminus M^*$. Furthermore, since $M$ and $M'$ overlap and
$M\subseteq M^*$, there is a $x\in M\setminus M'\subseteq M^*\setminus M'$.
Moreover, $M\cap M'\subseteq M \subseteq M^*$ implies that $ M^*\cap M'\neq
\emptyset$. In summary, $M'$ and $ M^*$ must overlap; a contradiction since $
M^*$ is a strong module. Consequently, $L(N(v))$ must be a strong module for all
qualified vertices which completes the proof.
\end{proof}

The results in Lemma \ref{lem:strongModuleInCycles} and \ref{lem:strong-nonCycle}
show which subsets of leaves form (strong) modules of the graph $G$ under consideration. 
We continue now with showing that, in fact, every strong module is ``displayed'' by
the network that explains $G$.

\begin{lemma}\label{lm-wherestrong}
	Let $(N,t)$ be a strong level-1 network on $X$ that explains $G$ and for
	which all cycles are quasi-discriminating. Then, every strong module $M$ of
	$G$ is \emph{displayed by} $N$, that is, $M$ is one of the sets $L(N(v))$,
	$\widetilde L^C_N(v)$, or $L(N(\rho_C))\setminus \widetilde L_N^C(\rho_C)$
	for some $v \in V(N)$ and, possibly, a cycle $C\subseteq N$.
\end{lemma}
\begin{proof}
	Let $(N,t)$ be a strong quasi-discriminating level-1 network on $X$ that
	explains $G$. Clearly, if $v\in X$, then $L(N(v))=\{v\}$ is displayed by
	$N$. Assume, for contradiction, that not all strong modules $M$ of $G$ are
	displayed by $N$. By the latter argument, $|M|>1$ must hold. Since the
	strong module $X=L(N(\rho_N))$ is displayed by $N$, we can choose among the
	strong modules that are not displayed by $N$, a strong module $M$ of $G$
	such that the inclusion-minimal strong module $\hat M$ of $G$ that contains
	$M$ is displayed by $N$. By Lemma \ref{lem:strongModuleInCycles} and Prop.\
	\ref{prop:1-1-prime-strong}, $\hat M$ cannot be prime since, otherwise,
	$\hat M=L(N(\rho_C))\setminus \widetilde L_N^C(\rho_C)$ for some cycle $C$
	in $N$, in which case $M=\widetilde L_N^C(v)$ for some $v\in V(C)$, and
	thus, $M$ would be displayed by $N$. Moreover, $M$ cannot be prime since,
	otherwise, Lemma \ref{lem:strongModuleInCycles} and Prop.\
	\ref{prop:1-1-prime-strong} imply that $M=L(N(\rho_C))\setminus \widetilde
	L_N^C(\rho_C)$ for some cycle $C\subseteq N$ and thus, $M$ would be
	displayed by $N$. W.l.o.g.\ assume that $\hat M$ is a parallel module. Since
	$\hat M$ is the inclusion-minimal strong module that contains $M$ and since
	both $\hat M$ and $M$ are strong modules it follows that $\hat M$ and $M$
	are adjacent in the MDT $(\MDT_G,\tau_G)$ of $G$. Since $(\MDT_G,\tau_G)$ is
	discriminating, $\tau_G(\hat M)\neq \tau_G(M)$ which, together with the fact
	that $M$ is not prime and $|M|>1$, implies that $M$ is a series module. In
	particular, $G[\hat M]$ is disconnected and $G[M]$ must be connected. Note,
	since $\hat M$ is not prime, $\hat M= L(N(\rho_C))\setminus \widetilde
	L_N^C(\rho_C)$ is not possible by Lemma \ref{lem:strongModuleInCycles}. This
	together with the assumption that $\hat M$ is displayed by $N$ implies that
	there is a vertex $v$ such that $\hat M = \mathcal L\in \{L(N(v)),
	\widetilde L^C_N(v) \text{ for some cycle } C \text{ in the latter case}\}$. 
   
	Consider the subgraph $N'$ of $N$ induced by all vertices that are located
	on some path from $v$ to $x \in \hat M$. Note that in case, $\mathcal L=
	\widetilde L^C_N(v)$, the vertex $v$ may have outdegree 1 in $N'$ and thus
	$\widetilde L^C_{N}(v) = \widetilde L^C_{N'}(v')$ for the unique child $v'$
	of $v$ in $N'$ which must then have outdegree at least 2. In this case, we
	remove $v$ and its incident edge from $N'$ and rename $v'$ as $v$. It is now
	an easy task to verify that $N'$ remains a strong level-1 network. Moreover,
	keeping the labels of all vertices in $N'$ yields a network $(N',t')$ whose
	cycles are all quasi-discriminating. Since $\hat M$ is not prime, Lemma
	\ref{lem:strongModuleInCycles} and Prop.\ \ref{prop:1-1-prime-strong} imply
	that not all children of $v$ can be contained in a single cycle $C\subseteq
	N'$. Let $\mathcal C$ be the set of cycles of $N'$ rooted at $v$ and
	$\child^*(v)$ be the children of $v$ that are not contained in any cycle
	rooted at $v$. Note that at least one $\mathcal C$ and $\child^*(v)$ must be
	non-empty and $|\mathcal C|+|\child^*(v)|\geq 2$, since $\hat M$ is not
	prime. Hence, we can partition $\mathcal L$ into the sets $M_1,\dots, M_k$,
	$k\geq 2$ with $M_i=L(N'(v))\setminus \widetilde L_{N'}^C(v)$, $C\in
	\mathcal C$ or $M_i = L(N(w))$, $w\in \child^*(v)$, $1\leq i \leq k$. If
	$M_i=L(N'(v))\setminus \widetilde L_{N'}^C(v)$, then Lemma
	\ref{lem:strongModuleInCycles} implies that $M_i$ is a strong module. We
	continue with showing that $M_i = L(N(w))$ with $w\in \child^*(v)$ must be a
	strong module as well. To this end, we argue first that, for all $w \in
	\child^*(v)$, there exists no cycle $C$ in $N$ such that $w$ and $v$ are
	vertices of $C$. Assume, for contradiction, that $v$ and $w$ belong both to
	some cycle $C\subseteq N$. If $v=\rho_C$, then $w \notin \child^*(v)$; a
	contradiction. Otherwise, $v \notin \{\rho_C, \eta_C\}$ must hold, in which
	case Lemma~\ref{lem:strong-nonCycle} implies that $\hat M=L(N(v))$ is not
	possible. But then $\hat M = \mathcal L\in \{L(N(v)), \widetilde L^C_N(v)\}$
	implies that $\hat M=\widetilde L_N^C(v)$ must hold for some cycle $C$,
	which, in turn, implies that $w \notin V(N')$; a contradiction. In summary,
	$v$ and $w$ cannot belong to some cycle $C$ of $N$ for all $w\in
	\child^*(v)$. This together with the fact that $N$ is level-1 and $w\prec_N
	v$ implies that, if there is any cycle $C$ containing $w$, then $w=\rho_C$.
	Thus, we can apply Lemma \ref{lem:strong-nonCycle} to conclude that all $M_i
	= L(N(w))$ with $w\in \child^*(v)$ must be a strong module as well. Since
	$\hat M$ is the inclusion-minimal module that contains $M$ and since $M$ is
	strong as well, it must hold that $\cup_{i\in I} M_i\subseteq M$ for some
	non-empty subset $I\subseteq \{1,\dots,k\}$ with $|I|\geq 2$. By Lemma
	\ref{lem:lca}, $\lca_N'(x,y)=v$ for all $x\in M_i$ and $y\in M_j$ with $j\in
	I\setminus \{i\}$ and thus, $t(\lca_{N'}(x,y))=0$ since $\hat M$ is a
	parallel module. But this implies that $G[M]$ must be disconnected; a
	contradiction. Consequently, all strong modules $M$ of $G$ are displayed by
	$N$.
\end{proof}

	To understand in more detail the structure of least-resolved networks 
    that explain a given graph $G$, it is necessary to know which type of 
    edges can or cannot be contracted such that the
    resulting network still explains $G$.  We show first 
	that edges in strong quasi-discriminating
	cycles cannot be contracted.
	
\begin{lemma}\label{lem:noEdgeContractionsAlongCycles}
	Let $(N,t)$ be a level-1 network that explains $G$. 
	If there is a strong quasi-discriminating cycle	$C\subseteq N$,  
	then any contraction
	of an edge $e\in E(C)$ yields a network $N_e$
	for which there is no labeling $t'$ of $N_e$ such that 
	 $(N_e,t')$ explains $G$.
\end{lemma}
\begin{proof}
	Let $(N,t)$ be a level-1 network that explains $G$, $C\subseteq N$ be a
strong quasi-discriminating cycle and $e=(u,v)\in E(C)$. In the following, $C'$
denotes the cycle in $N_e$ that is obtained from $C$ after contracting $e$. In
fact, $C' = C_e$ remains a cycle since $C$ is strong. For simplicity put
$\widetilde L\coloneqq \widetilde L_N^C$ and $\widetilde L'\coloneqq \widetilde
L_{N_e}^{C'}$. 
	
	Suppose first that $v=\eta_C$. Since $C$ is strong, $u\neq \rho_C$ and there
is a second parent $u'$ of $v$ that is also distinct from $\rho_C$. In
particular, $(u',v)\in E(C)$. Let $w$ be the child of $\rho_C$ that satisfies
$u'\preceq_N w$ (note that $u'=w$ may hold, in case $u'$ is a child of
$\rho_C$). Since $C$ is quasi discriminating, $t(\rho_C) \neq t(w)$ holds. Since
$u,w \neq \rho_C$, we have $|\widetilde L(u)|\geq 1$ and $|\widetilde L(w)|\geq
1$. Let $x\in \widetilde L_{N }(u)$, $y \in \widetilde L_{N }(v)=L(N(v))$ and $z
\in \widetilde L_{N }(w)$. We have $\lca_{N}(x,z)=\rho_C$ and $\lca_{N}(y,z)=w$,
and thus, $t(\lca_{N}(x,y)) \neq t(\lca_{N}(x,y'))$. Moreover, one easily
verifies that in $N _e$ we have $x,y \in L(N_e(v_e))$ and $z \in \widetilde
L'(w)$. Hence, $\lca_{N_e}(x,z)=\lca_{N_e}(y,z)$ and, therefore,
$t'(\lca_{N_e}(x,z))=t'(\lca_{N_e}(y,z))$ for any labeling $t'$ of the vertices
of $N_e$. Consequently, there is no labeling $t'$ such that $(N_e,t')$ can
explain $G$. In summary, edges $e=(u,v)\in E(C)$ with $v=\eta_C$ cannot be
contracted to obtain a level-1 network that still explains $G$. 
	
	Suppose now that $u=\rho_C$. Since $C$ is strong, $v\neq \eta_C$ and there
is the second child $w$ of $u$ in $C$ that is also distinct from $\eta_C$. Since
$v,w \neq \rho_C$, $|\widetilde L(v)|\geq 1$ and $|\widetilde L(w)|\geq 1$. Let
$x\in\widetilde L(v)$, $y\in\widetilde L(w)$ and $z\in L(N(\eta_C))$. We have
$\lca_N(x,y)=u$ and $\lca_N(x,z)=v$ which, since $C$ is quasi-discriminating,
implies $t(\lca_N(x,y)) \neq t(\lca_N(x,z))$. In addition, we have
$\lca_{N_e}(x,y)=\lca_{N_e}(x,z)=v_e$. In particular,
$t'(\lca_{N_e}(x,y))=t'(\lca_{N_e}(x,z))$ holds for any labeling $t'$ of the
vertices of $N_e$. Consequently, there is no labeling $t'$ such that $(N_e,t')$
can explain $G$. In summary, edges $e=(u,v)\in E(C)$ with $u=\rho_C$ cannot be
contracted to obtain a level-1 network that still explains $G$.

	Now, let $e=(u,v)\in E(C)$ with $v\neq \eta_C$ and $u\neq \rho_C$. It is
easy to see that $v\neq \rho_C$ and $u\neq \eta_C$. Since both $u$ and $v$ are
not hybrid-vertices in $N$ and $u\neq \rho_C$, we have $|L(N(v))|\geq 2$ and
$|\widetilde L(u)|\geq 1$. Moreover, since $C$ is quasi-discriminating, $t
(u)\neq t (v)$ must hold. Let $x,y\in L(N (v))$ with $\lca_{N }(x,y)=v$ (which
exist by Lemma \ref{lem:lca}) and let $z\in \widetilde L(u)$ One easily verifies
that $\lca_{N }(x,z)=u$. Hence, in $N$ we have $t(\lca_N(x,y))\neq
t(\lca_N(x,z))$. However, in $N _e$, we have $\lca_{N _e}(x,y)=v_e=\lca_{N
_e}(x,z)$ and thus $t'(\lca_{N_e}(x,y))= t'(\lca_{N_e}(x,z))$ for any labeling
$t'$ of $N_e$. Hence, there is no labeling $t'$ of $N _e$ such $(N _e,t')$ can
explain $G$.
	
	Therefore, any contraction of an edge $e\in E(C)$ yields a network $N_e$ for
which there is no labeling $t'$ of $N_e$ such that $(N_e,t')$ explains $G$.
\end{proof}

Lemma \ref{lem:noEdgeContractionsAlongCycles} shows that particular edges cannot
be contracted. We continue with showing that so-called dispensable edges can
always be contracted.

\begin{definition}
	An edge $e=(u,v)$ in a level-1 network $(N,t)$ is \emph{dispensable}
precisely if $e$ is not located on a cycle of $N$ and $t(u)=t(v)$ or $u=\eta_C$
for some cycle $C\subseteq N$ with $v\in V(N)\setminus L(N)$ and
$\child_N(u)=\{v\}$. 
\end{definition}
\begin{lemma}\label{lem:LRTnoEdgeContractionsAlongCycles}
	Let $(N,t)$ be a strong level-1 network whose cycles are all
quasi-discriminating and let $e=(u,v)\in E(N)$. Then, $(N_e,t')$ explains
$\mathscr{G}(N,t)$ for some labeling $t'$ of $N_e$ if and only if $e$ is
dispensable in $(N,t)$. 

	In particular, contraction of all edges $e\in E(N)$ that are dispensable in
$(N,t)$ yields a strong quasi-discriminating least-resolved network that
explains $\mathscr{G}(N,t)$.
\end{lemma}
\begin{proof}
	Let $(N,t)$ be a strong level-1 network that explains $G=\mathscr{G}(N,t)$
and whose cycles are all quasi-discriminating. Let $e=(u,v)\in E(N)$ be an edge
that is not located on a cycle of $N$. Suppose that $(N_e,t')$ explains $G$ for
some labeling $t'$ of $N_e$. If $t(u)=t(v)$, then we are done. Assume that
$t(u)\neq t(v)$. Clearly, $v\in V(N)\setminus L(N)$ must hold since otherwise
$L(N_e)\neq L(N)=V(G)$, in which case $(N_e,t')$ does not explain $G$. Assume,
for contradiction, that (i) $u\neq \eta_C$ for any cycle $C\subseteq N$ or (ii)
if $u= \eta_C$ then $|\child_N(u)|>1$. In Case (i), Condition (N3.a) must hold
and, therefore, $|\child_N(u)|>1$. Hence, in both Cases (i) and (ii), vertex $u$
must have an additional child $v'\neq v$. Note that $|L(N(v'))|\geq 1$. By Lemma
\ref{lem:noEdgeContractionsAlongCycles}, $e$ cannot be located in a cycle of $N$
and thus, $v\neq \eta_C$ for all cycles $C\subseteq N$ and, therefore,
$|L(N(v))|\geq 2$. Let $x,y\in L(N(v))$ such that $\lca_{N}(x,y)=v$ which exist
by Lemma \ref{lem:lca}. Since $u$ and $v$ are not located in a common cycle, one
easily verifies that $L(N(v))\cap L(N(v'))=\emptyset$ and thus, there is a leaf
$z\in L(N(v'))$ with $\lca_{N}(x,z)=u$. Hence, $t(u)\neq t(v)$ implies
$t(\lca_{N}(x,y))\neq t(\lca_{N}(x,z))$. In $N_e$, however, we have
$\lca_{N_e}(x,y) = \lca_{N_e}(x,z)$ and thus, $t'(\lca_{N_e}(x,y)) =
t'(\lca_{N_e}(x,z))$ for all labelings $t'$ of $N_e$. Hence, there is no
labeling $t'$ such that $(N_e,t')$ can explain $G$; a contradiction. Thus, $u =
\eta_C$ and $\child_N(u)=\{v\}$ must hold. This together with $v\in
V(N)\setminus L(N)$ implies that $e$ is dispensable.

	For the converse, assume that $e=(u,v)$ is dispensable in $(N,t)$. Hence,
$e$ is not located on a cycle of $N$ and thus, $v\neq \eta_C$ for all cycles $C$
in $N$. Assume first that $t(u)=t(v)$. Hence, we can use the arguments in the
proof of Prop.\ \ref{prop:NhatN-sameGraph} to show that $(N_e,t_e)$ explains
$G$. Assume now that $t(u)\neq t(v)$ but $v\in V(N)\setminus L(N)$,
$\child_N(u)=\{v\}$ and $u=\eta_C$ for some cycle $C\subseteq N$. This together
with that fact that $N$ is strong implies that $v$ cannot be a hybrid of any
cycle in $N$. Since $\child_N(u)=\{v\}$, vertex $u$ cannot be the last common
ancestor of any two leaves. Hence, we can relabel $u$ by replacing the label
$t(u)$ of $u$ by $t(v)$. In this way, we obtain a labeling $t''$ of $N$ such
that $t''(u)=t''(v)$ and such that $(N,t'')$ still explains $G$. As argued
above, such edges can be contracted to obtain $(N_e,t''_e)$ that explains $G$. 
	
	In summary, $(N_e,t')$ explains $G$ for some labeling $t'$ of $N_e$ if and
only if $e$ is dispensable in $N$.  
	
	We continue with showing that contraction of all edges $e\in E(N)$ that are
dispensable in $(N,t)$ yields a strong quasi-discriminating network that
explains $G$. Reusing the aforementioned arguments, we can always relabel all
vertices $u$ for which $u = \eta_C$ for some cycle $C\subseteq N$ and
$\child_N(u) = \{v\}$ by replacing the label $t(u)$ of $u$ by $t(v)$ which
yields the network $(N,t'')$ that explains the same graph as $(N,t)$. Hence,
contraction of all edges that are dispensable in $(N,t)$ is equivalent to
contract all edges in $e=(u,v)\in E(N)$ with $t''(u)=t''(v)$. By definition,
after contraction of all such edges in $(N,t'')$ the resulting network is
precisely the quasi-discriminating network $(\widehat N, \widehat t'')$ that, by
Prop.\ \ref{prop:NhatN-sameGraph}, explains $G$. Since none of the edges
contained in cycles have been contracted, $(\widehat N, \widehat{t}'')$ remains
strong. 
	
	It remains to show that $(\widehat N, \widehat t'')$ is least-resolved.
Observe that an edge $f \in E(N)\setminus \{e\}$ is dispensable in $(N,t)$
precisely if the corresponding edge $f' \in E(N_e)$ is dispensable in
$(N_e,t')$. Hence, after contraction of all dispensable edges in $(N,t)$ we
obtained the network $(\widehat N, \widehat t'')$ that does not contain any
dispensable edge. The results established above imply that there is no edge at
all that can be contracted in $(\widehat N, \widehat t'')$ such that the
resulting network together with some labeling still explains $G$. Consequently,
$(\widehat N, \widehat t'')$ is least-resolved.
\end{proof}

We are now in the position to prove one of the main results in this section
which shows that one can derive every strong quasi-discriminating level-1
network that explains a given graph $G$ from a pvr graph of the modular
decomposition tree of $G$.

\begin{theorem}\label{thm:1-1-prime-strong}
	Every strong quasi-discriminating level-1 network $(N,t)$ that explains $G$
can be obtained from some pvr graph $(N^*, t^*)$ of $(\MDT_G,\tau_G)$ after a
(possibly empty) sequence of edge contractions. None of the edges in this
sequence are contained in a cycle of $N^*$. 
\end{theorem}
\begin{proof}
Let $(N,t)$ be a strong quasi-discriminating level-1 network that explains $G$. 
	 We now apply the following steps (in this order) on $(N,t)$
	 until no vertices $v$ with the listed properties exist. 
	\begin{enumerate}[noitemsep]
	\item For all hybrid-vertices $v$ that have outdegree at least two
		 ``expand'' $v$ by replacing $v$ by an edge
	       $(v_1,v_2)$ such that the parents of $v$ become parents of $v_1$
	      and all children of $v$ become the children of $v_2$. \smallskip
	\item For all vertices $v$ that are located on a cycle $C$ 
			  but not the root or the hybrid-vertex of $C$ and that have outdegree at least three, 
			  ``expand'' $v$ by replacing $v$ by an edge
		      $(v_1,v_2)$ such that the unique parent of $v$ become the parent of $v_1$
		      the unique child of $v$ that is located on $C$ become a child  of
	          $v_1$ and all other children of $v$ become children of $v_2$.\smallskip

	\item If $v$ is a root of some cycle and has has outdegree at least three, 
			  then let $C_1,\dots,C_k$, $k\geq 1$ be all cycles that are rooted at $v$. In this case, 
	          ``expand'' $v$ by replacing $v$ by edges
		      $(v_1,w_1), \dots, (v_1,w_k)$ such that the parents of $v$ become the parents of $v_1$, 
		      all children of $v$ that are not located on any cycle become children of $v_1$, and 
		      $w_i$ becomes the root $\rho_{C_i}$ of $C_i$, $1\leq i\leq k$.
	\end{enumerate}

One easily verifies that every vertex $v\in V(N)$ cannot satisfy the properties
as in Step (1) and (2) at the same time. However, it is possible that a vertex
$v\in V(N)$ satisfies the he properties as in Step (1) and (3), resp., Step (2)
and (3). Whenever, we replaced in one of the Steps (1), (2) or (3) a vertex $v$
by and edge $(v_1,v_2)$, resp., $(v_1,w_i)$, the new vertices $v_1$ and $v_2$,
resp., $w_i$ obtain the label $t(v)$ and the label of all other vertices remain
unchanged. In this way we obtain a labeled network $(N',t')$. Note that each
Step (1), (2) and (3) is applied at most once to each $v\in V(N)$. Moreover,
since we never changed the internal structure of cycles, $N'$ remains strong.
In addition, since we always ``expanded'' vertices, the
level of $N$ remained unchanged, that is, $N'$ is a level-1 network. After
application of Step (1), all hybrid-vertices $v$ in $N'$ must have a unique
child $w$ such that $t'(v)=t'(w)$. After application of Step (2), all vertices
$u$ in $N'$ that are located on a cycle $C$ but not the root of $C$ have
precisely two children, one child $u_1$ is located in $C$ and the other child
$u_2$ is not. By construction and since $(N,t)$ is quasi-discriminating,
$t(u)\neq t(u_1)$ implies $t'(u)\neq t'(u_1)$. Hence, all cycles in $(N',t')$
remain quasi-discriminating. After application of Step (3), all roots $\rho_C$
of cycles $C$ must have indegree one and outdegree two. This, in particular,
implies that all cycles in $N'$ are vertex-disjoint. It is an easy exercise to
see that $(N,t)$ is precisely the quasi-discriminating network $(\widehat N',
\widehat t')$ obtained from $(N',t')$ as specified at the beginning of Section
\ref{sec:Cog}. By Prop.\ \ref{prop:NhatN-sameGraph}, $(N',t')$ and $(N,t)$
explain the same graph.

By Lemma \ref{lm-wherestrong}, every strong module $M$ of $G$ is displayed by
$N'$ and thus, $M$ is one of the sets $L({N'}(w))$, $\widetilde L^C_{N'}(w)$, or
$L({N'}(\rho_C))\setminus \widetilde L_{N'}^C(\rho_C)$ for some $w \in V(N')$
and, possibly, a cycle $C\subseteq N'$. Since the root $\rho_C$ of every cycle
in $N'$ has outdegree two, we have $\widetilde L_{N'}^C(\rho_C)=\emptyset$ for
all cycles $C\subseteq N'$. Moreover, since cycles in $N'$ are vertex disjoint
and since every vertex $w\in V(C)\setminus \rho_C$ has exactly one child $w'$
that is not located in $C$, it holds that $\widetilde L^C_{N'}(w)=L(N'(w'))$ for
all such vertices $w$. Consequently, every strong modules $M$ of $G$ satisfies
$M=L(N'(w))$ for some $w \in V(N')$. In particular $w$ is either the root of a
cycle of $N'$ or not contained in any cycle at all. 

Let $W\subseteq V(N')$ be the set of all vertices that are located on some cycle
$C\subseteq N'$ but not the root of $C$. Put $V^*\coloneqq V(N')\setminus W$. We
continue with showing that for all $u\in V^*$, the set $M\coloneqq L(N'(u))$ is
a strong module of $G$. Definition of $V^*$ together with the fact that cycles
in $N'$ are vertex disjoint and Lemma \ref{lem:strong-nonCycle} implies $M$ is a
module of $G$. Note, since $(N',t')$ is not quasi-discriminating, we can not
directly apply Lemma \ref{lem:strong-nonCycle} to conclude that $M= L(N'(u))$ is
a strong module of $G$. However, we can employ
Lemma~\ref{lem:strongModuleInCycles} and \ref{lem:strong-nonCycle} on $(N,t)$
and show instead that, for all $u \in V^*$, there exists a vertex $v$ in $N$
and possibly a cycle $C\subseteq N$ such that $L(N'(u))\in \{L(N(v)), \widetilde
L_N^C(v), L(N(\rho_C)) \setminus \widetilde L_N^C(\rho_C)\}$ with $v \neq
\rho_C$ in the second case and $v = \rho_C$ in the last case. 

Let $u \in V^*$. Assume that there exists a vertex $v$ in $N$ such that $M=
L(N'(u))=L(N(v))$. Note that, in this case, there cannot be a cycle $C$ of $N$
such that $v$ is a vertex of $C$ and distinct from the root of $C$. Since $M$ is a
module of $G$ and $(N,t)$ quasi-discriminating, Lemma \ref{lem:strong-nonCycle}
implies that $M$ is a strong module of $G$. Assume now that no such vertex $v$
with $L(N'(u))=L(N(v))$ exists in $N$. It is easy to verify, that $L(N(v)) =
L(N'(v_1))$ whenever Step (1), (2) and (3) was applied to expand $v$.
Furthermore, $L(N(w)) = L(N'(w))$ for all $w\in V(N')\cap V(N)$. In addition, we
have $L(N(v)) = L(N'(v_1)) = L(N'(v_2))$ when Step (1) was applied to expand
$v$. Taking the latter arguments together, $u$ was created as the end-vertex of
some new edge introduced by application of Step (2) or (3) to expand some vertex
$v$, i.e., $u=v_2$ if Step (2) was applied and $u=w_i$ for some $i$ if Step (3)
was applied. Suppose first that $u$ was created by applying Step (2) to expand
$v$. Then, there exists a cycle $C$ in $N$ such that $v$ is not the root nor the
hybrid of $C$. By construction, the children of $u$ are precisely the children
of $v$ that are not descendant to any vertex $v'$ of $C$ satisfying $v'
\prec_{N} v$. Hence, $L(N'(u))=\widetilde L_N^C(v)$ in that case. Since $v\neq
\rho_C$, Lemma \ref{lem:strongModuleInCycles} implies that $L(N'(u))$ is a
strong module of $G$. We now consider Step (3) in which case $u$ must be the
root of some cycle $C$ in $N'$. First observe that both Steps (1) and (2) are
only applied on vertices that exist in $N$. Step (3) however, might be applied
on newly created vertices. To be more precise, a vertex $v\in V(N)$ can satisfy
the properties as in Step (1) and (3) (resp., Step (2) and (3)) at the same
time. In this case, if Step (3) is applied on vertex $u$ then $u=v_2$ where
$v_2$ was created by either applying Step (1) or (2) on $v\in V(N)$. In all
cases, however, $u$ must be the root of some cycle $C$ in $N'$. This cycle must
exist in $N$ and, in particular, the root of $C$ in $N$ is $v$. By construction,
$u$ has outdegree two in $N'$. In particular, all descendants of $u$ in $N'$ are
descendant of some other vertex $u'$ of $C$ distinct from $u$. Hence,
$L(N'(u))=L(N(v)) \setminus \widetilde L_N^C(v)$ with $v=\rho_C$ which together
with Lemma \ref{lem:strongModuleInCycles} implies that $L(N'(u))$ is a strong
module of $G$. In summary, $M= L(N'(u))$ is a strong module of $G$.

By the aforementioned arguments, every strong module $M$ of $G$ is displayed by $N'$
and must satisfy $M=L(N'(u))$ for some $u\in V^*$. Moreover, by construction of 
$N'$, two distinct $u,u'\in V^*$ must satisfy $L(N'(u))\neq L(N'(u'))$. 
Hence, every $u\in V^*$ represents a unique strong module $M=L(N'(u))$. 
In summary, there is a 1-to-1 correspondence between the vertices in $V^*$
and the strong modules of $G$.
We now transform $(N',t')$ to a labeled tree $(\tilde T, \tilde t)$ 
by contraction of all edges that are located on a cycle. Since the cycles in
$N'$ are vertex disjoint, each cycle $C$ of $N'$ refers to a unique vertex $v_C$
in $\tilde T$ and we put $\tilde t(v_C) =\mathrm{prime}$. All other vertices $V$
of $N'$ that remain vertices in $\tilde T$ obtain label $\tilde t(v)=t'(v)$.
It is easy to see that, due to the 1-to-1 correspondence between the vertices in $V^*$
and the strong modules of $G$, there is a 1-to-1 correspondence between the vertices in $V(\tilde T)$
and the strong modules of $G$. Moreover, if $C$ is a cycle in $N$, then 
by Lemma \ref{lem:strongModuleInCycles}, $L(N(\rho_C))$ is a prime module. 
By construction, $L(N(\rho_C)) = L(N(v_C))$ and thus, $v$ is correctly
labeled as ``prime''. For all other vertices $v\in V^*$ the label remains unchanged.
Hence, $(\tilde T, \tilde t)$ is the MDT of $G$ and, by construction 
 $(N',t')$ is a pvr graph. Moreover, by construction, $(N,t)$ can be 
 obtained from  $(N',t')$ by contraction of edges that are not located on 
 any cycle of $N'$ which completes the proof.
\end{proof}

By Lemma \ref{lem:LRTnoEdgeContractionsAlongCycles}, contraction of all
dispensable edges yields a least-resolved network for a given graph $G$. In
particular, least-resolved networks cannot contain any dispensable edge and
thus, are in particular, quasi-discriminating. However, not all
(quasi-)discriminating network are least-resolved, see Fig.\ \ref{fig:LRT}.
Thus, networks that explain the same graph $G$ are not necessarily unique.
Moreover, networks can differ in the choice of of the elementary networks
$(N_v,t_v)$ that are used to replace prime vertices $v$ in the MDT
$(\MDT(G),\tau_G)$ of $G$. Nevertheless, the choice of the elementary networks
$(N_v,t_v)$ becomes unique (up to the label of the hybrid-vertex) if
$G[M]/\Mmax(G[M])$ is a well-proportioned polar-cat for the prime module $M$
associated with $v$. In particular, we obtain
\begin{theorem}\label{thm:uniqueNt}
There is a unique least-resolved strong level-1 network $(N,t)$ 
that explains $G$, if 
the quotient $G[M]/\Mmax(G[M])$ is a well-proportioned polar-cat for every prime
module $M$ of $G$. This network $(N,t)$ must be quasi-discriminating
and the labeling $t$ of $(N,t)$ is unique up to the label of hybrid-vertices
that have only one child and this child is a leaf.
\end{theorem}
\begin{proof}
	If $G$ is a cograph, then $(\MDT_G,\tau_G)$ does not contain prime vertices
	and it coincides with $(N,t)$, i.e., $(N,t)$ is the unique discriminating
	cotree (and thus, the unique least-resolved strong level-1 network) that
	explains $G$. Assume that $G$ contains prime modules and that
	$G[M]/\Mmax(G[M])$ is a well-proportioned polar-cat for every prime module
	$M$ of $G$. In this case, Prop.\ \ref{prop:uniqueNt} implies that each such
	quotient $G[M]/\Mmax(G[M])$ can only be explained by a strong
	quasi-discriminating elementary network $(N_M,t_M)$ that is unique up to the
	label of the hybrid-vertex. In other words, the networks that are allowed to
	replace prime vertices in $(\MDT_G,\tau_G)$ are uniquely determined (up to
	the label of its hybrid-vertex). This implies that the pvr graph $(N^*,
	t^*)$ is uniquely determined (up to the label of its hybrid-vertex) for a
	given MDT $(\MDT_G,\tau_G)$.  By Thm.\ \ref{thm:1-1-prime-strong}, every
	strong quasi-discriminating level-1 network that explains $G$ can be
	obtained from some pvr graph $(N^*, t^*)$ of $(\MDT_G,\tau_G)$ after a
	(possibly empty) sequence of edge contractions while keeping the labels of
	all remaining vertices. In particular, Lemma
	\ref{lem:LRTnoEdgeContractionsAlongCycles} shows that we can assume that all
	dispensable edges in $N^*$ have been contracted, which yields the least-resolved network
	$(N,t)$. Moreover, Lemma \ref{lem:LRTnoEdgeContractionsAlongCycles} implies
	that that none of the contracted edges $(u,v)\in E(N^*)$ are contained in a
	cycle of $N^*$. Hence, all strong quasi-discriminating cycles of $(N^*,t^*)$
	are also present in $(N,t)$. This together with the fact that the set of
	dispensable and, therefore, the set of contracted edges in $(N^*,t^*)$
	to obtain $(N,t)$ are uniquely determined implies that $(N,t)$ is uniquely
	determined for the given pvr graph $(N^*, t^*)$. 

	Moreover, since the MDT of $G$ is unique and since every prime-vertex $v$ in
	$(\MDT_G,\tau_G)$ is associated with a unique prime module $M$ and thus,
	every prime vertex $v$ can only be replaced by the unique (up to the label
	of its hybrid-vertex) strong quasi-discriminating elementary network
	$(N_M,t_M)$ on which no edges can be contracted, it follows that
	$(N,t)$ is the unique least-resolved strong level-1 network that
	explains $G$. Since $(N,t)$ does not contain dispensable edges, it must be
	quasi-discriminating. Henceforth, we call hybrid-vertices ``special'' if
	they have only one child and this child is a leaf. It is easy to see that
	the label of special hybrid-vertices can be chosen arbitrarily, since they
	are not the last common ancestor of any two leaves. As shown in the proof of
	Lemma \ref{lem:LRTnoEdgeContractionsAlongCycles}, all hybrid-vertices in $N$
	that are not special have outdegree at least two. Hence, \emph{all} inner
	vertices in $N$ (except special hybrid-vertices) have outdegree at least
	two. By Lemma \ref{lem:lca}, for \emph{all} inner vertices $v$ distinct from
	special hybrids there are at least two leaves $x,y$ such that
	$v=\lca_N(x,y)$. Therefore, the label $t$ of all such inner vertices cannot
	be changed without violating the property that the resulting network still
	explains $G$. Hence, the labeling $t$ of $N$ is uniquely determined up to
	the label of special hybrid-vertices, which completes the proof.	
\end{proof}

\section{Algorithms}
\label{sec:algo}

In this section, we provide linear-time algorithms for the recognition of
pseudo-cographs, polar-cats as well as graphs in $\PrimeCat$ and for the construction
of level-1 networks to explain them. Pseudocodes are provided in Alg.\
\ref{alg:construct-N}, \ref{alg:pseudoCograph}, \ref{alg:elementary} and
\ref{alg:general} and detailed description can be found in the parts proving their
correctness in Lemma \ref{lem:construct-network}, Thm.\ \ref{thm:PsC-recognition},
\ref{thm:polar-recognition} and \ref{thm:AlgGeneral}, respectively.

\begin{lemma}\label{lem:construct-network}
	Algorithm \ref{alg:construct-N} correctly constructs a level-1 network $(N,t)$
	that explains a pseudo-cograph $G=(V,E)$ in  $O(|V|+|E|)$ time. 
\end{lemma}
\begin{proof}
	If $G$ is a cograph, then the tree $(T,t)$ is returned. Otherwise, $G$ is a
	$(v,G_1,G_2)$-pseudo-cograph. The construction of $(N,t)$ is precisely the label
	level-1 network $(N(v,G_1,G_2),t(v,G_1,G_2))$ as specified in Def.\
	\ref{def:prop:PsC-l1N}. Prop.\ \ref{prop:PsC-l1N} implies that $(N,t)$ is a level-1
	network that explains $G =(V,E)$. For the runtime, observe that computation of the
	cotrees in Line \ref{alg:cotree} and \ref{alg:cotrees-1} can be done in
	$O(|V|+|E|)$ time \cite{Corneil:85}. Moreover, it is an easy task to verify that
	the construction of $(N,t)$ in Line \ref{alg:construct1} and \ref{alg:construct2}
	can be done within the same time complexity. 
\end{proof}

\begin{theorem}\label{thm:PsC-recognition}
	Pseudo-cographs $G=(V,E)$ can be recognized in $O(|V|+|E|)$ time. In the
	affirmative case, one can determine a vertex $v$ and subgraphs $G_1,G_2\subset G$
	such that $G$ is a $(v,G_1,G_2)$-pseudo-cograph and a labeled level-1 network that
	explains $G$ in $O(|V|+|E|)$ time. 
\end{theorem}
\begin{proof}
	To show that pseudo-cographs $G=(V,E)$ can be recognized in $O(|V|+|E|)$ time we
	use Alg.\ \ref{alg:pseudoCograph}. We show first that Alg.\ \ref{alg:pseudoCograph}
	correctly verifies that $G$ is a pseudo-cograph or not. Let $G=(V,E)$ be an
	arbitrary graph. If $|V|\leq 2$, then Line \ref{alg:true} ensures that
	\texttt{true} is correctly returned and the algorithm stops. Suppose that $|V|\geq
	3$. 
	
	We first check in Line \ref{alg:if-cograph-1} if $G$ is a cograph or not. Assume
	that $G$ is a cograph. By Lemma \ref{lem:NONunique-cograph}, $G$ is a
	$(v,G_1,G_2)$-pseudo-cograph for every $v\in V(G)$ and all graphs $G_1$ and $G_2$
	that satisfy $G_1 = G[V_1]$ and $G_2=G[V_2]$ with $V_i \coloneqq
	\{v\}\cup\left(\bigcup_{H\in \mathfrak{C}_i} V(H)\right)$, $1\leq i\leq 2$ for an
	arbitrary bipartition $\{\mathfrak{C}_1, \mathfrak{C}_2\}$ of the connected
	components in the disconnected graph in $\{G-v,\overline{G-v}\}$. The latter task
	is precisely what is computed in Line \ref{alg:if-cograph-1} to
	\ref{alg:vG1G2-CoG}. 
	
	Assume now that $G$ is not a cograph. In this case, $G$ contains an induced $P_4$,
	called here $P$, and we put $U\coloneqq V(P)$ in Line \ref{alg:defU}. By Lemma
	\ref{lem:v-in-allP4}, if $G$ is a $(v,G_1,G_2)$-pseudo-cograph then $v$ must be
	located on all induced $P_4$s of $G$. Hence, it suffices to consider only the
	vertices in $U$ to verify that $G$ is a $(v,G_1,G_2)$-pseudo-cograph. If for all
	$v\in U$ it holds that both $G-v$ and $\overline{G-v}$ are connected, then Thm.\
	\ref{thm:CharCograph} implies that $G-v$ cannot be a cograph for all such $v\in U$
	and thus, Obs.\ \ref{obs:G-v-Cograph} implies that $G$ cannot be a pseudo-cograph.
	In this case, Line \ref{alg:CC}-\ref{alg:end-if-disc} are never executed and
	\texttt{false} is correctly returned in Line \ref{alg:return-noPsC}. 
	
	Assume that $G-v$ or $\overline{G-v}$ is disconnected for some $v\in U$ (Line
	\ref{alg:if-disc}). We then consider the connected components in $\mathfrak{C}$ of
	the disconnected graph in $\{G-v,\overline{G-v}\}$ in Line \ref{alg:CC}. Since $v$
	is located on $P$ and $P$ is self-complementary it follows that at least one edge of 
	$P$ must be contained in a
	connected component of the disconnected graph in $\{G-v,\overline{G-v}\}$. Hence,
	there is a connected component $H'\in \mathfrak{C}$ that contains two vertices of
	$P$. Note, that we can assume w.l.o.g.\ that $G-v$ is disconnected,
	otherwise we compute its complement and proceed with $\overline{G-v}$. Hence, in
	case $v$ is adjacent in $G$ to some vertex in another connected component $H''\in
	\mathfrak{C}\setminus \{H'\}$ then $G[V(H')\cup V(H'')\cup\{v\}]$ must contain an
	induced $P_4$. Speaking in terms of the graph $\Gamma(G,v)$ and referring to Lemma
	\ref{lem:star-center-gamma}, if $G$ is a pseudo-cograph, then the latter implies
	that $H'$ must be a center of the unique star in $\Gamma(G,v)$. In other words,
	$H'$ is a candidate such that $G[V(H')\cup \{v\}]$ and $G[V(G)\setminus V(H')]$ are
	cographs (cf.\ Thm.\ \ref{thm:CharPsG}(C2)) which is checked in Line
	\ref{alg:If-coG1}. 	
	
	To summarize, if $G$ passes the latter tests, then (F1) holds by construction of
	$G_1$ and $G_2$ in Line \ref{alg:induced}, (F2) holds because of Line
	\ref{alg:If-coG1} and (F3) is satisfied because of Line \ref{alg:if-disc} and the
	construction of $G_1$ and $G_2$ in Line \ref{alg:induced}. Hence, $G$ is a
	$(v,G_1,G_2)$-pseudo-cograph and $(v,G_1,G_2)$ is correctly returned in Line
	\ref{alg:retG1G2-2}. If any of these tests fail for all $v\in U$, \texttt{false} is
	correctly returned in Line \ref{alg:return-noPsC}. 
	
	We proceed now with determining the runtime of Alg.\
	\ref{alg:pseudoCograph}. We emphasize first that an explicit construction of
	the complement of $G-v$ is not needed (in case $G-v$ is connected) to
	determine the connected components of $\overline{G-v}$, see
	\cite{ITO1998209}. In particular, Line \ref{alg:CC-coG}, \ref{alg:if-disc}
	and \ref{alg:CC} can be performed in $O(|V|+|E|)$ time using breadth first
	search.  	
	Line	\ref{alg:true} takes constant time. Checking whether $G=(V,E)$ is a cograph or not
	in Line \ref{alg:if-cograph-1} can be done $O(|V|+|E|)$ time with the algorithm of
	Corneil et al.\ \cite{Corneil:85}. This incremental algorithm constructs the cotree
	of a subgraph $G[V']$ of $G$ starting from a single vertex and then increasing $V'$
	by one vertex at each step of the algorithm. If $G$ is not a cograph, then at some
	point there is a set $V'$ and a chosen vertex $w$ such that $G[V']$ is a cograph
	and $G(V'\cup \{w\})$ contains an induced $P_4$. Based on these findings, Capelle
	et al.\ \cite{capelle1994cograph} showed that one can find an induced $P_4$
	containing $w$ in $O(\deg_G(w)) \subseteq O(|V|)$ time. Hence, Line
	\ref{alg:if-cograph-1} and \ref{alg:defU} can be done in $O(|V|+|E|)$. Moreover,
	the subtasks in Line \ref{alg:CC-coG} to \ref{alg:vG1G2-CoG} can be done within the
	same time complexity. Since $|U|=4$ the \emph{for}-loop in Line \ref{alg:for-v-U}
	runs at most four times	 and the computation in Line \ref{alg:if-disc} can be done
	in $O(|V|+|E|)$ time. Within the same time-complexity we can determine with a
	simple breadth-first search the set $\mathfrak{C}$ and the connected component $H'$
	that contains two vertices of the respective induced $P_4$ without the explicit
		construction of the complement \cite{ITO1998209}. The induced subgraphs
	in Line \ref{alg:induced} can be computed in $O(|V|+|E|)$ time. We finally check
	two times in Line \ref{alg:If-coG1} whether $G_1$ and $G_2$ are cographs in
	$O(|V|+|E|)$ time. Since the \emph{for}-loop in Line \ref{alg:for-v-U} runs runs at
	most four times, the overall time-complexity of Alg.\ \ref{alg:pseudoCograph} is in
	$O(|V|+|E|)$. By Lemma \ref{lem:construct-network}, a labeled level-1 network that
	explains $G$ can be constructed in $O(|V|+|E|)$ provided that $G$ is a
	pseudo-cograph.
\end{proof}

\begin{algorithm}[tb] 
\small 
  \caption{\texttt{Construction of Level-1 network that explains Pseudo-Cograph}}
\label{alg:construct-N}
\begin{algorithmic}[1]
  \Require  Pseudo-cograph $G$ given as $(v,G_1,G_2)$-pseudo-cograph in case $|V(G)|\geq 4$.
  \Ensure 	Labeled level-1 network $(N,t)$ that explains $G$
	 \If{$G$ is a cograph} \Return cotree $(T,t)$ \label{alg:cotree}
	 \Else \Comment{ $G$ is a  $(v,G_1,G_2)$-pseudo-cograph}
  	\State Compute discriminating cotree $(T_1,t_1)$ and $(T_2,t_2)$ for $G_1$ and $G_2$, respectively \label{alg:cotrees-1}
		\State add a new vertex $\eta_i$ along the edge $(\textrm{parent}_{T_i}(v),v)$ in $T_i$, $i\in\{1,2\}$   \label{alg:construct1}
		\State \multiline{ 																							%					
									construct $(N,t)$ by first joining $T_1$ and $T_2$ under a common 												\label{alg:construct2}
									root $\rho_N$ which is adjacent to $\rho_{T_1}$ and $\rho_{T_2}$ \\ 
									and identify vertices $\eta_1$ and $\eta_2$ in both trees and
									remove one copy of the leaf $v$ and its incident edge
									}%
		\State \Return  $(N,t)$			 \label{alg:return-N}
\EndIf	
	\end{algorithmic}
\end{algorithm}

\begin{algorithm}[tb]
%\algsetup{linenosize=\tiny}
\small %\small, \footnotesize, \scriptsize, or \tiny
  \caption{\texttt{Pseudo-Cograph Recognition}} 
\label{alg:pseudoCograph}
\begin{algorithmic}[1]
  \Require Graph $G=(V,E)$
  \Ensure \multiline{% 
  					returns \texttt{true} if $G$ is a pseudo-cograph with $|V(G)|\leq 2$; \\
  					returns $(v,G_1,G_2)$ if $G$ is a $(v,G_1,G_2)$-pseudo-cograph with $|V(G)|\geq 3$;\\ 
  					returns \texttt{false} in all other cases
  					}
  	\If{$|V(G)|\leq 2$}		\Return \texttt{true} \EndIf \label{alg:true}		
  	\If{$G$ is a cograph}  \label{alg:if-cograph-1}
				\State Let $\{H_1,\dots, H_k\}$ be the set of connected components of the disconnected graph in $\{G-v,\overline{G-v}\}$ \label{alg:CC-coG}	  
				\State $V_1\gets \{v\}\cup V(H_1)$
				\State $V_2\gets \{v\}\cup\left(\bigcup_{i=2}^k V(H_i)\right)$
				\State \Return $(v,G[V_1],G[V_2])$  \label{alg:vG1G2-CoG}
	  \EndIf
	  	\State $U\gets V(P_4)$ for some induced  $P_4$  in $G$  \label{alg:defU}
		\For{all vertices $v\in U$}  \label{alg:for-v-U}
					\If{ one of $G-v$ or $\overline{G-v}$ is disconnected} \label{alg:if-disc}
							\State Let $\mathfrak{C}$ 
										 be the set of connected components of the disconnected graph in $\{G-v,\overline{G-v}\}$ \label{alg:CC}
							\State Let $H'\in \mathfrak{C}$ be the connected that contains two vertices of the $P_4$. 
							\State $G_1\gets G[V(H')\cup \{v\}]$ and $G_2\gets G[V(G)\setminus V(H')]$ \label{alg:induced}
              			\If{$G_1$ and $G_2$ are cographs} \label{alg:If-coG1}
											\State \Return $(v,G_1,G_2)$ \label{alg:retG1G2-2}
										\EndIf 														
							\EndIf \label{alg:end-if-disc}
			\EndFor	 \label{alg:end-for-i}
	\State \Return \texttt{false} \label{alg:return-noPsC}
\end{algorithmic}
\end{algorithm}

\begin{algorithm}[tb] 
%\algsetup{linenosize=\tiny}
\small %\small, \footnotesize, \scriptsize, or \tiny
  \caption{\texttt{Polar-Cat Recognition and Construction of Explaining Level-1 network }}
\label{alg:elementary}
\begin{algorithmic}[1]
  \Require Graph $G=(V,E)$
  \Ensure \multiline{% 
  					Strong quasi-discriminating elementary network $(N,t)$ that explains $G$, if $G$ is a Polar-Cat and, \\ otherwise, \texttt{false}
  					}%
	 \If{$G$ is a cograph} \Return \texttt{false} \label{alg:CheckCograph}
  \EndIf
  	\State Find induced path $P$ on four vertices in $G$ \label{alg:FindP4}
		\For{all vertices $v\in V(P)$}  \label{alg:for-all-v-P4}
		  \If{one of $G-v$ or $\overline{G-v}$ is disconnected with set of connected components $\mathfrak{C}$}\label{G-v-CC}
				\If{$|\mathfrak{C}| \neq 2$} \label{G-v-CC-4}
					\State \textbf{go-to} Line \ref{alg:for-all-v-P4}
			 \EndIf
			\State Let $\mathfrak C = \{H_1, H_2\}$ \label{alg:for-all-part}
				\State	$G_1\gets G[V(H_1) \cup v]$ \label{G1}
				\State $G_2\gets G[V(H_2) \cup v]$ \label{G2}
				\If{$G_1$ and $G_2$ are both connected (resp., disconnected) cographs} \label{alg:def1}
					\If{$G-v$ is the disjoint union (resp., join) of $G_1-v$ and $G_2-v$} \label{alg:ddd}
						\State Compute discriminating cotree $(T_1,t_1)$ and $(T_2,t_2)$ for $G_1$ and $G_2$, respectively \label{alg:cotrees}
						\If{both $T_1$ and $T_2$ are caterpillars in which $v$ is part of a cherry} \label{alg:caterpillar}
								\State \Return $(N,t)$ as computed with  Alg.\ \ref{alg:construct-N} with input $(v,G_1,G_2)$			 \label{alg:return2}
					  \EndIf
					\EndIf
				\EndIf			
	 		%\EndFor
			\EndIf				
		\EndFor
	\State \Return \texttt{false} \label{alg:return3}
\end{algorithmic}
\end{algorithm}

\begin{theorem}\label{thm:polar-recognition}
	Polar-cats $G=(V,E)$ can be recognized in $O(|V|+|E|)$ time. In the affirmative
	case, one can determine a vertex $v$ and subgraphs $G_1,G_2\subset G$ such that $G$
	is a $(v,G_1,G_2)$-polar-cat and an (elementary) strong quasi-discriminating
	level-1 network that explains $G$ in $O(|V|+|E|)$ time. 
\end{theorem}
\begin{proof}
	To show that polar-cats can be recognized in $O(|V|+|E|)$ time we use Alg.\
	\ref{alg:elementary}. We show first that Alg.\ \ref{alg:elementary} correctly
	verifies that $G$ is a polar-cat or not. Let $G=(V,E)$ be an arbitrary graph. 

	In Algorithm \ref{alg:elementary}, we first check if $G$ is a cograph (Line
	\ref{alg:CheckCograph}). If $G$ is a cograph, then Cor.\ \ref{cor:CographNotPC}
	implies that $G$ is not a polar-cat and \emph{false} is returned. We then proceed
	to find an induced $P_4$ in Line \ref{alg:FindP4}. Since $G$ is not a cograph, at
	least one such path $P$ must exists. By Lemma \ref{lem:v-in-allP4}, possible
	candidates $v$ for $G$ being a $(v,G_1,G_2)$-polar-cat must be located on $P$.
	Hence, it suffices to check for all $v\in V(P)$ whether $G$ is a
	$(v,G_1,G_2)$-polar-cat (Line \ref{alg:for-all-v-P4} to \ref{alg:caterpillar})
	until we found one, in which case $(N,t)$ is returned (Line \ref{alg:return2}) or
	none of these choices yields polar-cat, in which case \emph{false} is returned
	(Line \ref{alg:return3}). We check in Line \ref{G-v-CC} whether $G-v$ or
	$\overline{G-v}$ is disconnected. If $G-v$ and $\overline{G-v}$ are connected, then
	$G-v$ cannot be a cograph and thus, Obs.\ \ref{obs:G-v-Cograph} implies that $G$
	cannot be a pseudo-cograph. Let $\mathfrak{C}$ be the set of connected components
	of the disconnected graph $G-v$ or $\overline{G-v}$. If $|\mathfrak{C}| \neq 2$
	(Line \ref{G-v-CC-4}), then Lemma \ref{lem:PolarCat-complement-CC} implies that $G$
	cannot be a $(v,G_1,G_2)$-polar-cat and we go back to Line \ref{alg:for-all-v-P4}.
	If $G$ is a pseudo-cograph, then Lemma \ref{lem:star-center-gamma} implies that
	there must be at least one partition $\{\mathfrak{C}_1,\mathfrak{C}_2\}$ of
	$\mathfrak{C}$ such that $G$ is a $(v,G_1,G_2)$-pseudo cograph with $G_1\gets
	G[\cup_{H\in \mathfrak{C}_1}V(H)\cup \{v\}]$ $G_2\gets G[\cup_{H\in
	\mathfrak{C}_2}V(H)\cup \{v\}]$. Since $|\mathfrak{C}|=2$, there is only one such
	partition, that is, $\{\{H_1\},\{H_2\}\}$. Finally, we check in Line \ref{alg:def1}
	to \ref{alg:caterpillar} simply whether the definition of polar-cats is satisfied
	for $(v,G_1,G_2)$. If this is the case, then we compute in a level-1 network
	$(N,t)$ with Alg.\ \ref{alg:construct-N} with input $(v,G_1,G_2)$. Note, $(N,t)$ is
	precisely the strong elementary quasi-discriminating network as constructed in the
	proof of the \emph{only-if}-direction of Thm.\ \ref{thm:StrongElementary}. In summary,
	Alg.\ \ref{alg:elementary} is correct.

	We proceed now with determining the runtime of Alg.\ \ref{alg:elementary}. Checking
	whether $G=(V,E)$ is a cograph or not can be done $O(|V|+|E|)$ time with the
	algorithm of Corneil et al.\ \cite{Corneil:85}. As argued in the proof of Thm.\
	\ref{thm:PsC-recognition}, finding an induced $P_4$ containing $w$ can be achieved
	in $O(|V|)$ time \cite{capelle1994cograph}. Checking if $G-v$ or $\overline{G-v}$
	is disconnected and computing the respective set of connected components
	$\mathfrak{C}$ can be done in $O(|V|+|E|)$ time. Within the same time-complexity we
	can compute $G_1$ and $G_2$ in Line \ref{G1} and \ref{G2} and check if both $G_1$
	and $G_2$ are connected (resp., disconnected) cographs. While doing the latter
	task, we also compute the discriminating cotree $(T_1,t_1)$ and $(T_2,t_2)$ for
	$G_1$ and $G_2$, respectively in Line \ref{alg:cotrees}. We then verify in Line
	\ref{alg:caterpillar} in $O(L(T_1)) = O(|V|)$ $O(L(T_2)) = O(|V|)$ time, if $T_1$
	and $T_2$ are caterpillars in which $v$ is part of a cherry. If this is the case,
	we compute with Alg.\ \ref{alg:construct-N} the network $(N,t)$ in $O(|V|+|E|)$
	time. Hence, the overall runtime of Alg.\ \ref{alg:elementary} is in $O(|V|+|E|)$. 
\end{proof}

\begin{algorithm}[tbp] 
%\algsetup{linenosize=\tiny}
\small %\small, \footnotesize, \scriptsize, or \tiny
  \caption{\texttt{Check if $G\in \protect\PrimeCat$ and construct Level-1 network that explains $G$}}
\label{alg:general}
\begin{algorithmic}[1]
  \Require Graph $G=(V,E)$
  \Ensure \multiline{% 
  					Labeled level-1 network $(N,t)$ that explains $G$, if $G\in \PrimeCat$
  					and, \\ otherwise, statement ``\emph{$G$ cannot be explained by a level-1
  					network}''
  					}%
	\State  Compute MDT $(\MDT_G,\tau_G)$ of $G$ \label{alg:MDT}
	\State $(N,t)\gets (\MDT_G,\tau_G)$ \label{alg:init}
	\State $\mathcal{P}\gets$ set of prime vertices in $(\MDT_G,\tau_G)$
	\For{all $v\in \mathcal{P}$}
		\State $M\gets$ prime modules of $G$ with $M=L(\MDT_G(v))$. \label{alg:prime-M}
		\If{$G[M]/\Mmax(G[M])$ is not a polar-cat} \label{alg:if-quotient}
			\State \Return  \emph{$G$ cannot be explained by a level-1 network}
		\Else	
	   \State\multiline{% 
	   			 $(N,t) \gets$ level-1 network obtained from $(N,t)$ by
			    replacing $v$ by strong quasi-discriminating elementary network \\
			    \hspace{1cm} $(N_v,t_v)$ 
				 that explains $G[M]/\Mmax(G[M])$ according to Def.\ \ref{def:pvr}
			    }
		\EndIf	    
	\EndFor	    
	\State \Return $(N,t)$
\end{algorithmic}
\end{algorithm}

\begin{theorem}\label{thm:AlgGeneral}
	It can be verified in $O(|V|+|E|)$ time if a given graph $G=(V,E)$ can
	be explained by a labeled level-1 network and, in the affirmative case, 
	a strong labeled level-1 network $(N,t)$  that explains $G$  can
	be constructed within the same time complexity. 
\end{theorem}
\begin{proof}
We use Alg.\ \ref{alg:general}. This algorithm checks for all prime modules $M$ of
$G$ if $G[M]/\Mmax(G[M])$ is a polar-cat. If the latter is violated for some prime
module $M$ then $G\notin \PrimeCat$ and Thm.\ \ref{thm:CharprimeCat} implies that $G$
cannot be explained by a labeled level-1 network. Moreover, all prime modules $M$ of
$G$ correspond to some vertex $v$ in the MDT $(\MDT_G,\tau_G)$ of $G$. All such
vertices $v$ are replaced by the respective elementary strong networks $(N_v,t_v)$
that explains $G[M]/\Mmax(G[M])$. According to Def.\ \ref{def:pvr}, we obtain the pvr
graph $(N, t)$ and Prop.\ \ref{prop:well-defPVR} implies that $(N, t)$ is a strong
labeled level-1 network that explains $G$. 

   To obtain a bound on the running time, we first note that the MDT
$(\MDT_G,\tau_G)$ for $G=(V,E)$ and thus, the initial graph $(N,t)$ in Line
\ref{alg:MDT} and \ref{alg:init} can be computed in $O(|V|+ |E|)$ time \cite{HP:10}.
Within the same time complexity we obtain the set $\mathcal{P}$ of prime vertices in
$(\MDT_G,\tau_G)$. Let $n_v$ denote the children of a vertex $v$ in $\MDT_G$. We then
compute for every prime vertex $v$ and its corresponding module $M$ (cf.\ Line
\ref{alg:prime-M}) the quotient $G[M]/\Mmax(G[M])$ in Line \ref{alg:if-quotient}.
As shown in \cite[Lemma 2]{DUCOFFE2021201}, \emph{all} quotients
 $G[M]/\Mmax(G[M])$ can be computed in total $O (|V| + |E|)$ time 
 whenever $(\MDT_G,\tau_G)$ is given.
The total effort is therefore in $O(|V|+|E|)$.
\end{proof}

\section{Outlook and Summary}
\label{sec:outlook}

We characterized the class $\PrimeCat$ of graphs that can be explained by labeled
level-1 networks $(N,t)$. Subclasses of $\PrimeCat$ are the classes of cographs,
pseudo-cograph and polar-cats $\PolarCat$. It particular, a graph $G$ is contained in
$\PrimeCat$ if $H\in \PolarCat$ for each of its prime subgraphs $H$. We provided
several characterizations for $\PrimeCat, \PolarCat$ and $\textsc{PseudoCograph}$
and designed linear-time algorithms for their recognition and the construction of
level-1 networks to explain them. These new graph classes open many further
directions for future research. 
		
Many types of graphs are defined by the kind of subgraphs they are not permitted to
contain, e.g.\ forests, cographs \cite{Corneil:81} or $P_4$-sparse graphs
\cite{Hoang:85,Jamison:89,Jamison:92}. The property of being in one of the classes
$\PrimeCat$ and $\textsc{PseudoCograph}$ is hereditary. It is therefore natural to
ask if these classes can be solely characterized in terms of a (finite) collection of
forbidden induced subgraphs, see Fig.\ \ref{fig:forb-primitive} for some examples.
Lemma \ref{lem:star-center-gamma} already provides strong structural hints for such
forbidden subgraphs. Moreover, are there characterizations of graphs $G\in \PrimeCat$
in terms of their diameter and possibly additional properties?
	
In many applications, constructions of labeled trees that explain a given graph $G$
are based on collection triples and clusters that are encoded by $G$ and that must be
``displayed'' by any tree that also explains $G$
\cite{BD98,HSW:16,HHH+13,Hellmuth:19}. Hence, one may ask, if there is a similar way
to obtain the information of clusters (cf.\ \cite[Sec.\
8.5]{huson_rupp_scornavacca_2010}, \cite{HSS:22cluster} or \cite{gambette2017challenge}) or trinets
\cite{huber2010practical,huber2017reconstructing,Huber2012EncodingAC} directly from
$G$ and to construct a level-1 network that explains $G$ based on this information
without the explicit construction of the quotients $G[M]/\Mmax(G[M])$ of prime
modules $M$ of $G$.
		
Many NP-hard problems become easy for graphs for which the particular solution can be
constructed efficiently on the quotients $G[M]/\Mmax(G[M])$ of prime modules $M$
\cite{VHSH:20,COURCELLE200077,bonomo2014minimum,BRANDSTADT2003521,Giakoumakis1997OnPG}.
Since $G[M]/\Mmax(G[M])$ must be a polar-cat, that is, a graph with strong structural
constraints, it is of interest to understand in more detail if the class $\PrimeCat$
admits also efficient solutions for general intractable problems. Moreover, what can
be said about invariants of pseudo-cographs or graphs in $\PrimeCat$ as the path-width
and tree-width \cite{BM:93,ADLER2019126}, the chromatic number
\cite{Corneil:81,VHSH:20}, the clique-, stability- or scattering-number
\cite{Corneil:81,Giakoumakis1997OnPG}, among many others.
				
In our contribution, we considered graphs that can be explained by level-1 networks.
In applications, such a graph $G$ reflects pairwise relationship between objects
that are represented as vertices in $G$. In many cases, one is interested in a single
tree or level-1 network that explains not only one but many different such
relationships at once. This leads to the more general concept of 2-structures (or
equivalently symbolic ultrametrics or multi-edge colored graphs)
\cite{ER1:90,ER1:90,BD98,HSW:16} or 3-way maps
\cite{huber2019three,grunewald2018reconstructing}. The latter type of structures are
all defined in terms of (rooted or unrooted) trees and we suppose that our current
results can be used to generalize these structures to (rooted or unrooted) level-1
networks. Moreover, it is of interest to understand to which extent our results can
be generalized to  level-$k$ networks with $k>1$ or other types of
phylogenetic networks \cite{huson_rupp_scornavacca_2010}.

\section*{Acknowledgments}
We thank the anonymous referees for their valuable comments and David Schaller
 for stimulating discussions about level-1 networks. Moreover,
we thank the anonymous designer of the symbol $\PolarCat$ which we downloaded
from \cite{CatData}.

\bibliographystyle{plain}      % mathematics and physical sciences
\bibliography{biblio}   % name your BibTeX data base

\end{document}